\documentclass[final,onefignum,onetabnum]{siamart190516}

\usepackage{lipsum}
\usepackage{amsfonts}
\usepackage{graphicx}
\usepackage{epstopdf}
\usepackage{algorithmic}
\usepackage{cite}
\usepackage{graphicx,epstopdf,diagbox,multirow} 
\usepackage[caption=false]{subfig}
\usepackage{rotating}

\makeatletter
\def\hlinewd#1{%
  \noalign{\ifnum0=`}\fi\hrule \@height #1 \futurelet
   \reserved@a\@xhline}

  \newcounter{para}
\newcommand\mypara[1]{\par\refstepcounter{para}\textit{\thepara)\space#1\space}}

\ifpdf
  \DeclareGraphicsExtensions{.eps,.pdf,.png,.jpg}
\else
  \DeclareGraphicsExtensions{.eps}
\fi

\newsiamremark{remark}{Remark}
\newsiamremark{hypothesis}{Hypothesis}
\crefname{hypothesis}{Hypothesis}{Hypotheses}
\newsiamthm{claim}{Claim}

\headers{Elementary Symmetric Polynomials and Inverse of Vandermonde}{M. S. Hosseini, A. Chen, and K. N. Plataniotis}

\title{On the Closed Form Expression of Elementary Symmetric Polynomials and the Inverse of Vandermonde Matrix}

\author{Mahdi S. Hosseini\thanks{The Edward S. Rogers Sr. Department of Electrical \& Computer Engineering , University of Toronto, ON M5S 3G4, Ontario, Canada (\email{mahdi.hosseini@mail.utoronto.ca}, \email{lixiang.chen@mail.utoronto.ca}, \email{kostas@ece.utoronto.ca}).}
\and Alfred Chen\footnotemark[1]
\and Konstantinos N. Plataniotis\footnotemark[1]}

\usepackage{amsopn}

\ifpdf
\hypersetup{
  pdftitle={On the Closed Form Expression of Elementary Symmetric Polynomials and the Inverse of Vandermonde Matrix},
  pdfauthor={Mahdi S. Hosseini, Alfred Chen, and Konstantinos N. Plataniotis}
}
\fi

\begin{document}

\maketitle

\begin{abstract}
Inverse Vandermonde matrix calculation is a long-standing problem to solve nonsingular linear system $Vc=b$ where the rows of a square matrix $V$ are constructed by progression of the power polynomials. It has many applications in scientific computing including interpolation, super-resolution, and construction of special matrices applied in cryptography. Despite its numerous applications, the matrix is highly ill-conditioned where specialized treatments are considered for approximation such as conversion to Cauchy matrix, spectral decomposition, and algorithmic tailoring of the numerical solutions. In this paper, we propose a generalized algorithm that takes arbitrary pairwise (non-repetitive) sample nodes for solving inverse Vandermonde matrix. This is done in two steps: first, a highly balanced recursive algorithm is introduced with $\mathcal{O}(N)$ complexity to solve the combinatorics summation of the elementary symmetric polynomials; and second, a closed-form solution is tailored for inverse Vandermonde where the matrix' elements utilize this recursive summation for the inverse calculations. The numerical stability and accuracy of the proposed inverse method is analyzed through the spectral decomposition of the Frobenius companion matrix that associates with the corresponding Vandermonde matrix. The results show significant improvement over the state-of-the-art solutions using specific nodes such as $N$th roots of unity defined on the complex plane. A basic application in one dimensional interpolation problem is considered to demonstrate the utility of the proposed method for super-resolved signals.
\end{abstract}

\begin{keywords}
Elementary symmetric polynomials, inverse Vandermonde matrix, ill-conditioned linear system, $N$th roots of unity, generalized sampling nodes
\end{keywords}

\begin{AMS}
05E05, 11C08, 11C20, 15B05, 15A09, 65F05
\end{AMS}

\section{Introduction}
The Vandermonde inverse is extensively researched, continued onto this decade \cite{werts_1965,nasa_60,EisinbergFedele2006,Mikkawy_2003,taher_2016}. It appears in many applications such as in matrix pencil method \cite{hua1990matrix, sarkar1995using} where the idea is to approximate signals via a linear decomposition framework whose unknown parameters are the solution to the generalized eigenvalue problem. The super-resolution problem is a variant of this approximation approach where a summation of complex exponential basis functions can be represented by a Vandermonde system of equations \cite{superResolution}. The matrix can be used as a generator matrix in the transformation function of Reed-Solomon codes \cite{solomonCode} as one of the many possible applications. In another venue, the Vandermonde matrix is used to construct distance separable (MDS) matrices by multiplication of two block Vandermonde (variants of partitioned matrices) that are utilized in cryptography such as designing block ciphers and hash functions \cite{sajadieh2012construction, lacan2004systematic, li2018direct, yaici2019particular}. However, the main driver for the appearance of Vandermonde matrices roots back to the applications in polynomial approximation and interpolation \cite{nasa_60}.

Without loss of generality, consider a continuously differentiable function $f(\cdot)$ and a set of real pairwise distinct sampling points $\{v_1,v_2,\ldots,v_N\}$ with the goal of solving for the coefficients of the polynomial $\{c_0,c_1,\ldots,c_{N-1}\}$ in  
\begin{align}
c_0+c_1v_i+c_2v_i^2+\ldots+c_Nv_i^{N-1} = f(v_i),~\text{for}~i\in\{1,2,\cdots,N\} \nonumber
\end{align}
Rewriting in the well known form 
\begin{align}
V^T c = f\nonumber
\end{align}
with the general Vandermonde matrix $V$, coefficient vector $c$, and observation vector $f$ defined as
\begin{align}
V = \begin{bmatrix}
1 & 1 & \cdots & 1\\
v_1 & v_2 & \cdots & v_N\\
\vdots & \vdots &\vdots & \vdots\\
v_1^{N-1} & v_2^{N-1} & \cdots & v_N^{N-1}
\end{bmatrix},~~~
c = 
\begin{bmatrix}
c_0\\c_1\\\vdots\\c_{N-1}
\end{bmatrix}, ~\text{and}~~
f =
\begin{bmatrix}
f(v_1)\\f(v_2)\\\vdots\\f(v_{N})
\end{bmatrix}.\label{eq:variables}
\end{align}
In this definition, the Vandermonde matrix $V$ is defined on a finite field of real pairwise distinct nodes $\{v_1,v_2,\cdots,v_N\}$. Most Vandermonde applications fall into this type of problem. For example, in Reed-Solomon coding the message of length $k$ can be decoded by solving the coefficients of a $k^{th}$ polynomial function at $k+1$ different points--a variant of polynomial approximation problem. The transpose notation $V^T$ is another representation of the Vandermonde matrix \cite{interp_text,EisinbergFedele2006,stirlingInv}. However we chose this Vandermonde form based on the inverse decomposition from \cite{chen1981new,man2014inversion} that is used and analyzed in \cref{sec:main}. For simplicity and flow, the inverse and the Vandermonde inverse are interchangeable in context unless explicitly stated.

The Vandermonde inverse with pairwise distinct nodes always exists because the columns are linearly independent. However, the inverse is generally ill conditioned \cite{ill_posed1,ill_posed2,ill_posed3, bazan2000conditioning}. The condition number of these matrices grow exponentially in $N$ i.e. $\kappa\left(V\right)>2^{N-2}/\sqrt{N}$ \cite{gautschi1975optimally, gautschi1990stable} unless the sample elements are equally spaced on the unit circle on the origin (these samples are roots of unity on the complex plane) \cite{ill_posed3}. Large condition numbers have been observed to subject a digital inversion method (such as the matrix inversion function in MATLAB) to severe round-off errors caused by large floating point arithmetic. The existence of the Vandermonde inverse and ill-conditioning enabled a significant amount of research in finding accurate and fast solutions to the inverse \cite{macon_1958,werts_1965,Mikkawy_2003,Traub_1966,taher_2016,csaki_1975,eisin_1981,EisinbergFedele2006,nasa_60,bjorck_1970,eisinberg_1998,stirlingInv,man2014inversion,pantelous_2013,kaufman_1969,klinger_1967,neagoe_1996,youness_2016}. The stability condition of the Vandermonde matrices with different sampling nodes is also of broad interest in super-resolution problem to analyze the recovery conditions in different sampling scenarios that is mainly contaminated with noise \cite{li2017stable, kunis2018condition, batenkov2018conditioning, batenkov2019super, batenkov2019rethinking}. The Vandermonde matrix can be also converted into Cauchy matrix representation via discrete Fourier transform (DFT) to overcome the issue of ill-posedness \cite{demmel2000accurate, drmavc2018data}. 

A function commonly used in the majority of competitive Vandermonde inversion algorithms is the \textit{elementary symmetric polynomials (ESP)}, which is a specific sum of products without permutation of repetitions and defined by
\begin{align}
 \sigma_{N,j} = \sum\limits_{r_1\neq r_2\neq\cdots\neq r_j}{v_{r_1}v_{r_2}\cdots v_{r_j}},\label{eq:elem_func}
 \end{align}
where, $\sigma_{N,j}$ is called the j'th ESP over the roots/samples $\{v_1,v_2,\ldots,v_N\}$ \cite{dick_1945}. This notation is used for the rest of the paper. Macon et.al. was the first to use this function in a Vandermonde inversion solution in \cite{macon_1958}. It takes a Lagrange interpolation approach and uses the ESP to solve the inverse problem. Some inversion approaches such as LU decomposition do not encounter such functions. The only known exception is in \cite{Yang_2009} where Yang et.al. uses it as a tertiary variable in a LU decomposition solution. Solving these polynomials directly is inefficient. To reduce the time, there has been some research in finding algorithms to solve them \cite{Traub_1966,Yang_2009,Mikkawy_2003}. The majority of Vandermonde inversion techniques either use a recursive solution introduced in \cite{Traub_1966}, or do not explicitly state a solution. For example, there are a handful of algorithms that use the ESP to form an explicit algorithmic solution but are presented in a theoretical viewpoint \cite{macon_1958,klinger_1967,neagoe_1996,taher_2016}. Therefore, the solutions and performance analysis is left to the reader. In contrast, Eisinberg et. al. \cite{eisin_1981} uses the ESP as a set of variables to define an explicit formulation for the Vandermonde inverse and later generalized in \cite{EisinbergFedele2006}. Both these methods uses the solution first introduced in \cite{Traub_1966}.

\subsection{Related Works}
In this section, we briefly discuss existing solutions for both the ESP and the Vandermonde inverse. 
\subsubsection{Elementary Symmetric Polynomials}\label{sec_ESP_introduction}
To the best of our knowledge, there are three closed form expressions to the ESP \cite{Traub_1966,Yang_2009,Mikkawy_2003}.

The first expression was introduced by Traub in \cite{Traub_1966} and is the most commonly used in the literature. The formulation of this algorithm is defined by
\begin{equation}
 \sigma_{n,j} =  \sigma_{n-1,j} +  v_n\sigma_{n-1,j-1}\label{eq:traub}
 \end{equation}
 where
 \begin{align*}
n=2,3,\ldots,N; \qquad j = 1,2,\ldots,n \\
 \sigma_{n,0} = 1,\enskip \forall n;\quad  \sigma_{1,1} =v_1; \quad  \sigma_{n,j} = 0,\enskip j>n.
 \end{align*}
It is a simple fully recursive algorithm that creates a $N\times N$ matrix where the $(i,j)$'th entry is the $j$'th ESP over the first $i$ nodes from the set. Since $\sigma_{i,j} = 0,\enskip j>i$, the output matrix is lower triangular.

The second algorithmic expression was proposed by Mikkawy in \cite{Mikkawy_2003}. The algorithm is designed for a Vandermonde inverse solution where the one of the elements is removed in the sample set. Building on the notation of \cref{eq:elem_func}, the removal of an element in the sample set for the ESP will be defined by 
 \begin{equation}\label{eq:removed}
 \sigma_{N,j}^{(1)} = \sum\limits_{r_1\neq r_2\neq\cdots\neq r_j\neq1}{v_{r_1}v_{r_2}\cdots v_{r_j}}.
 \end{equation}
The removal of an element in the sample set is used in Vandermonde inverse solutions such as \cite{Traub_1966,Mikkawy_2003} and the proposed inverse that will be introduced in \cref{sec:main}. The ESP solution in \cite{Mikkawy_2003} is defined by
 \begin{equation}
\sigma_{n,n}^{(1)} =  \sigma_{n-1,n-1}^{(1)}v_n;\quad
\sigma_{n,j}^{(1)} =  \sigma_{n-1,j-1}^{(1)}v_n + \sigma_{n-1,j}^{(1)}, \label{eq:mikkawy}
 \end{equation}
where
 \begin{align*}
 \sigma_{1,1}^{(1)} =  1;  \quad n = 2,3,\ldots,N; \quad j = 2,3,\ldots,n-1.
 \end{align*}
The expression when the $i$'th element of the sample set is removed is
 \begin{align*}
 \sigma_{n,j}^{(i)} = \sigma_{n,j}^{(1)} \Big\vert_{v_i \rightarrow v_1} \text{for } \quad n = 2,3,\ldots,N; \quad j = 2,3,\ldots,n-1.
 \end{align*}
It implies that for a specific $i$ and $j$, $\sigma_{n,j}^{(i)}$ is obtained by replacing every $v_k$ with $v_1$ in $\sigma_{n,j}^{(1)}$. Since algorithmic expression removes one of the elements when calculating, it will create a $(N-1)\times (N-1)$ lower triangular matrix.

Finally, the third solution to the ESP was introduced by Yang et.al. in \cite{Yang_2009}. It is defined by
\begin{equation}
 \sigma_{n,j} =  \sum\limits_{k=0}^{j}({\prod_{i=1}^{k}}{v_{n-i+1}}) \sigma_{n-k,j-k}\label{eq:yang}
 \end{equation}
 with
 \begin{align*}
 \sigma_{n,0} =  1; \quad \prod_{i=1}^{0}{v_{n-i+1}} = 1.
 \end{align*}
This expression takes a more direct approach to calculate ESP. The partially recursive solution is a re-written form of the original definition in \cref{eq:elem_func}. The algorithm \cref{eq:yang} produces a $N\times N$ lower triangular matrix with same entries as expressed in \cref{eq:traub}. 

The three solutions to the ESP provides a matrix containing the combinations of $n,j$ in $\sigma_{n,j}$ where the rows relate to $n$ (number of samples) and the columns relate to $j$ ($j$'th ESP). Although \cite{Traub_1966,Yang_2009,Mikkawy_2003} does not have analysis on the performance on their respective solutions, it is easy to see why \cref{eq:traub} is the most used by inspection. The simple implementation is favorable while the fully recursive structure with basic computations (addition and multiplication) should provide faster computation speeds. 

\subsubsection{Vandermonde Inverse}
Vandermonde inverse solutions can be further divided into two categories. Focused on the approaches taken, the solutions are either based on polynomials or matrix relations.
 
\mypara{Polynomials} 
In the literature, there are a number of ways that solves the Vandermonde inverse using the polynomial approach. Lagrange interpolation polynomials over the unique nodes of the Vandermonde matrix are used in a few approaches to develop a solution for the inverse of a Vandermonde \cite{macon_1958,werts_1965,Mikkawy_2003}. Although the initial approach in the solutions are the same, the methods they use to solve the inverse is unique. For example, Macon et.al. uses an explicit formulation for the derivatives of a $n+1$ ordered fitted polynomial while Mikkawy uses partial fractions to refactor the Lagrange basis polynomials and invert the Vandermonde matrix in \cite{macon_1958, Mikkawy_2003}, respectively. Traub uses the orthonormality relation between a monic-polynomial and its associated polynomials to derive a closed form inverse solution in \cite{Traub_1966}.  Taher et.al. uses the Binet formula for a weighted r-generalized Fibonacci sequence to solve the Vandermonde linear system in \cite{taher_2016}. The solutions to the interpolated polynomial coefficients in \cite{taher_2016} are then used to develop the entries to the Vandermonde inverse. Csaki in \cite{csaki_1975} was one of the first to determine the elements of the inverse through the Hermite-Kronecker polynomials. Later, Eisinberg et.al. in \cite{eisin_1981} provides an equivalent solution to \cite{csaki_1975} with a highly recursive structure that improves on computation. This algorithm was further generalized in \cite{EisinbergFedele2006} by providing a fully recursive algorithm to the solution. This solution is more flexible and it allows special algorithms to be obtained for specific nodes such as equidistant and Chebyshev nodes.

\mypara{Matrix Relations} LU decomposition is one of the earliest matrix approaches used for Vandermonde inversion. The inverse is a product of two matrices ($U^{-1}L^{-1}$) where the elements of matrices $U^{-1}$ and $L^{-1}$ are derived in various ways. Readers are referred into \cite{nasa_60,bjorck_1970,eisinberg_1998} for some approaches in solving the factorized matrices. There are other matrix approaches to find the inverse. Similar to \cite{eisinberg_1998}, Bender et.al. develops a recursive solution for Vandermonde matrices with equidistant integer nodes  \cite{stirlingInv}. It is a linear recursion relation in the form of a 2D Pascal pyramid by inspecting the differences between consecutive matrix orders. Discussed in detail in \cref{sec:main}, Man uses the cover-up technique in partial fraction decomposition to formulate the inverse as the product of two matrices \cite{man2014inversion}. Pantelous et.al. factorizes and calculates the inverse by a set of left-multiplied and a set of right-multiplied matrices \cite{pantelous_2013}. Kaufman in \cite{kaufman_1969} formulates a recursive formula based on Hermite interpolating polynomials in order to determine the rows of the inverse matrix. \cite{klinger_1967} and \cite{neagoe_1996} both develops a relationship between consecutive matrix orders using the determinants of the Vandermonde. Klinger of \cite{klinger_1967} forms a relation between the determinants of a Vandermonde and a $n+1$ column (powers) adjoined, j'th column removed Vandermonde matrix. Neagoe, on the other hand, forms a relation between the determinants of a Vandermonde and a Vandermonde with the j'th column removed \cite{neagoe_1996}. Both uses ESP in their formulation. In \cite{youness_2016}, Ghassabeh factorizes the inverse into three matrices and uses them in a recursive solution. The algorithm iteratively calculates the inverse of the Vandermonde. The order of the matrix increases by one for every iteration and stops when the desired order is reached \cite{youness_2016}. This design allows the formulation of the inverse when the nodes are observed sequentially.

\subsection{Shortcomings and Contributions} 
Although the ESP solutions in \cite{Traub_1966,Mikkawy_2003,Yang_2009} are tailored differently, a common disadvantage is a specific recursive recalling behaviour, which produces inaccuracies when using the roots of unity sample set. This is of paramount importance in applications such as in super-resolution \cite{superResolution, li2017stable, kunis2018condition, batenkov2018conditioning, batenkov2019super, batenkov2019rethinking} or special matrix form construction such as in \cite{sajadieh2012construction, lacan2004systematic, li2018direct, yaici2019particular}. Moreover, despite vigorous research efforts in finding Vandermonde inverse solutions, there has yet to be a simple, general, fast and accurate inverse solution. Under certain applications, some solutions are designed to perform accurately under only integer nodes such as in \cite{eisinberg_1998,stirlingInv}. To the best of our knowledge, the solution by Eisinberg et.al. in \cite{EisinbergFedele2006} is the current state-of-the-art inverse solution for general square Vandermonde matrices. It uses recursive formulas from a set of defined functions. Depending on the type of nodes, these recursive functions may be reformulated into closed expressions.

This paper encapsulates solutions for the ESP and inverse Vandermonde that will address some of the shortcomings above. The main contributions are as follows:
\begin{itemize}
\item We proposed a new recursive solution to the ESP that solves the issue of imbalance summation and significantly outperforms other ESPs on certain nodes such as the $N$th roots of unity defined on the complex plane.
\item Utilized by our new ESP method, we develop a novel and compact approach to calculate the inverse of a general Vandermonde matrix that can be defined on any arbitrary pairwise (non-repetitive) nodes.
\item We employ the spectral decomposition of Frobenius companion matrix for indirect evaluation of the inverse Vanderomonde and provide a numerical approach for stability and accuracy analysis.
\item Thorough analysis on numerical experimentation are provided to analyze the utility of our proposed solution on different nodes for 1D interpolation problem and compared to the sate-of-the-art methods. The results suggest that the proposed method can achieve great performances on certain nodes such as $N$th roots of unity.
\end{itemize}

The early work in \cite{hosseini2017finite} proposed a closed-form solution to the finite difference methods, where the solution to the fullband finite difference calculations are expressed by the inverse Vandermonde matrix-vector calculation using equal distance nodes. We found a closed form expression to the ESP using this node design. In this paper, we generalize this design into any arbitrary pairwise distinct nodes $\{v_1,v_2,\cdots,v_N\}$ that can contain real or complex structure and obtain a closed form solution to both ESP and inverse Vandermonde matrix calculation. Furthermore, we emphasize its high stability over approximating polynomials sampled over the complex unit circle.

The remainder of this paper is organized as follows. The closed form expressions for the ESP and the Vandermonde inverse are proposed in \cref{sec:main}. The analysis and discussion on ESP over the unit circle is presented in \cref{sec:elementary}. And the experimental results and discussions are provided in \cref{sec:app}. We conclude the paper in \cref{sec:conclusions}.

\section{Main results}\label{sec:main}
Chen \emph{et al.} \cite{chen1981new} introduced a method to determine the coefficients of high ordered polynomial expansions and to incorporate the inverse of a Vandermonde matrix. Given the following rational function:
\begin{align}
f(x) 
&= \frac{b_1x^{N-1}+b_2x^{N-2}+\cdots+b_N}{x^N + a_1x^{N-1}+\cdots+a_N}\nonumber \\
&= \frac{b_1x^{N-1}+b_2x^{N-2}+\cdots+b_N}{(x-v_1)(x-v_2)\cdots(x-v_N)}\nonumber,
\end{align}
with partial decomposition
\begin{displaymath}
f(x) = \frac{k_1}{x-v_1} + \frac{k_2}{x-v_2} + \cdots + \frac{k_N}{x-v_N},
\end{displaymath}
their formulation for $k_i$ is as follows: 
\begin{displaymath}
\begin{bmatrix}
k_1\\k_2\\ \vdots\\k_N
\end{bmatrix}=
\begin{bmatrix}
1 & 1 & \cdots & 1\\
v_1 & v_2 & \cdots & v_N\\
\vdots & \vdots &\vdots & \vdots\\
v_1^{N-1} & v_2^{N-1} & \cdots & v_N^{N-1}
\end{bmatrix}^{-1}
\begin{bmatrix}
1&0&\cdots&0\\
a_1&1&\cdots&0\\
a_2&a_1&\cdots&0\\
\vdots &\vdots & \vdots &\vdots\\
a_{N-1} & a_{N-2}&\cdots&1
\end{bmatrix}^{-1}
\begin{bmatrix}
b_1\\b_2\\\vdots\\b_N
\end{bmatrix}
\end{displaymath}
and rewritten as
\begin{equation}
k= V^{-1}A^{-1}b\label{eq:1981}
\end{equation}
Where $V$ is a Vandermonde matrix and $A$ is a Stanley matrix. The form of matrix $A$ was first introduced by William D. Stanley in his time-to-frequency domain matrix formulation \cite{stanley1964time}. Y. K. Man \cite{man2014inversion} solves coefficient vector $k$ and uses the formulation of \cite{chen1981new} to create a new solution for the Vandermonde inverse.
\\
\\
In \cite{man2014inversion}, the coefficients $k_i$ can be solved using
\begin{displaymath}
f(x)(x-v_i)\vert_{x=v_i} = \frac{b_1v_i^{N-1}+b_2v_i^{N-2}+\cdots+b_N}{\lambda_i}
\end{displaymath}
where $\lambda_k=\prod\limits_{\substack{j=1\\ j\neq k}}^{N}(v_k-v_j)$. Rewritten in ma  trix form
\begin{displaymath}
k = \begin{bmatrix}
\frac{v_1^{N-1}}{\lambda_1}&
\frac{v_1^{N-2}}{\lambda_1}&
\cdots&
\frac{1}{\lambda_1}\\
\frac{v_2^{N-1}}{\lambda_2}&
\frac{v_2^{N-2}}{\lambda_2}&
\cdots&
\frac{1}{\lambda_2}\\
\vdots & \vdots & \vdots & \vdots \\
\frac{v_1^{N-1}}{\lambda_N}&
\frac{v_1^{N-2}}{\lambda_N}&
\cdots&
\frac{1}{\lambda_N}\\
\end{bmatrix}
\begin{bmatrix}
b_1\\b_2\\ \vdots \\b_N
\end{bmatrix}
\end{displaymath}
\begin{displaymath}
= W b \nonumber
\end{displaymath}
\\
\\Combined with Equation \cref{eq:1981}, \cite{man2014inversion} obtains:
 
\begin{equation}\label{eq:eq1}
V^{-1}=WA
\end{equation}
\\
In this work, $W$ is further refactored
\begin{displaymath}\label{eq:eq2}
W=\left[\begin{array}{c@{\hspace{.5em}}c@{\hspace{.5em}}c@{\hspace{.5em}}c@{\hspace{.5em}}}
\lambda_1 & 0 & \cdots & 0 \\
0 & \lambda_2 & \cdots & 0 \\
\vdots & \vdots & \ddots & \vdots \\
0 & 0 & \cdots & \lambda_N
\end{array}\right]^{-1}
\left[
\begin{array}{l@{\hspace{.5em}}l@{\hspace{.5em}}l@{\hspace{.5em}}l@{\hspace{.5em}}}
v_1^{N-1} & v_1^{N-2} & \cdots & 1 \\
v_2^{N-1} & v_2^{N-2} & \cdots & 1 \\
\vdots & \vdots & \ddots & \vdots \\
v_N^{N-1} & v_N^{N-2} & \cdots & 1
\end{array}
\right]
\end{displaymath}
The pivot elements of the diagonal matrix is computed as
\begin{equation}\label{eq:eq3}
\lambda_k=\prod\limits_{\substack{j=1\\ j\neq k}}^{N}(v_k-v_j)
\end{equation}
The Stanely matrix \cite{stanley1964time} $A$ is defined by 
\begin{equation}\label{eq:eq4}
A=\left[
\begin{array}{l@{\hspace{.45em}}l@{\hspace{.45em}}l@{\hspace{.45em}}l@{\hspace{.45em}}}
1 & 0  & \cdots & 0 \\
a_1 & 1 & \cdots & 0 \\
a_2 & a_1 & \cdots & 0 \\
\vdots & \vdots & \ddots & \vdots \\
a_{N-1} & a_{N-2} & \cdots & 1
\end{array}
\right],~\text{where}~
a_j=(-1)^j\sum\limits_{r_1\neq r_2\neq\cdots\neq r_j}{v_{r_1}v_{r_2}\cdots v_{r_j}} = (-1)^j\sigma_{N,j}
\end{equation}
\\
\\
We define the closed form expression for the ESP below.
\begin{theorem}
\label{thm:hypercube_summation}
Consider the elementary symmetric polynomial $\sigma_{N,n}$
\begin{equation}
 \sigma_{N,n} = \sum\limits_{\chi}{v_{r_1}v_{r_2}\cdots v_{r_n}}
\end{equation}
where, $r_j\in \{1,2,\cdots,N\}$ and $\chi=\{r_1\neq r_2\neq\cdots\neq r_n\}$ is a topological space containing a set of elements with no repetition. The $n$ recursive formulation with initialized parameters $i=0$ and $f_{0}(v) = v$.
\begin{equation}
\left\{
\begin{array}{l}
f_{i}(v_d) = v_d\left[C_{i-1} - (n-i)f_{i-1}(v_d)\right] \\
C_{i} = \sum\limits_{d}f_{i}(v_d)\\
i\leftarrow i+1
\end{array}
\right.
\end{equation}
with $d = \{1,2,\cdots,N\}$ leads to the closed form expression with $\sigma_{N,n} = \frac{1}{n!}C_{n-1}$ 
\end{theorem}

\begin{proof}

The conditions of the summation set can be rewritten as
\begin{equation}
\sigma_{N,n} =
{\sum\limits_{\substack{r_1}}}
{\sum\limits_{\substack{r_2\\ r_2\neq r_1}}}
\ldots
{\sum\limits_{\substack{r_n\\ r_n\neq r_{n-1}\\ \vdots \\ r_n\neq r_1}}}
{v_{r_1}v_{r_2}\cdots v_{r_n}}
\label{eq:hyper_cube_proof_1}
\end{equation}
An equivalent interpretation of \cref{eq:hyper_cube_proof_1} is to relax the restrictions of the last summation and subtract the new terms from the relaxation. Therefore \cref{eq:hyper_cube_proof_1} is rewritten as

\begin{multline}
\sigma_{N,n}  =
{\sum\limits_{\substack{r_1}}}
{\sum\limits_{\substack{r_2\\ r_2\neq r_1}}}
\ldots
{\sum\limits_{\substack{r_{n-1}\\ r_{n-1}\neq r_{n-2}\\ \vdots \\ r_{n-1}\neq r_{1}}}}
{\sum\limits_{r_n}}
{v_{r_1}v_{r_2}\cdots v_{r_n}} \\
-
{\sum\limits_{\substack{r_1}}}
{\sum\limits_{\substack{r_2\\ r_2\neq r_1}}}
\ldots
{\sum\limits_{\substack{r_{n-1}\\ r_{n-1}\neq r_{n-2}\\ \vdots \\ r_{n-1}\neq r_{1}}}}
{v_{r_1}^2v_{r_2}\cdots {v_{r_{n-1}}}} \\ 
-
{\sum\limits_{\substack{r_1}}}
{\sum\limits_{\substack{r_2\\ r_2\neq r_1}}}
\ldots
{\sum\limits_{\substack{r_{n-1}\\ r_{n-1}\neq r_{n-2}\\ \vdots \\ r_{n-1}\neq r_{1}}}}
{v_{r_1}v_{r_2}^2\cdots {v_{r_{n-1}}}} \\ 
\cdots -
{\sum\limits_{\substack{r_1}}}
{\sum\limits_{\substack{r_2\\ r_2\neq r_1}}}
\ldots
{\sum\limits_{\substack{r_{n-1}\\ r_{n-1}\neq r_{n-2}\\ \vdots \\ r_{n-1}\neq r_{1}}}}
{v_{r_1}v_{r_2}\cdots {v_{r_{n-1}}}^2}~~~~~~~~~~~~~~~~
\end{multline}
Collecting and re-arranging terms:
\begin{multline}
\sigma_{N,n} =
{\sum\limits_{\substack{r_1}}}
{\sum\limits_{\substack{r_2\\ r_2\neq r_1}}}
\ldots
{\sum\limits_{\substack{r_{n-1}\\ r_{n-1}\neq r_{n-2}\\ \vdots \\ r_{n-1}\neq r_{1}}}}
{\sum\limits_{r_n}}
{v_{r_1}v_{r_2}\cdots v_{r_n}} \\
-
{\sum\limits_{\substack{r_1}}}
{\sum\limits_{\substack{r_2\\ r_2\neq r_1}}}
\ldots
{\sum\limits_{\substack{r_{n-1}\\ r_{n-1}\neq r_{n-2}\\ \vdots \\ r_{n-1}\neq r_{1}}}}\Bigg(
{v_{r_1}^2v_{r_2}\cdots {v_{r_{n-1}}}} \\ 
\hspace*{4.6cm}+{v_{r_1}v_{r_2}^2\cdots {v_{r_{n-1}}}} \\ 
\hspace*{6.6cm}\\
\hspace*{5cm}\ldots+{v_{r_1}v_{r_2}\cdots {v_{r_{n-1}}}^2}\Bigg)\\
\end{multline}

\begin{flushleft}
Due to symmetry, the collected terms can simplify by allowing the the last term of each summand $v_{r_{n-1}}$ to be squared:
\end{flushleft} 
\begin{multline} \label{eq_hypercube_proof_4}
\sigma_{N,n} =
{\sum\limits_{\substack{r_1}}}
{\sum\limits_{\substack{r_2\\ r_2\neq r_1}}}
\ldots
{\sum\limits_{\substack{r_{n-1}\\ r_{n-1}\neq r_{n-2}\\ \vdots \\ r_{n-1}\neq r_{1}}}}
{\sum\limits_{r_n}}
{v_{r_1}v_{r_2}\cdots v_{r_n}} \\
-(n-1)
{\sum\limits_{\substack{r_1}}}
{\sum\limits_{\substack{r_2\\ r_2\neq r_1}}}
\ldots
{\sum\limits_{\substack{r_{n-1}\\ r_{n-1}\neq r_{n-2}\\ \vdots \\ r_{n-1}\neq r_{1}}}}
{v_{r_1}v_{r_2}\cdots {v_{r_{n-1}}}^2}
\end{multline}

\begin{flushleft}
The common terms in \cref{eq_hypercube_proof_4} can be factorized and the expression rewritten as
\end{flushleft}
\begin{gather}
\sigma_{N,n} =
{\sum\limits_{\substack{r_1}}}
{\sum\limits_{\substack{r_2\\ r_2\neq r_1}}}
\ldots
{\sum\limits_{\substack{r_{n-1}\\ r_{n-1}\neq r_{n-2}\\ \vdots \\ r_{n-1}\neq r_{1}}}}\left[
{v_{r_1}v_{r_2}\cdots v_{r_{n-1}}}
\left(
{\sum\limits_{r_n}v_{r_n}-(n-1)v_{r_{n-1}}}
\right)\right]\label{eq_hypercube_proof_5}
\end{gather}
\\
Define $f_{1}(v_{r_{n-1}})=v_{r_{n-1}}(\sum\limits_{r_n}(f_{0}(v_{r_n}))-(n-1)f_{0}(v_{r_{n-1}}))$ with terminal function $f_{0}(v)=v$. The combinatorial sum \cref{eq_hypercube_proof_5} is reduced to $n-1$ summations. Repeat this reduction recursively  $n$ times to obtain:

\begin{gather*}
\sigma_{N,n} =
{\sum\limits_{\substack{r_1}}}\left[
{v_{r_1}}
\left({\sum\limits_{{r_2}}f_{n-2}(v_{r_2})-(1)f_{n-2}(v_{r_1})}\right)\right]
\end{gather*}

\begin{gather*}
\sigma_{N,n} =
{\sum\limits_{\substack{r_1}}}
f_{n-1}(v_{r_1})
\end{gather*}
Since  $ \forall i,  r_i\in \{1,2,\cdots,N\}$, we define a general variable $d$
\begin{gather*}
\sigma_{N,n} =
{\sum\limits_{\substack{r}}}
f_{n-1}(v_r)
\end{gather*}
\end{proof}

Now, using \cref{thm:hypercube_summation}, we define a closed form solutions to the matrix equations defined in \cref{eq:eq1}:

\begin{theorem}\label{thm:main_vander}
Let row index $1\leq i \leq N$, column index $1\leq j \leq N$, and $v_j \in C$ be given. The closed form solution to the inverse Vandermonde matrix's (i'th row,j'th column) component is given by
\begin{equation}
\left(V^{-1}\right)_{i,j} =\begin{cases}
    \frac{(-1)^{N-j}}{\lambda_i}
\sigma^{(i)}_{N-1,N-j}\\
     \frac{1}{\lambda_i}, & \text{j=N}.
  \end{cases}\label{eq:eq5}
\end{equation}
where, the set of values is $\{r_1\neq r_2\neq\cdots\neq r_{N-j}\neq i\}$. The ESP with the $i$'th element removed, $\sigma_{N-1,N-j}^{(i)}$, and the pivotal elements, $\lambda(i)$, are obtained from \cref{thm:hypercube_summation} and \cref{eq:eq3}, respectively.
\end{theorem}

\begin{proof}
Use the decomposing matrices $W$ and $A$ defined in \cref{eq:eq2} and \cref{eq:eq4} for the Vandermonde inverse in \cref{eq:eq1}. The solution to the $i,j$th coefficient will be obtained by
\begin{align}
\left(V^{-1}\right)_{i,j} &=W[i\text{th row}] \times A[j\text{th column}]\nonumber\\
 &=\frac{1}{\lambda(i)}\left[v_i^{N-j}+v_i^{N-(j+1)}a_1+v_i^{N-(j+2)}a_2+\cdots+a_{N-j}\right] \label{eq:eq6}
\end{align}
Using the definition for $a_i$ in \cref{eq:eq4}, \cref{eq:eq6} yields
\begin{align}
\left(V^{-1}\right)_{i,j} &=\frac{1}{\lambda_i}\times\left[v_i^{N-j}+(-1)v_i^{N-(j+1)}\sum\limits_{r_1}v_{r_1} + (-1)^2v_i^{N-(j+2)}\sum\limits_{r_1\neq r_2}v_{r_1} v_{r_2} +\cdots \right. \nonumber\\
& \left. +(-1)^{N-j}\sum\limits_{r_1\neq r_2\neq\cdots\neq r_{N-j}}{v_{r_1} v_{r_2}\cdots v_{r_{N-j}}}\right] \label{eq:eq7}
\end{align}
Rewriting each summation into two terms, \cref{eq:eq7} can be rewritten as
\begin{align}
\left(V^{-1}\right)_{i,j} &=\frac{1}{\lambda(i)} \times\Bigg[v_i^{N-j}-v_i^{N-(j+1)}\left(v_i+\sum\limits_{r_1\neq i}v_{r_1}\right)+\nonumber \\
& (-1)^2 v_i^{N-(j+2)}\left(v_i\sum\limits_{r_1\neq i} v_{r_1} + \sum\limits_{r_1\neq r_2\neq i}v_{r_1} v_{r_2}\right)+ \nonumber\\
&\cdots + (-1)^{N-j}\Bigg(v_i\sum\limits_{r_1\neq r_2\neq\cdots\neq r_{N-(j+1)}\neq i}{v_{r_1} v_{r_2}\cdots v_{r_{N-(j+1)}}} \nonumber\\
&+\sum\limits_{r_1\neq r_2\neq\cdots\neq r_{N-j}\neq i}{v_{r_1} v_{r_2}\cdots v_{r_{N-j}}}\Bigg)\Bigg] 
\end{align}
Expanding and collecting possible terms, 
\begin{align}
        V^{-1}(i,j) &=\frac{1}{\lambda(i)} 
        \times
	\Bigg[
            \left(
                (v_i)^{N-j}-(v_i)^{N-j}
            \right) + \nonumber\\
            &\left(
                -(v_i)^{N-(j+1)}\sum\limits_{r_1\neq i}(v_{r_1})+(v_i)^{N-(j+1)}\sum\limits_{r_1\neq i}(v_{r_1})
            \right) + \ldots \nonumber\\ 
            &\ldots+(-1)^{N-j} \sum\limits_{r_1\neq r_2\neq\cdots\neq r_{N-j}\neq i}{(v_{r_1})(v_{r_2})\cdots(v_{r_{N-j}})}
            \Bigg]
\end{align}
The remaining non-cancelled terms are written as
\begin{equation}
V^{-1}(i,j) =\frac{(-1)^{N-j}}{\lambda(i)}\sum\limits_{r_1\neq r_2\neq\cdots\neq r_{N-j}\neq i}{(v_{r_1})(v_{r_2})\cdots(v_{r_{N-j}})}
\end{equation}
\end{proof}

\section{Complex Nodes on the Unit Circle--A Case Study}\label{sec:elementary}
In this section, we consider a particular case of polynomials where the sample nodes are obtained from the $N$th roots of the unit circle in the complex domain. These specific nodes can play a crucial role in sampling problems such as super-resolution where the idea is to recover super-resolved signals obtained from the discrete Fourier transform (DFT) domain. Super-resolution is an inverse problem which becomes highly ill-conditioned in the presence of noise. Readers are referred to the related topics in \cite{candes2014towards, superResolution, batenkov2019super, batenkov2019rethinking} for more information. Motivated by such application in this type of node sampling, our aim in this section is to explore the utility of both theorems, introduced in \ref{thm:hypercube_summation} and \ref{thm:main_vander}, as a numerical framework for the inverse Vandermonde calculation. In particular, we conduct an ablation study on the well-posedness of both ESP and inverse Vandermonde in terms of numerical accuracy and stability.
\begin{figure}[ht]
\centering
\begin{tabular}{p{2.1cm}p{2.1cm}p{2.1cm}p{2.1cm}p{2.1cm}}
\subfloat[$N=4$]{\label{fig:simple_case}\includegraphics[width=.8\linewidth]{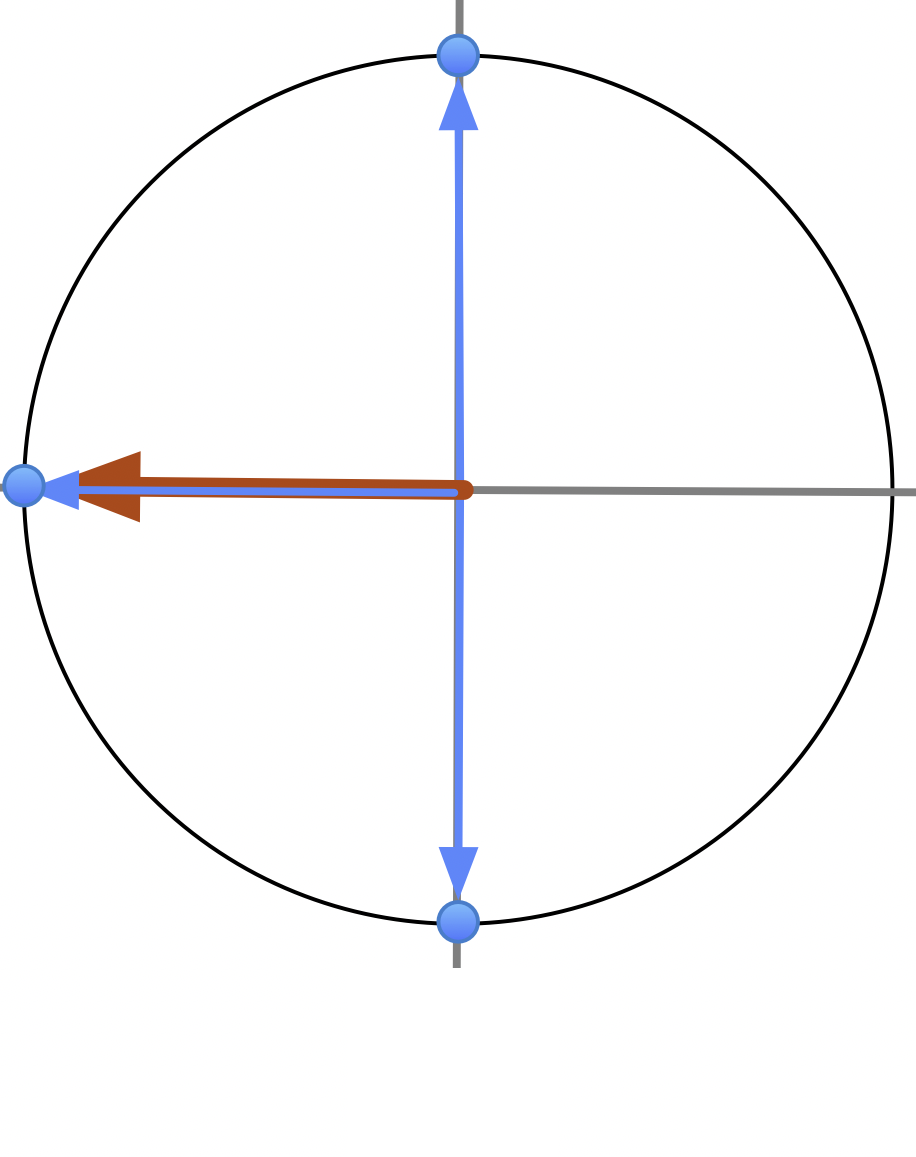}} &
\subfloat[$N=50$]{\label{fig:N_50}\includegraphics[width=1\linewidth]{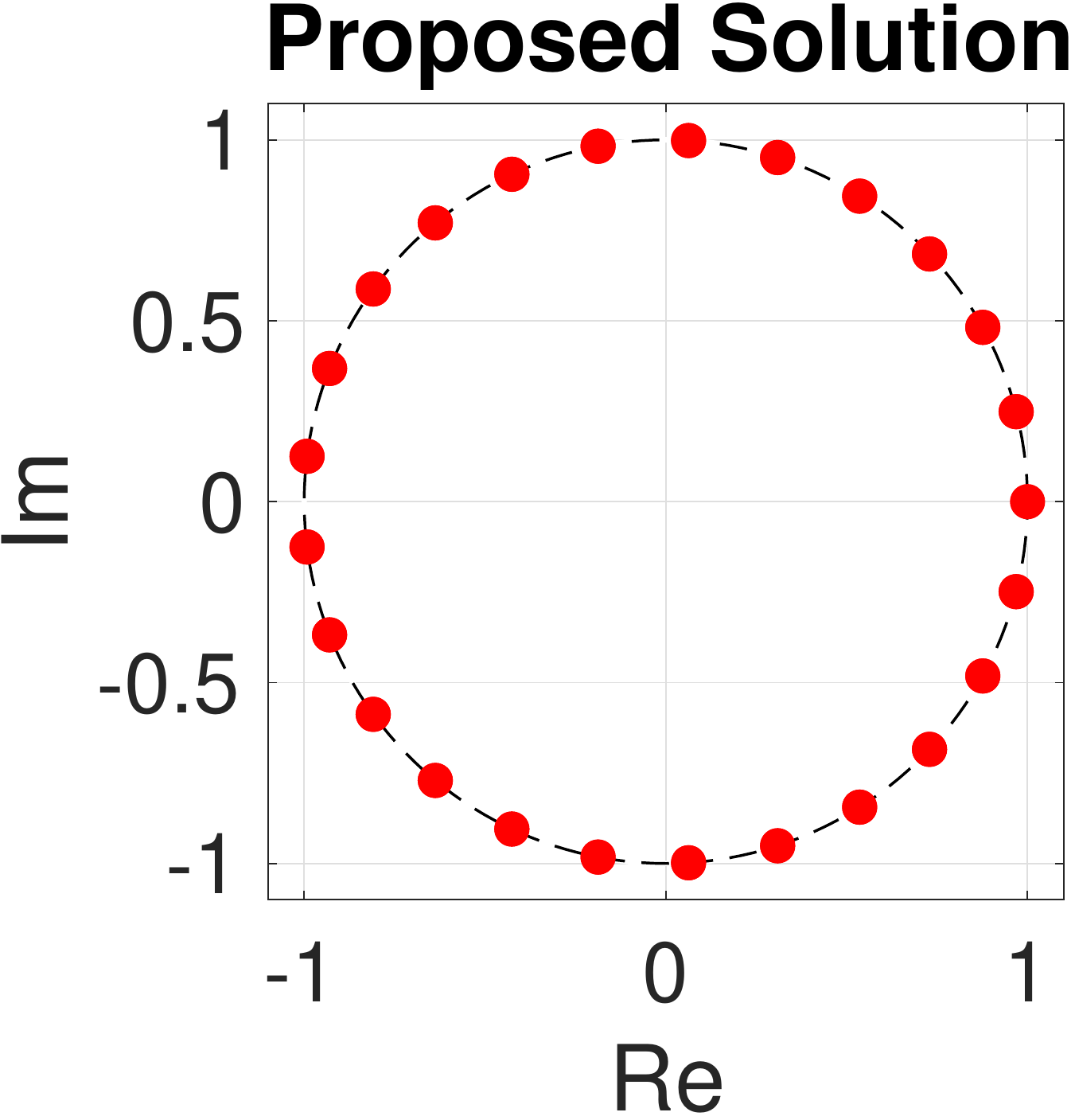}}\enskip &
\subfloat[$N=50$]{\label{fig:NT_50}\includegraphics[width=1\linewidth]{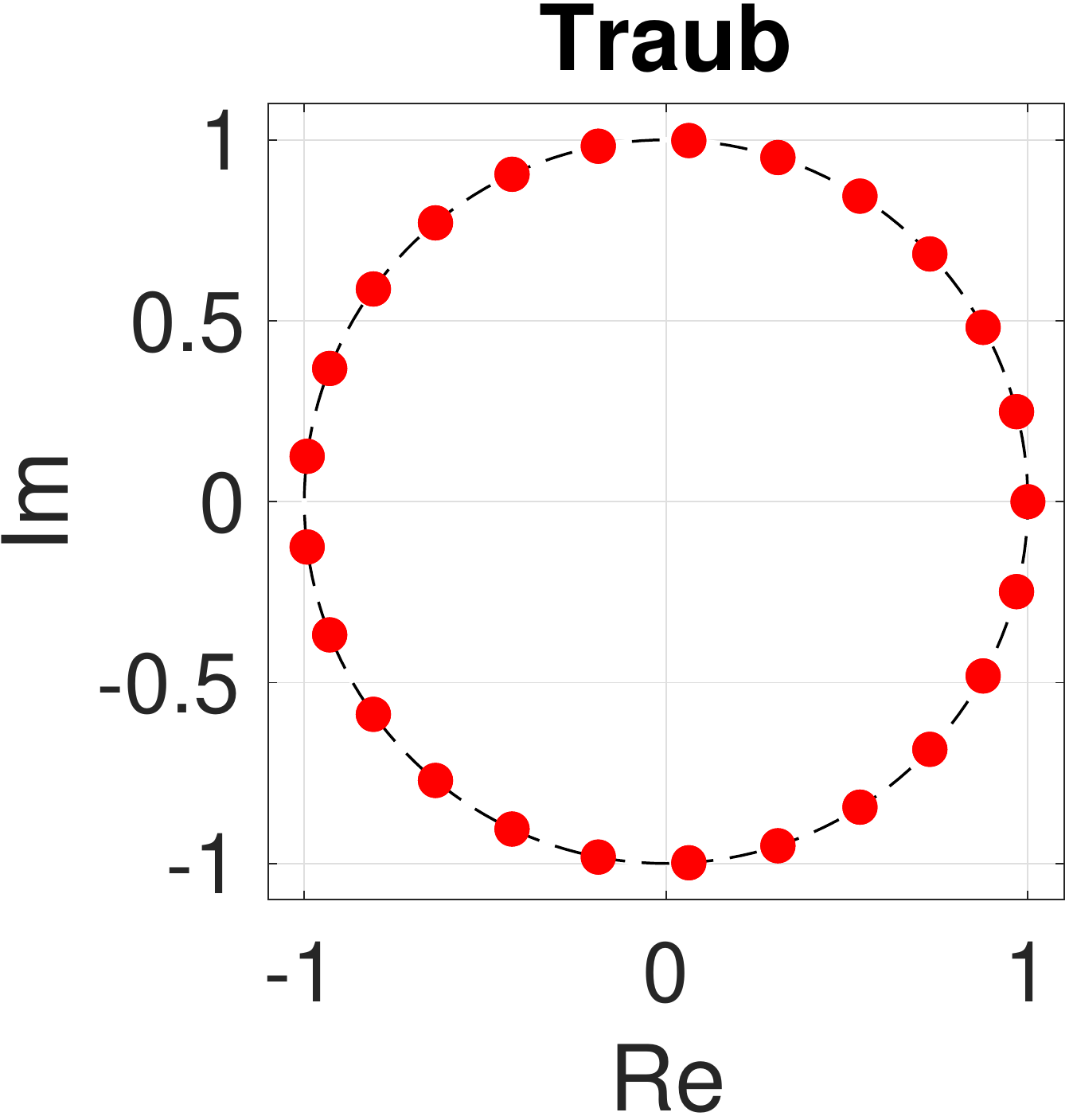}}\enskip &
\subfloat[$N=50$]{\label{fig:NM_50}\includegraphics[width=1\linewidth]{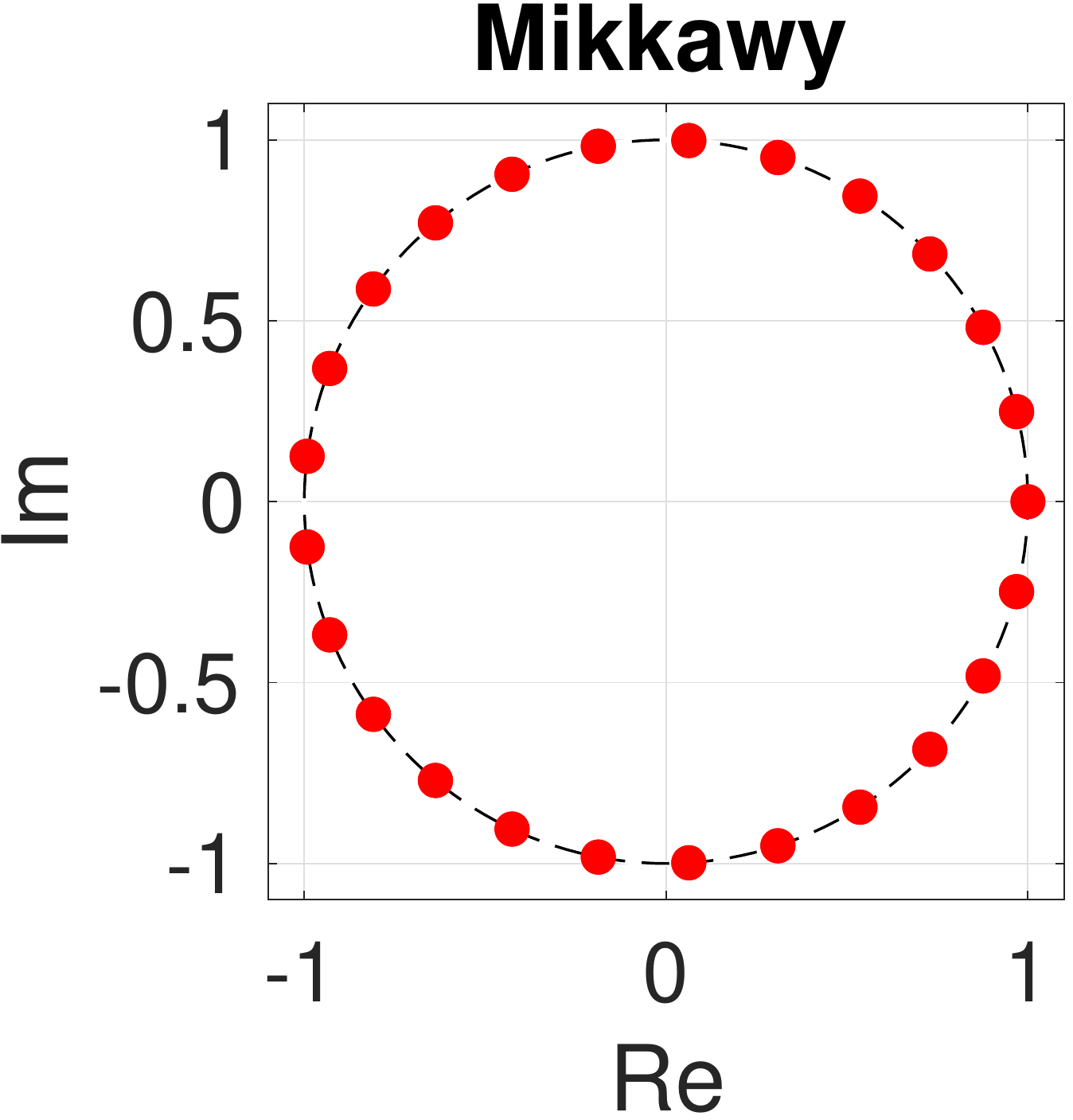}}\enskip &
\subfloat[$N=50$]{\label{fig:NY_50}\includegraphics[width=1\linewidth]{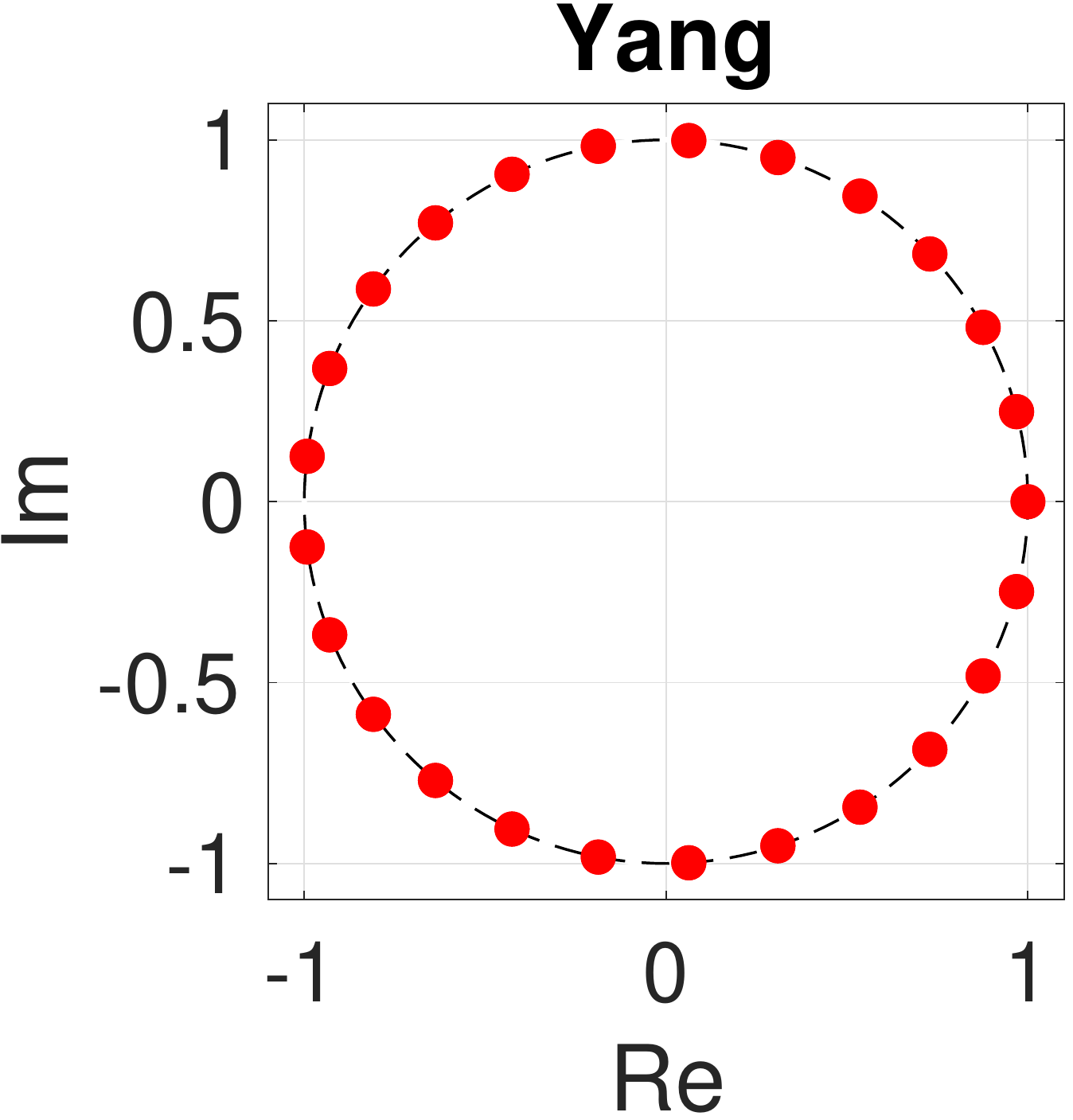}} \\
\multirow{2}{*}{\subfloat[The first consecutive $\frac{N}{2}$ samples ($N=8$)]{\label{fig:half_N}\includegraphics[width=.8\linewidth]{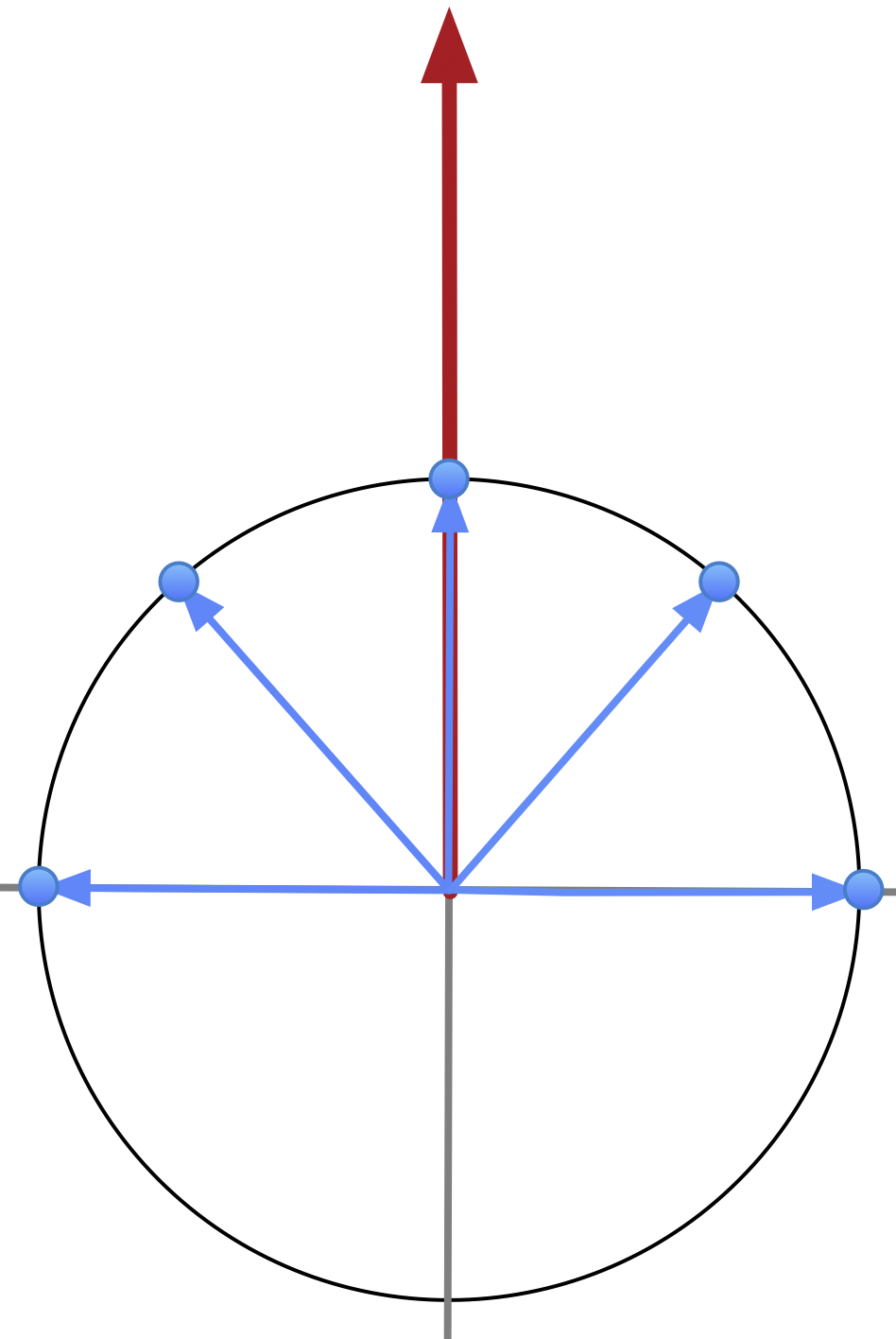}}} &
\subfloat[$N=64$]{\label{fig:N_64}\includegraphics[width=1\linewidth]{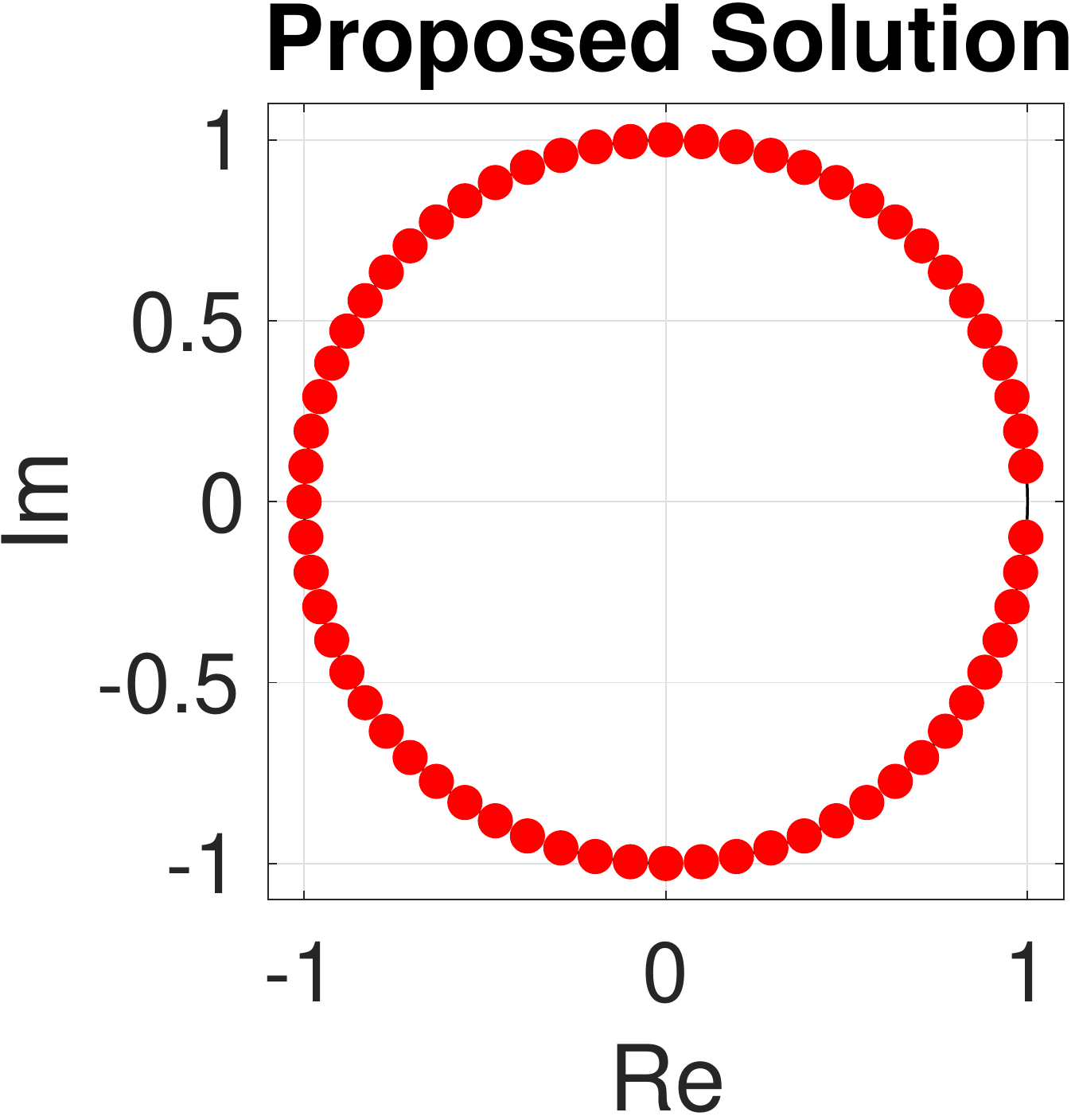}}  \enskip &
\subfloat[$N=64$]{\label{fig:NT_64}\includegraphics[width=1\linewidth]{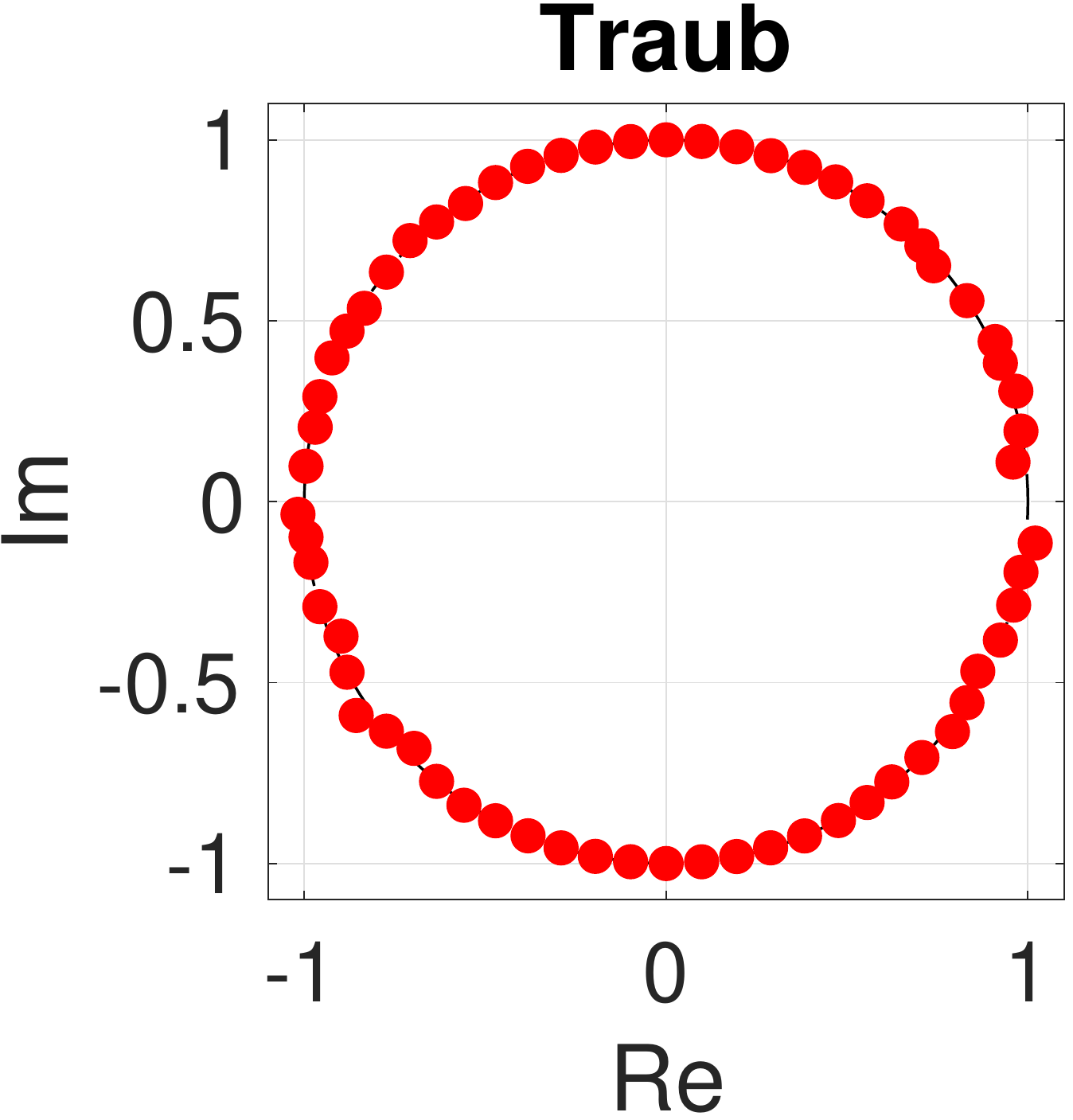}}\enskip &
\subfloat[$N=64$]{\label{fig:NM_64}\includegraphics[width=1\linewidth]{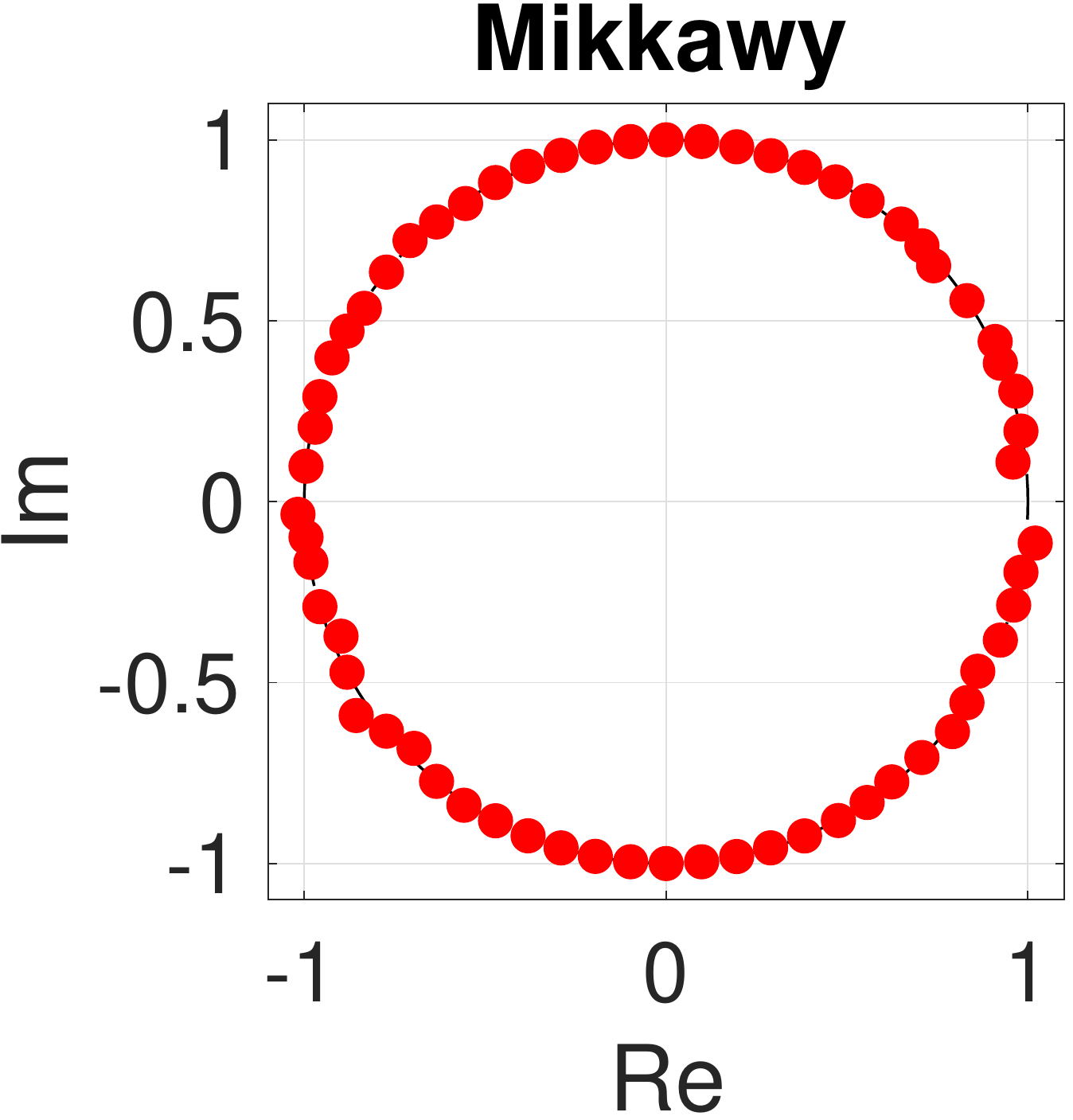}}\enskip &
\subfloat[$N=64$]{\label{fig:NY_64}\includegraphics[width=1\linewidth]{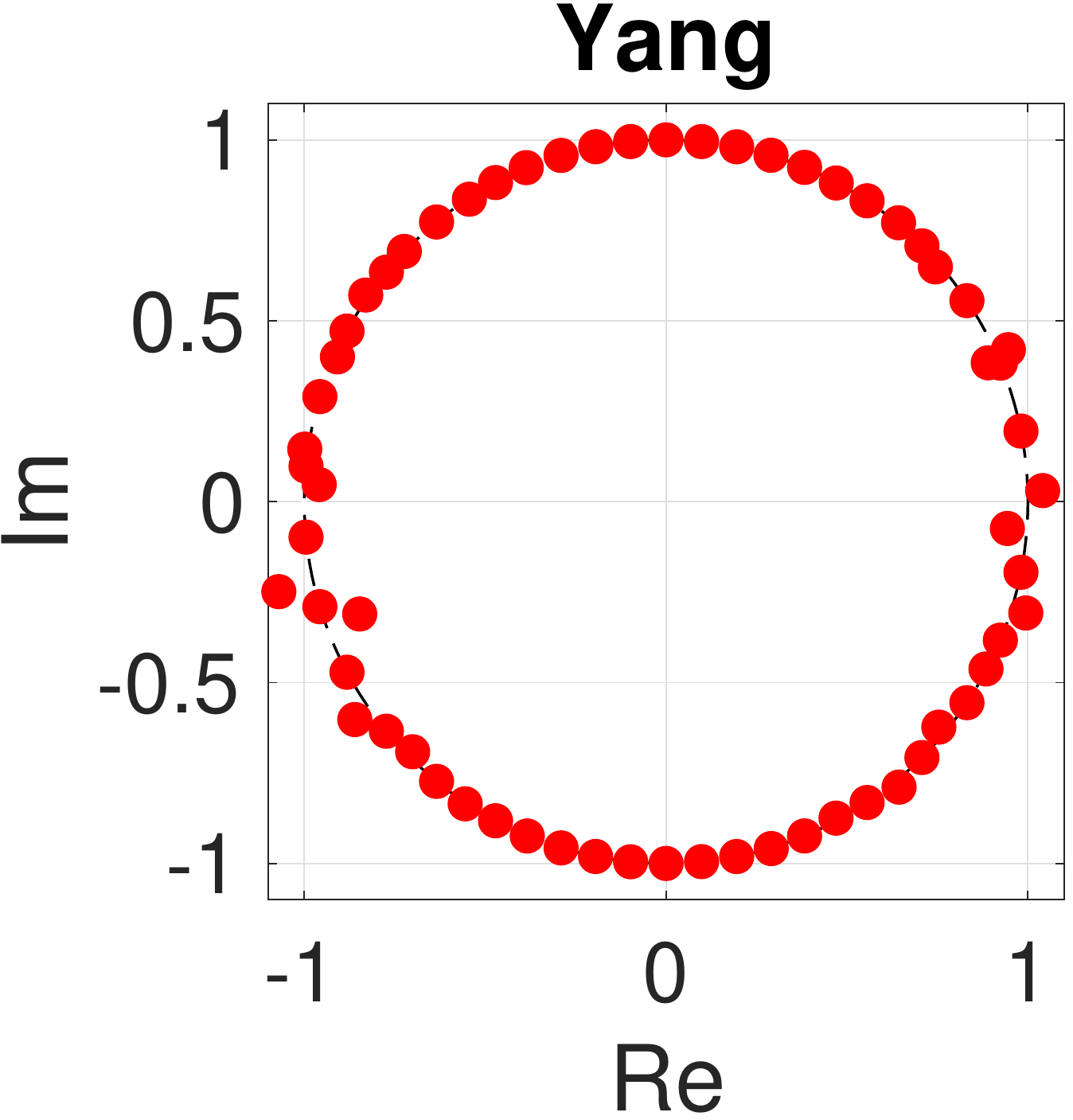}} \\
                                                                                                                                            &
\subfloat[$N=70$]{\label{fig:N_70}\includegraphics[width=1\linewidth]{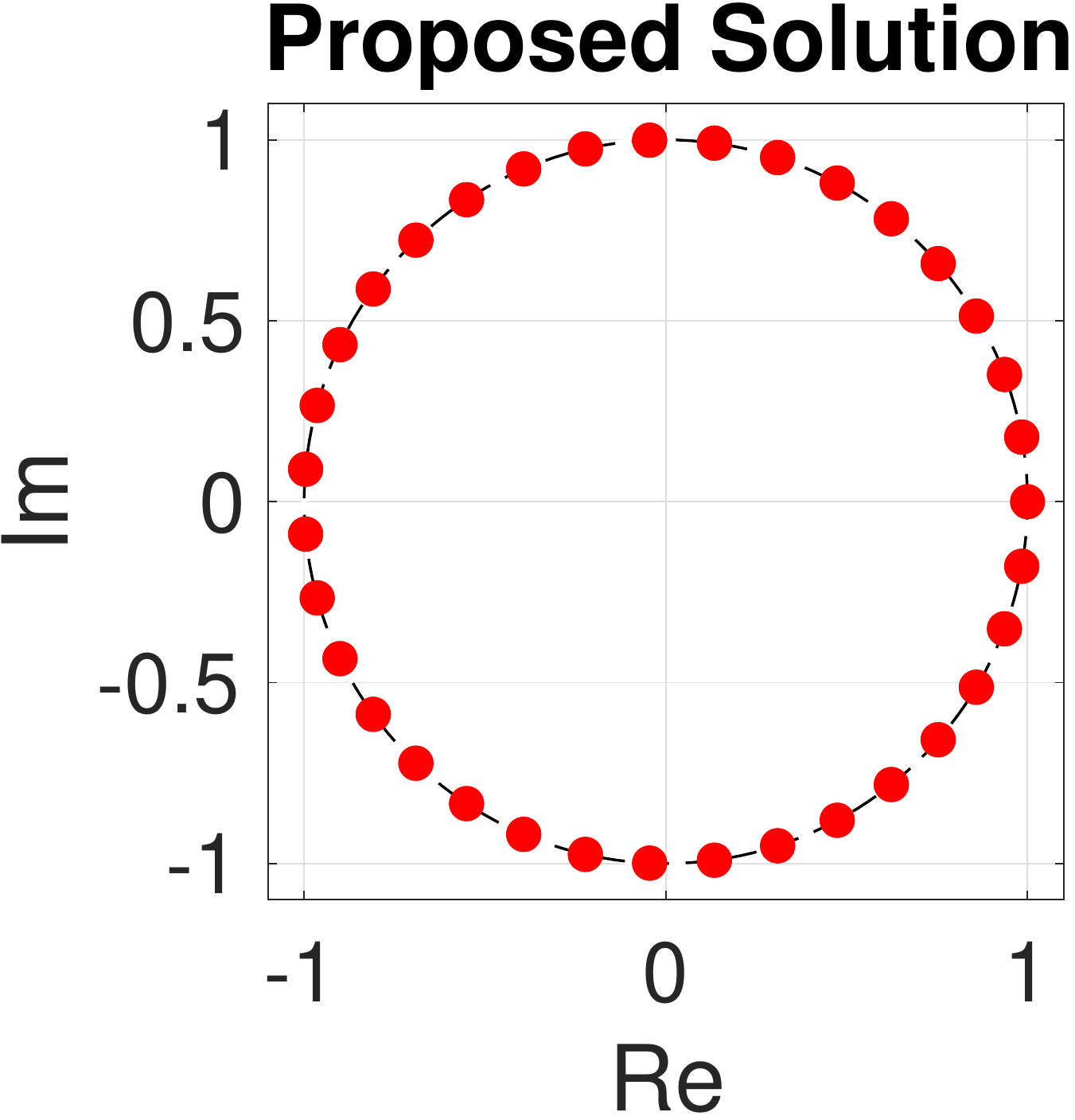}}\enskip &
\subfloat[$N=70$]{\label{fig:NT_70}\includegraphics[width=1\linewidth]{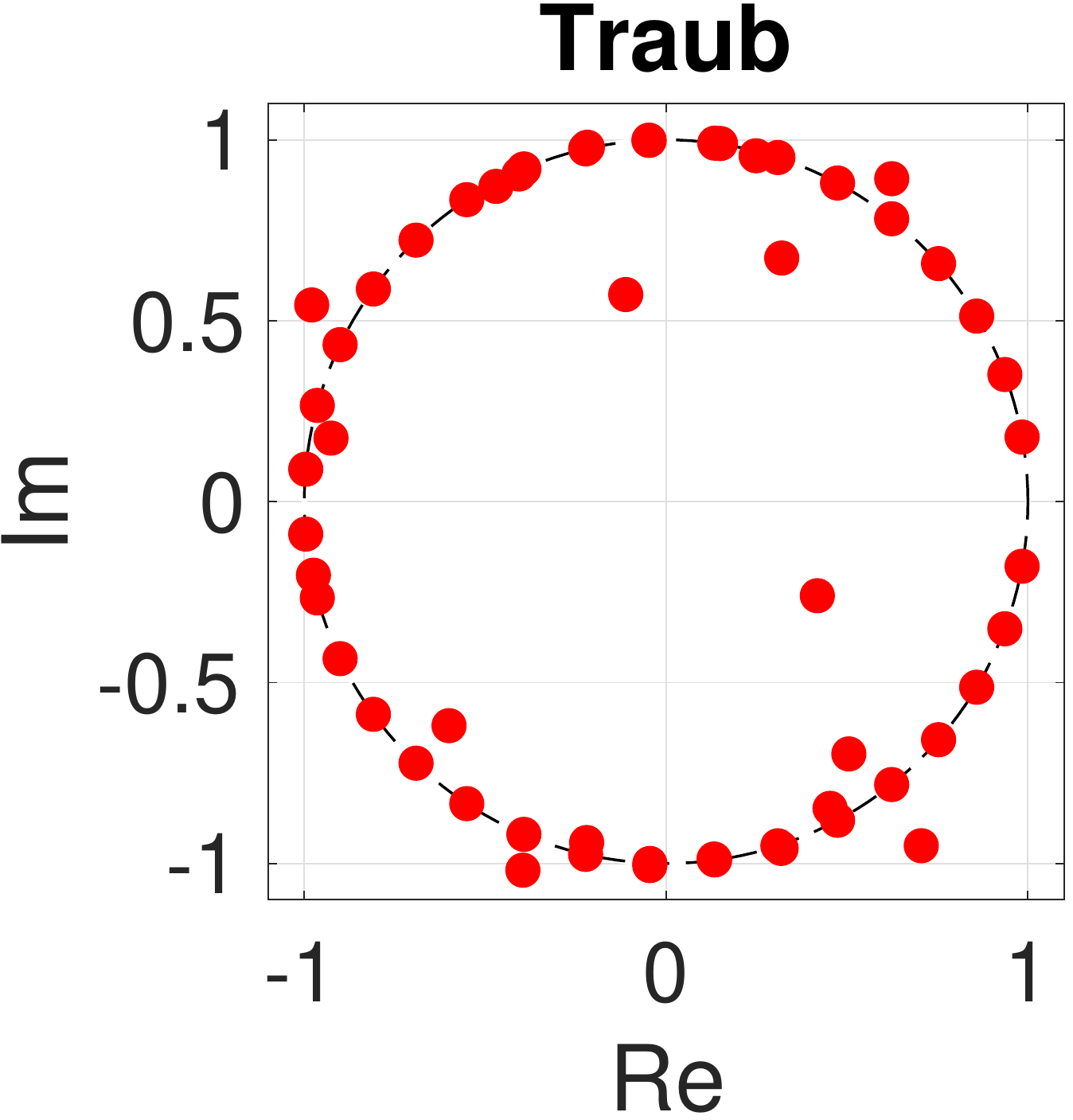}}\enskip &
\subfloat[$N=70$]{\label{fig:NM_70}\includegraphics[width=1\linewidth]{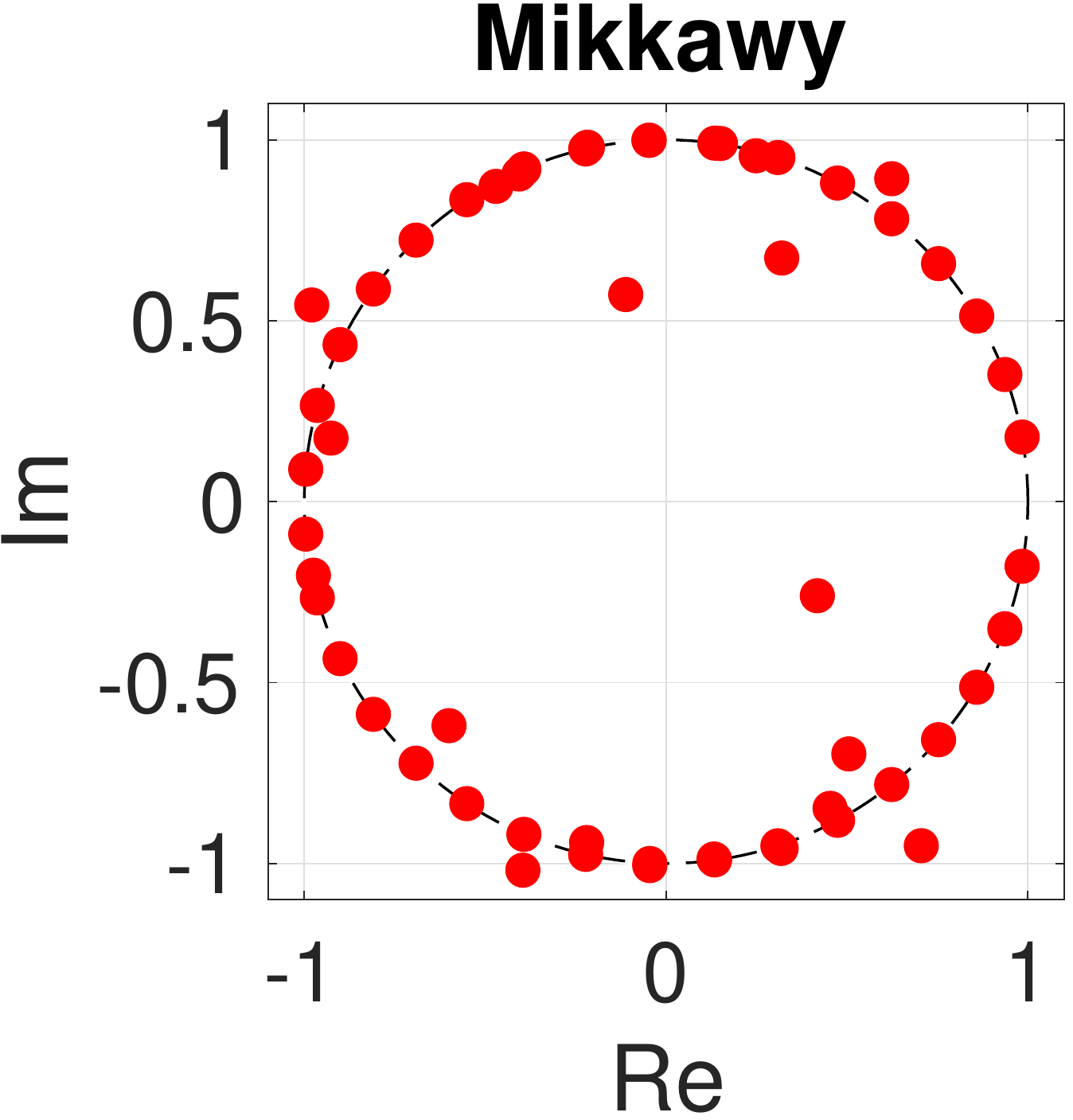}} &
\subfloat[$N=70$]{\label{fig:NY_70}\includegraphics[width=1\linewidth]{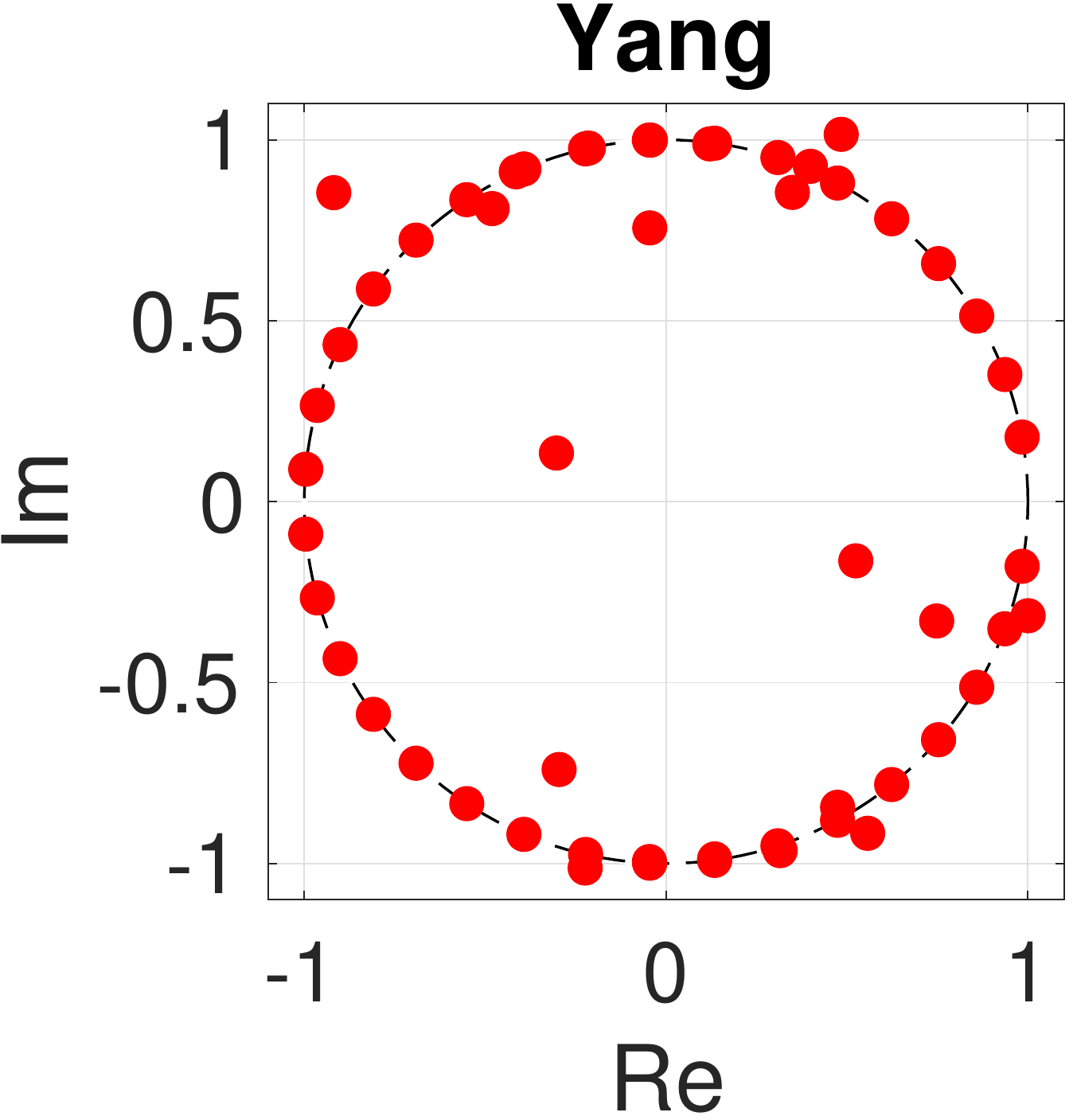}}																																																																			
\end{tabular}
\caption{\cref{fig:simple_case} and \cref{fig:half_N} are to help visualize the properties of the ESP using the roots of unity. Blue points are the samples while the red vector is the resulting imbalance caused by summing all the vectors. In \cref{fig:half_N}, only the horizontal components are canceled while the vertical components of the sample points accumulate and result in a very large offshoot. \cref{fig:NT_50} to \cref{fig:NY_70}) visually presents the significant gain in accuracy when the sampling nodes are the roots of unity. Inaccuracies can be seen in existing methods starting at $64$ samples and will exponentially get worse. Note that for certain sample sizes, such as $70$, the ESP mappings are not unique and hence overlap on top of each other.}\label{fig:bigfig}
\end{figure}
\subsection{Performance on the ESP Calculation}
Provided by the sample points from the $N$th roots of unity, our aim in this section is to calculate the ESP method defined in \cref{thm:hypercube_summation} and compare against the three methods listed in Section \ref{sec_ESP_introduction}. To analyze the stability of symmetric summation methods, we design the following experiment. Note that by dropping an arbitrary element $i$ from the $N$th roots of unity samples the ESP can be simplified to
\begin{align}
|\sigma_{N,n}|=1,~~~\text{for}~n\in\{0,1,\ldots,i-1,i+1,\ldots,N-1\}
\end{align}
The intuition behind this comes from the geometric symmetry of the $N$th roots of unity. Visualized in \cref{fig:simple_case} using the 4'th roots of unity, the roots can be thought of as vectors pointing from the origin to the unit circle. If a vector is removed, the complementary opposite vector on the unit circle will create an imbalance. This imbalance, along with the set being closed under multiplication, helps map the ESP onto the unit circle. The idea is similar for an odd number of samples. However, since an odd number of samples are not evenly distributed on the angle domain, the complementary opposite vector is a summation of components from multiple vectors. 

The results of $\sigma_{N,n}$ using the roots of unity sample set, for $N=\{50, 64, 70\}$ samples, are visualized in \cref{fig:bigfig} and summarized in \cref{fig:UC_Acc}. Both figures show the significant improvements in accuracy for the proposed ESP method against other state-of-the-art methods. Perturbations can be visually seen in \cref{fig:bigfig} at around $64$ samples for current ESP methods. Note that for some samples, the ESP may not map to unique nodes on the unit circle. For example in \cref{fig:N_70}, the proposed method maps the $70$ roots of unity onto only $36$ points. The inherent differences in the proposed algorithm allows significant accuracy gains when compared to the existing methods.

\begin{figure}[htp]
\centering 
\subfloat[Accuracy]{\label{fig:UC_Acc}\includegraphics[width=0.3\linewidth]{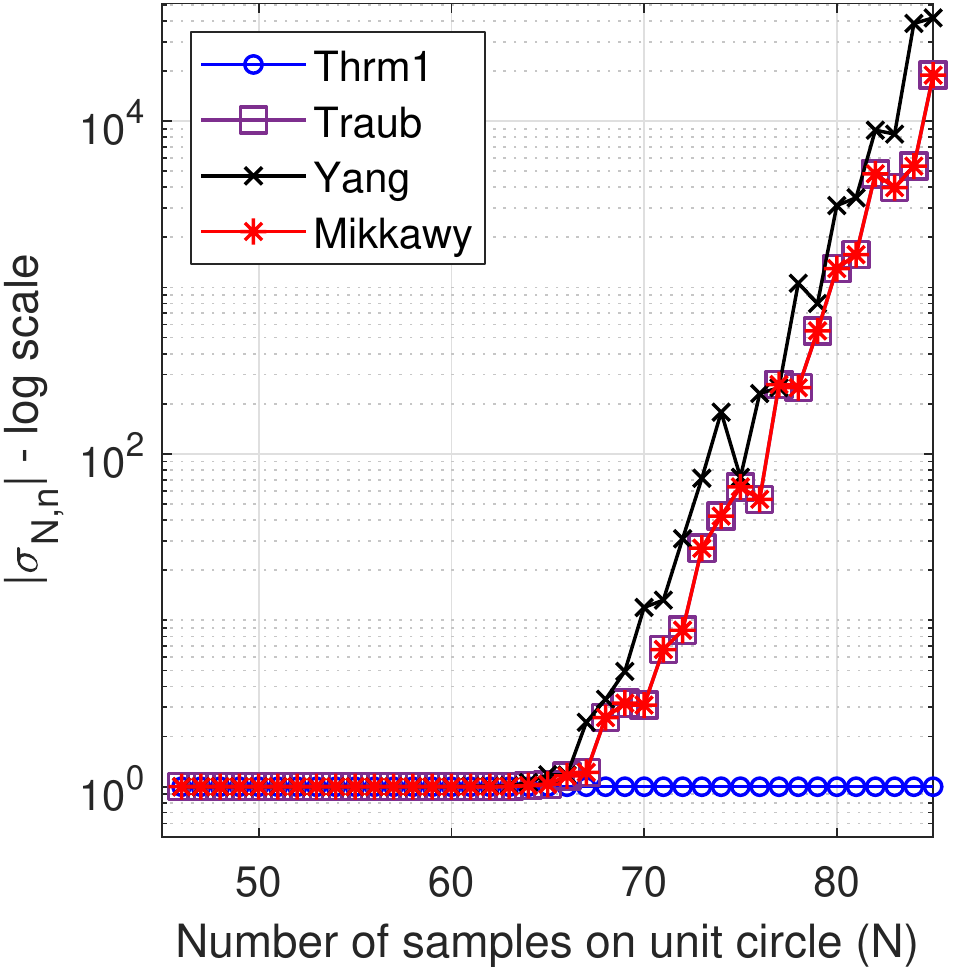}}\hspace{.2in}
\subfloat[Computation speed]{\label{fig:UC_time}\includegraphics[width=0.3\linewidth]{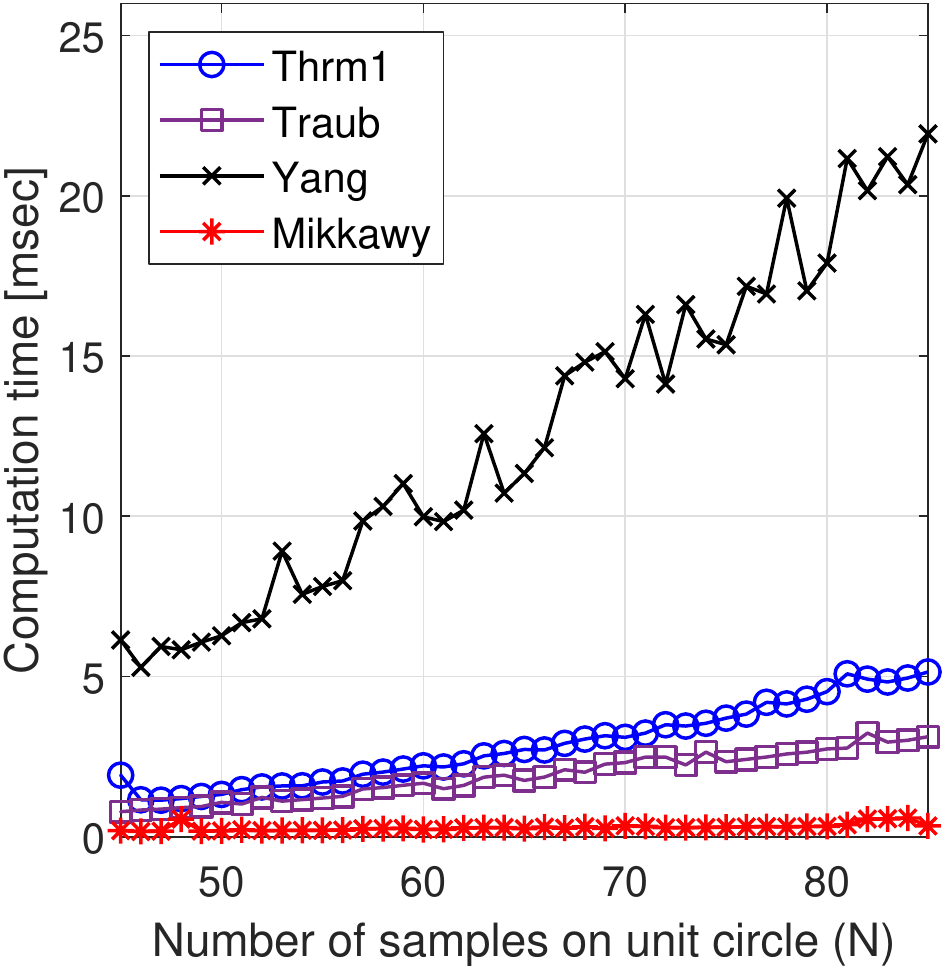}}
\caption{Performance analysis on various solutions to the elementary symmetric polynomials. \cref{fig:UC_Acc} accuracy measurement, the true answer is $10^0$. \cref{fig:UC_time} is measured using MATLAB's ``\texttt{timeit()}'' function. The speed is listed in seconds. }
\end{figure}

The main difference between the proposed and the existing ESP methods is the way the recursions incorporate past calculations defined in Section \ref{sec_ESP_introduction}. Specifically, the algorithm in \cite{Traub_1966} requires $\sigma_{N-1,n-1}$ and $\sigma_{N-2,n-1}$. The algorithm in \cite{Yang_2009} requires past calculations from $\sigma_{N-1-n,0}$ to $\sigma_{N-1,n-1}$. And the algorithm in \cite{Mikkawy_2003} requires $\sigma_{N-2,n}$. Using sample sets of various sizes will lead to the inaccuracies seen in \cref{fig:UC_Acc,fig:bigfig}. The equally distant samples on the unit circle require symmetry to map back to the unit circle. However, the existing algorithms take the first consecutive elements in the set. For example consider $\sigma_{\frac{N}{2},n}$ shown in \cref{fig:half_N} with $N=8$. The ESP methods will take the first $4$ samples for $\sigma_{\frac{N}{2},n}$. Geometrically, this translates to the top half of the unit circle. With only the top half of the circle, the calculations will be heavily clustered towards the top of the complex plane. This will map to a large vector that, by the architecture of the existing algorithms, will propagate throughout the output matrix. In contrast, $\sigma_{N-1,n-1}$ using the proposed solution in \cref{thm:main_vander} only requires $\sigma_{N-1,n-2}$. Therefore, avoiding the problem by retaining the full set of elements for the calculations.

The computational complexity analysis of the ESP methods is summarized in \cref{fig:UC_time}. The analysis is performed using a discrete range of samples $N=\{40,\cdots,85\}$ using a Monte-Carlo simulation with a $1000$ iterations. The proposed solution in \cref{thm:hypercube_summation} remains competitive with the top two ESP methods i.e. Traub and Mikkawy. Note that the summation of $C_n$ in every recursive step for \cref{thm:hypercube_summation} is the largest computational footprint of the algorithm. Still, the novel approach to the proposed method deviates from the trend seen from the existing solutions and allow a very stable calculation for the $N$th roots of unity.

\subsection{Performance on the Vandermonde Inverse Calculation}
The numerical instability caused by ESP methods are the leading sources of error for the Vandermonde inverses considered in this paper. A direct way to prove the claim is to compare the method results with the true inverse. However, the ill-conditioning of the Vandermonde make it impossible to conduct direct performance analysis. Our aim in this section is to establish an indirect evaluation framework to analyze the well-posedness of an inverse Vandermonde solution regardless of its node specification. Recall the Frobenius companion matrix $C_p$ of the monic polynomial $p(t)=c_1+c_2t+\cdots+c_Nt^{N-1}+t^{N}$ where the matrix is diagonalizable by means of a spectral decomposition 
\begin{align}
C_p = V^{-T} \Lambda V,~\text{where}~
\Lambda=
\begin{bmatrix}
v_1 & & \\
& \ddots & \\
& & v_N
\end{bmatrix}~\text{and}~
C_p =
\begin{bmatrix}
0 & 0 & \cdots & 0 & c_1 \\
1 & 0 & \cdots & 0 & c_2 \\
0 & 1 & \cdots & 0 & c_3 \\
\vdots & \vdots & \ddots & \vdots \\
0 & 0 & \cdots & 1 & c_{N}
\end{bmatrix}.\label{eq:Frobenius_companion_matrix}
\end{align}
The sample nodes $v_1,\cdots,v_N$ from the Vandermonde matrix $V_N$ are the eigenvalues associated with the roots of the corresponding monic polynomials. The left identity block matrix from the companion matrix $C_p$ in (\ref{eq:Frobenius_companion_matrix}) is an intriguing framework where it can be used to evaluate the well-posedness of a numerical solution to the Vandermonde matrix inversion. We define the normalized mean square error $\text{NMSE}(\tilde{I}, I) = ||\tilde{I}-I||/||I||$ to measure the skewness of the estimated block identity matrix $\tilde{I}$ (obtained from the above spectral decomposition) from the true identity matrix $I$. This error evaluates the performance of a particular inverse Vandermonde calculation method in terms of its numerical stability and accuracy.

\begin{table}[htp]
\caption{Recovery error $\text{NMSE}(\tilde{I}, I)$ obtained from seven different combinations for inverse Vandermonde calculations. The error is shown for different discritization of the $N$th roots of unity.}
\label{table_iV_varying_N}
\center
\scriptsize{
\begin{tabular}{c|ccc|cccc}
\hlinewd{1.3pt}
\hspace{-.06in}\multirow{2}{*}{$N$}&\multicolumn{3}{c|}{$V_{N}^{-1}$ by Eisinberg \cite{EisinbergFedele2006}} & \multicolumn{4}{c}{$V_{N}^{-1}$ by \cref{thm:main_vander}}\\ \cline{2-8}
& \hspace{-.03in}Traub \cite{Traub_1966} & \hspace{-.03in}Yang \cite{Yang_2009} & \hspace{-.03in}\cref{thm:hypercube_summation} & \hspace{-.03in}Traub \cite{Traub_1966} & \hspace{-.03in}Mikkawy \cite{Mikkawy_2003} & \hspace{-.03in}Yang \cite{Yang_2009} & \hspace{-.03in}\cref{thm:hypercube_summation}\\
\hlinewd{1.3pt}
\hspace{-.06in}$5$  & $5.68\mathrm{e}{-16}$ & $6.00\mathrm{e}{-16}$ & $3.67\mathrm{e}{-16}$ & $3.62\mathrm{e}{-16}$ & $3.62\mathrm{e}{-16}$ & $3.55\mathrm{e}{-16}$ & $3.31\mathrm{e}{-16}$\\ \cline{1-8}
\hspace{-.06in}$10$ & $1.57\mathrm{e}{-14}$ & $1.49\mathrm{e}{-14}$ & $8.54\mathrm{e}{-16}$ & $4.51\mathrm{e}{-15}$ & $4.51\mathrm{e}{-15}$ & $4.09\mathrm{e}{-15}$ & $7.01\mathrm{e}{-16}$\\ \cline{1-8}
\hspace{-.06in}$15$ & $3.19\mathrm{e}{-13}$ & $3.21\mathrm{e}{-13}$ & $1.45\mathrm{e}{-15}$ & $7.51\mathrm{e}{-14}$ & $7.51\mathrm{e}{-14}$ & $7.62\mathrm{e}{-14}$ & $1.06\mathrm{e}{-15}$\\ \cline{1-8}
\hspace{-.06in}$20$ & $5.65\mathrm{e}{-12}$ & $7.02\mathrm{e}{-12}$ & $1.83\mathrm{e}{-15}$ & $1.18\mathrm{e}{-12}$ & $1.18\mathrm{e}{-12}$ & $1.43\mathrm{e}{-12}$ & $1.43\mathrm{e}{-15}$\\ \cline{1-8}
\hspace{-.06in}$25$ & $1.12\mathrm{e}{-10}$ & $1.47\mathrm{e}{-10}$ & $2.28\mathrm{e}{-15}$ & $2.12\mathrm{e}{-11}$ & $2.12\mathrm{e}{-11}$ & $2.80\mathrm{e}{-11}$ & $1.84\mathrm{e}{-15}$\\ \cline{1-8}
\hspace{-.06in}$30$ & $2.09\mathrm{e}{-09}$ & $3.01\mathrm{e}{-09}$ & $3.12\mathrm{e}{-15}$ & $3.74\mathrm{e}{-10}$ & $3.74\mathrm{e}{-10}$ & $5.39\mathrm{e}{-10}$ & $2.41\mathrm{e}{-15}$\\ \cline{1-8}
\hspace{-.06in}$35$ & $3.72\mathrm{e}{-08}$ & $6.34\mathrm{e}{-08}$ & $3.62\mathrm{e}{-15}$ & $6.46\mathrm{e}{-09}$ & $6.46\mathrm{e}{-09}$ & $1.05\mathrm{e}{-08}$ & $2.90\mathrm{e}{-15}$\\ \cline{1-8}
\hspace{-.06in}$40$ & $6.93\mathrm{e}{-07}$ & $1.17\mathrm{e}{-06}$ & $4.04\mathrm{e}{-15}$ & $1.16\mathrm{e}{-07}$ & $1.16\mathrm{e}{-07}$ & $1.92\mathrm{e}{-07}$ & $3.38\mathrm{e}{-15}$\\ \cline{1-8}
\hspace{-.06in}$45$ & $1.44\mathrm{e}{-05}$ & $2.33\mathrm{e}{-05}$ & $4.99\mathrm{e}{-15}$ & $2.22\mathrm{e}{-06}$ & $2.22\mathrm{e}{-06}$ & $3.67\mathrm{e}{-06}$ & $4.03\mathrm{e}{-15}$\\ \cline{1-8}
\hspace{-.06in}$50$ & $2.25\mathrm{e}{-04}$ & $4.66\mathrm{e}{-04}$ & $5.22\mathrm{e}{-15}$ & $3.55\mathrm{e}{-05}$ & $3.55\mathrm{e}{-05}$ & $6.91\mathrm{e}{-05}$ & $4.50\mathrm{e}{-15}$\\
\hlinewd{1.3pt}
\end{tabular}
}
\end{table}

\cref{table_iV_varying_N} demonstrates the error recovery of the inverse Vandermonde calculations using seven different combinations: (a) the Eisinberg et.al. \cite{EisinbergFedele2006} inverse method utilized by three possible ESPs; and (b) the proposed \cref{thm:main_vander} inverse method utilized by four possible ESPs. The nodes are sampled using the $N$th roots of unity with no noise. The progression of the error performance is shown for different discrete samples in the table. Notice the high robustness of the proposed solution throughout different node samples that clearly shows the impact from ESP accuracy for inverse calculations.

Using the same seven combinations for evaluation, we study the robustness of the Vandermonde inversion methods on perturbed sample nodes measured in the vicinity of the $Nth$ roots of unity. We employ two different contaminating factors on the nodes by
\begin{align}
v_n = e^{i2\pi n/N +\eta_{S}} + \eta_M,\label{eq:node_complex_Nth_root}
\end{align}
where $\eta_S\sim\mathcal{N}(0,\sigma^2_S)$ simulates sampling irregularities by randomly shifting the nodes along the roots of unity, and $\eta_M\sim\mathcal{N}(0,\sigma^2_M)$ is the noise magnitude perturbing the the frequency samples. 

\begin{figure}[htp] 
\centering 
\subfloat[Eisin/Prop.]{\label{fig:iV_Eisin_ESP_1_noise_variation}\includegraphics[width=0.225\linewidth]{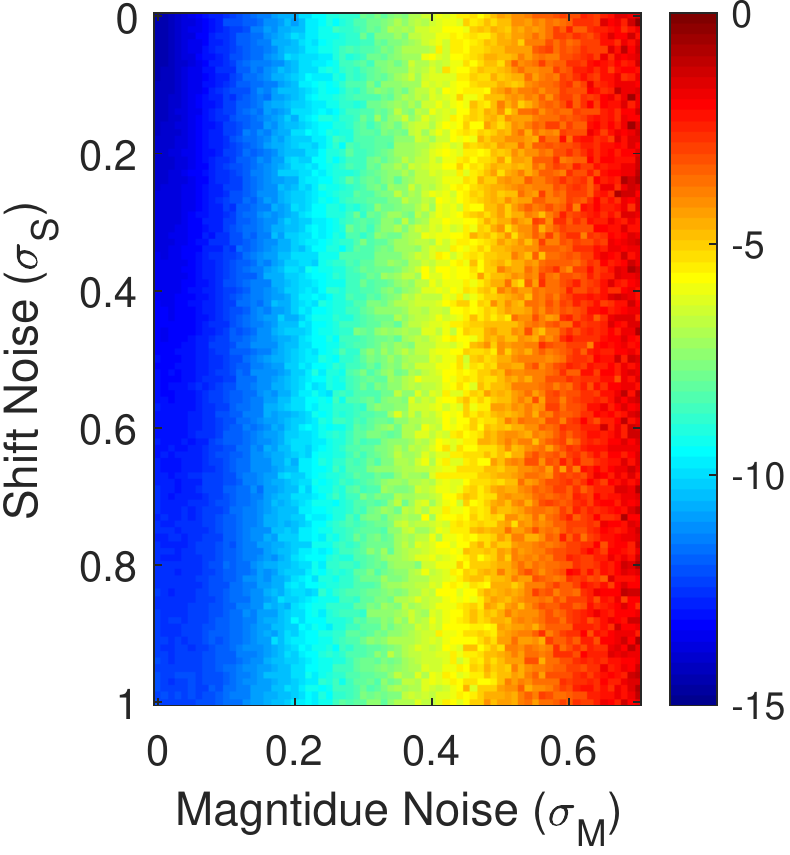}} \enskip
\subfloat[Eisin/Traub]{\label{fig:iV_Eisin_ESP_2_noise_variation}\includegraphics[width=0.225\linewidth]{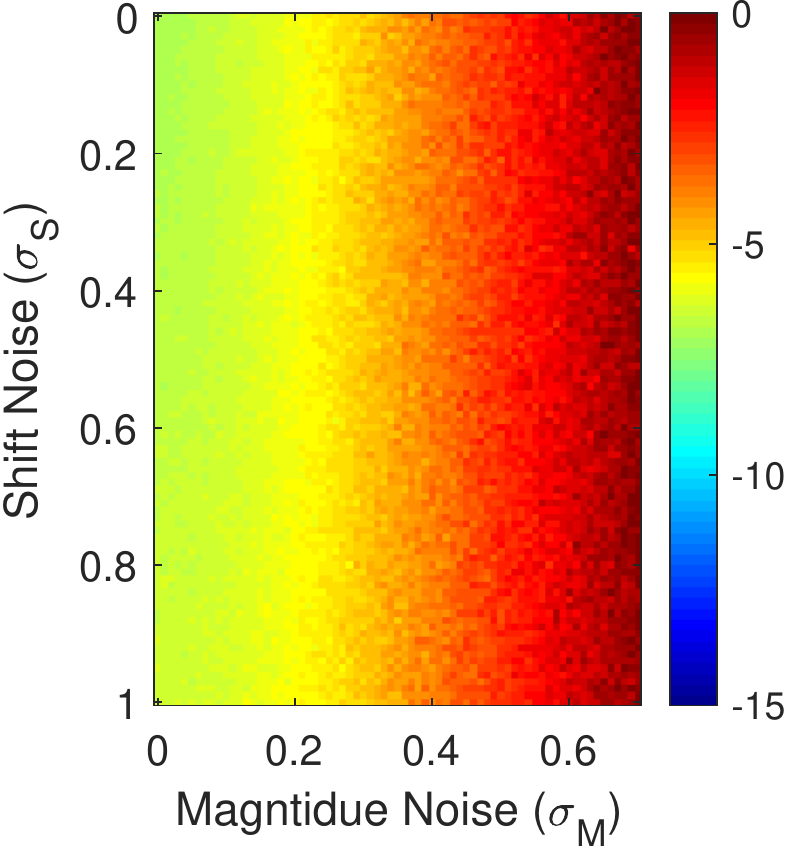}} \enskip
\subfloat[Eisin/Yang]{\label{fig:iV_Eisin_ESP_3_noise_variation}\includegraphics[width=0.225\linewidth]{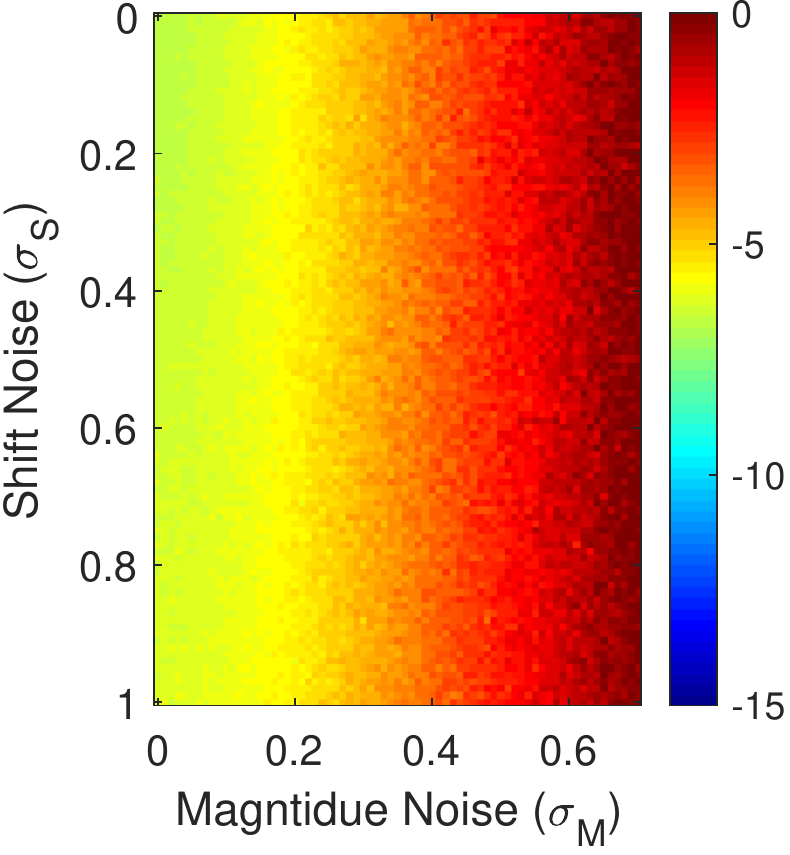}} \\
\subfloat[Prop./Prop.]{\label{fig:iV_proposed_ESP_1_noise_variation}\includegraphics[width=0.225\linewidth]{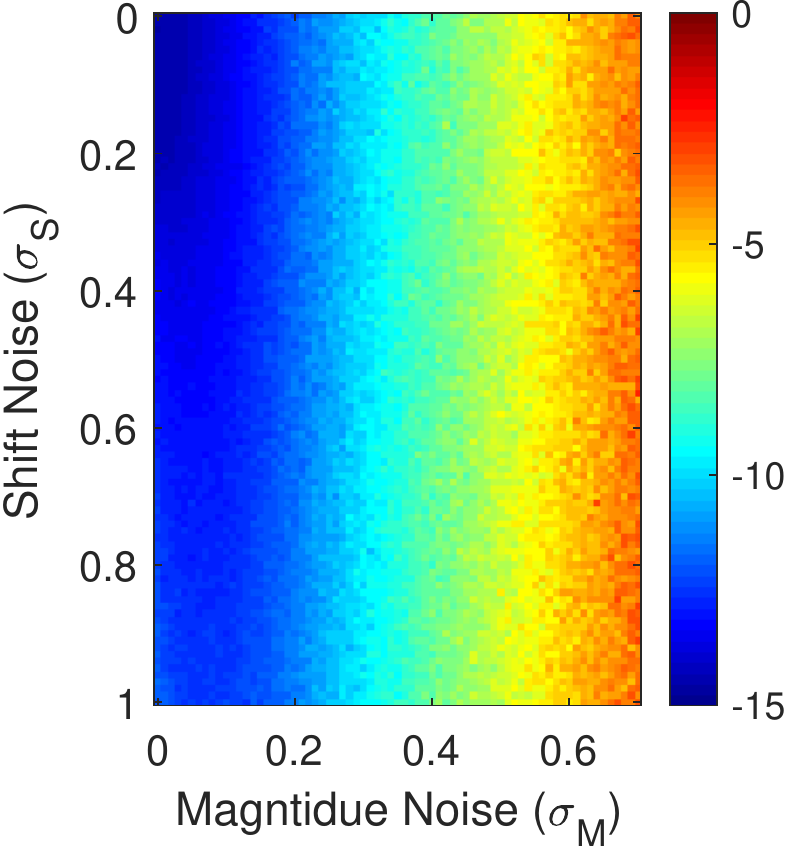}} \enskip
\subfloat[Prop./Traub]{\label{fig:iV_proposed_ESP_2_noise_variation}\includegraphics[width=0.225\linewidth]{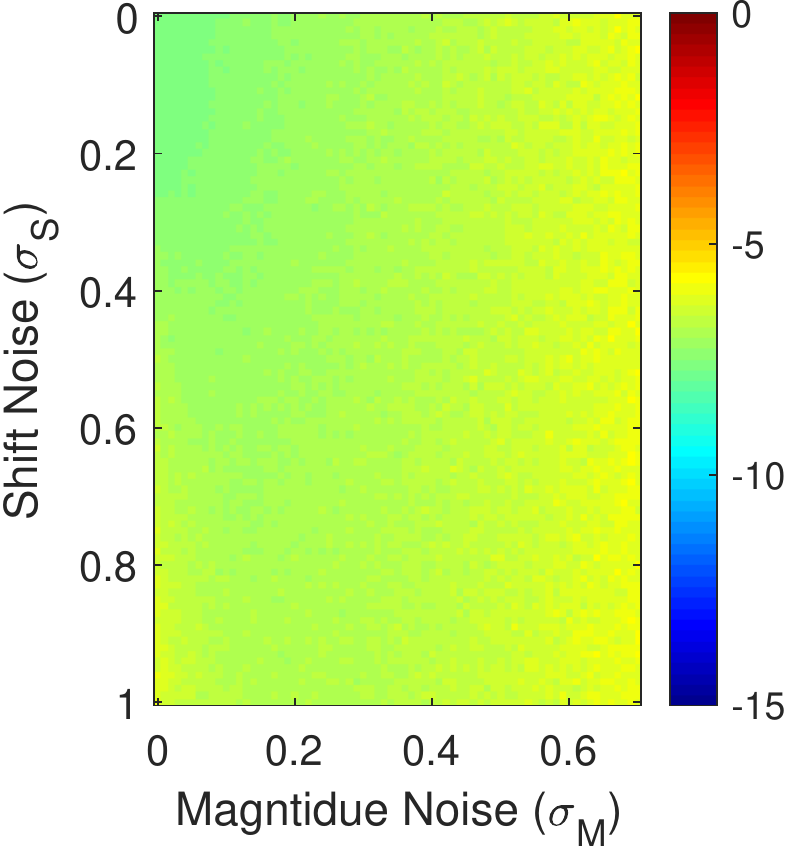}} \enskip
\subfloat[Prop./Yang]{\label{fig:iV_proposed_ESP_3_noise_variation}\includegraphics[width=0.225\linewidth]{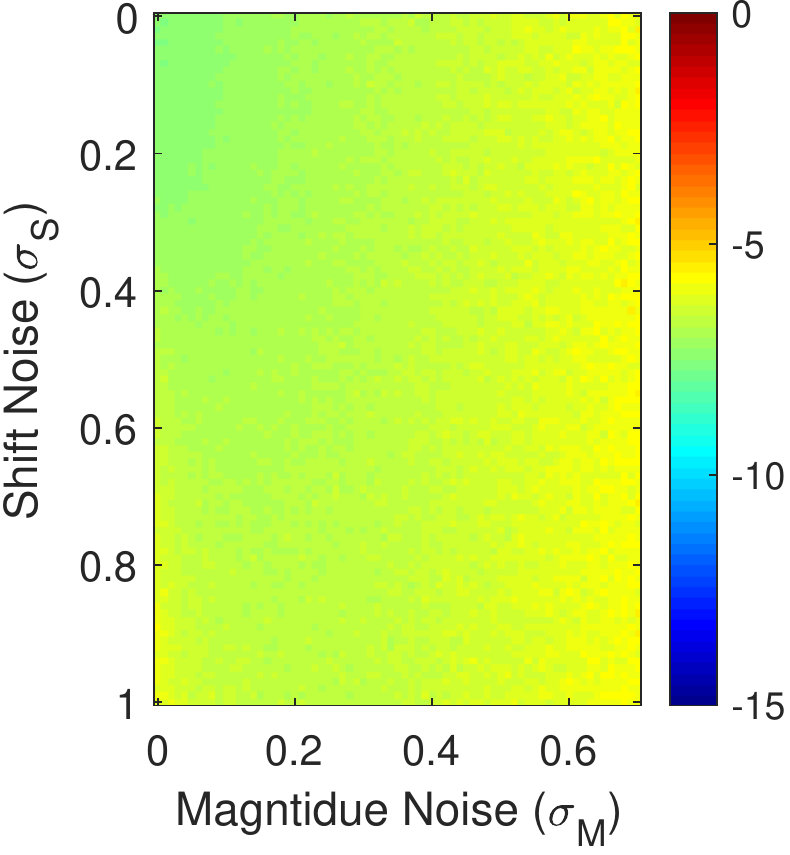}} \enskip
\subfloat[Prop./Mikkawi]{\label{fig:iV_proposed_ESP_4_noise_variation}\includegraphics[width=0.225\linewidth]{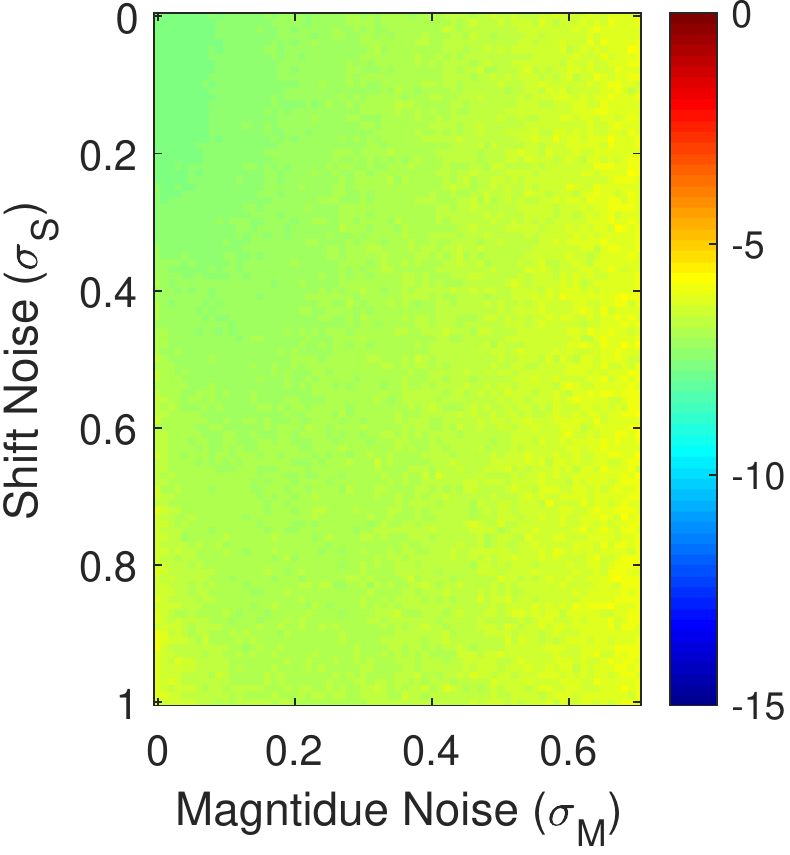}}
\label{fig:iV_noise_variation_Nth_roots_unity}
\caption{Numerical results of Vandermonde inverse calculation, populated 37 samples of complex nodes defined in \cref{eq:node_complex_Nth_root} by modeling two different perturbations for sampling irregularities and frequency magnitudes. Two Vandermonde inverse methods (i.e. Eisinberg \cite{EisinbergFedele2006}, and our proposed \cref{thm:main_vander}) are utilized by four different ESP methods (i.e.  Traub \cite{Traub_1966}, Yang \cite{Yang_2009}, Mikkawy \cite{Mikkawy_2003}, and our proposed \cref{thm:hypercube_summation}) for comparison. The spectral analysis framework defined in \cref{eq:Frobenius_companion_matrix} is used for numerical validations. The shade of colors corresponds to $\log_{10}$ magnitude of error i.e. $\log_{10}\text{NMSE}(\tilde{I},I)$.}\label{fig:Nth_root_noise_variation}
\end{figure}

\cref{fig:iV_noise_variation_Nth_roots_unity} demonstrates the numerical accuracy of the Vandermonde inverse methods. We test the performance of each inverse method by utilizing four different ESP methods i.e. Traub \cite{Traub_1966}, Yang \cite{Yang_2009}, Mikkawy \cite{Mikkawy_2003}, and the proposed \cref{thm:hypercube_summation}. The nodes of the Vandermonde matrix used in \cref{fig:iV_noise_variation_Nth_roots_unity} are the perturbed complex nodes defined in \cref{eq:node_complex_Nth_root} and the number of samples considered to discretize the complex nodes is $37$ in this experiment. In \cref{fig:iV_noise_variation_Nth_roots_unity}, the error is shown in logarithmic scale where the transition of color shades from blue to the red corresponds from low to the high error magnitude, respectively. The standard deviations $\sigma_S$ and $\sigma_M$ are also shown in normalized scale. Notice the robustness of the proposed solution over wide selection of irregular node shifts. For example in \cref{fig:iV_proposed_ESP_1_noise_variation}, the \cref{thm:main_vander} inverse method with the \cref{thm:hypercube_summation} ESP method at a noise level of $\sigma^2_S=0.2$ and $\sigma^2_M=0.1$ has a $6.12$ ($log_{10}$ scale) lower NMSE than the best competing method. Furthermore, \cref{thm:main_vander} significantly improves the robustness of all current ESP methods discussed in this paper from noise magnitudes $\sigma^2_M$ shown in \cref{fig:iV_proposed_ESP_2_noise_variation}-\cref{fig:iV_proposed_ESP_4_noise_variation}.

\section{Experiments}\label{sec:app}
A comprehensive library code written in MATLAB called \texttt{GVAN} is provided along the submission of this paper, available to download at \footnote{\url{https://github.com/mahdihosseini/GVAN/}}, which includes the numerical implementation of \cref{thm:hypercube_summation}, \cref{thm:main_vander}, and existing ESP and inverse methods from the literature for comparison. For more information on this library, please refer to the \texttt{GVAN} user's guidelines and provided demo examples.

In this section, we evaluate the performance of our proposed Vandermonde inverse on one-dimensional interpolation problem; a variant of super-resolution problem. The problem is to solve the system of equation $V^T c = f$ with variables defined in \cref{eq:variables} for finite number of samples. The estimated coefficients are used after for interpolating super-resolved nodes i.e.
\begin{align}
V_N^T c = f_N \rightarrow c = V_N^{-T} f_N \xrightarrow[\text{}]{\text{Use }c} \tilde{f}_{2N} =V_{2N}^{T} c.
\end{align} 
Three analytical functions i.e. $f(x)\in\{\cos(2\pi t x), \tanh(tx), \exp(tx)\}$ are employed to obtain initial samples for interpolation. The nodes for sampling are defined by: (1) equidistant; (2) Chebyshev; (3) extended-Chebyshev; (4) Gauss-Lobbatto (extrema Chebyshev); and (5) $N$th-roots of Unity on the complex plane. The definitions of nodes for (1)-(4) can be found in \cite{EisinbergFedele2006}. Note that the fixed variables for these experiments are chosen based on unique characteristics and usefulness. For example, harmonics is a simple function with only one frequency coefficeint whereas $\tanh$ has a wider frequency spectrum and isolates edge behaviors (such as images). These two functions are chosen because they provide a broad range of analysis for interpolation. The Chebyshev and Extrema Chebyshev nodes were chosen because they are the best and commonly used choice in minimizing the effects of the Runge phenomenon \cite{cheby_facts}. In contrary, the $N$th roots of unity is chosen as it manifests the hardcore problem for super-resolution \cite{candes2014towards, superResolution, batenkov2019super, batenkov2019rethinking}. The normalized mean square error, defined in previous section, is chosen here as the performance metric to analyze the error relative to the true signal i.e. $\text{NMSE}(f_{2N}, \tilde{f}_{2N})$. All experiments are implemented and analyzed in MATLAB R2018b on an 2.9GHz Intel Core i7 machine with 16GB of 2133MHz memory.

Given by two Vandermonde inverse methods and possible ESP solutions, the possible combinations for Vandermonde inversion is summarized in \cref{tab:table_summary} by possible utilization of ESPs in different inverse framework.
\begin{table}[h!]
\caption{Summary of parameter design for Vandermonde inverse calculation. Note that the inherent differences in the Mikkawy solution prevents it from being used in the Eisinberg et.al. inverse.}\label{tab:table_summary}
\center
\scriptsize{
\begin{tabular}{*{5}{c}}
\hlinewd{1.3pt}
\multirow{2}{*}{Inverse Method} &\multicolumn{4}{c}{ESP Solution}\\
&Proposed \cref{thm:hypercube_summation}&{Traub \cite{Traub_1966}}&{Yang \cite{Yang_2009}}&{Mikkawy \cite{Mikkawy_2003}}\\
\hline
Proposed \cref{thm:main_vander} & \checkmark&\checkmark & \checkmark& \checkmark\\
Eisinberg et.al. \cite{EisinbergFedele2006}& \checkmark&\checkmark &\checkmark& $\times$\\   
\hlinewd{1.3pt}
\end{tabular}
}
\end{table}

\cref{fig:nmse_samples} demonstrates the recovery error of the above-mentioned interpolation problem using different analytical functions and sampling nodes. Notice the huge improvement of recovery using $N$th-roots of unity on all three function sampling. This is in concordance with the results discussed in \cref{sec:elementary} where the ESP calculation using \cref{thm:hypercube_summation} is the leading source of such a boost. We also obtained competitive performance using the other sampling nodes (1)-(4) mentioned above. We noticed that for sample nodes in (1)-(4) the accuracy of recovery on signal boundaries hampered the overall results for comparison. Therefore, we excluded seven nodes from both boundaries to minimize any uncontrollable errors and characteristics that happen at the boundaries such as Runge's Phenomenon. For Nth roots of unity sample node at (5) we considered the whole sample domain for error reporting. We have further noticed that on certain sample measurements such as sinusoid and exponential functions, the Eisinberg's framework, associated with our ESP calculation method, provides even further improvement on the recovery error.

\begin{figure}[tbhp]
\centering 
\subfloat[$\sin$: Equid]{\label{fig:NMSE_equidistant_sin}\includegraphics[width=0.2\linewidth]{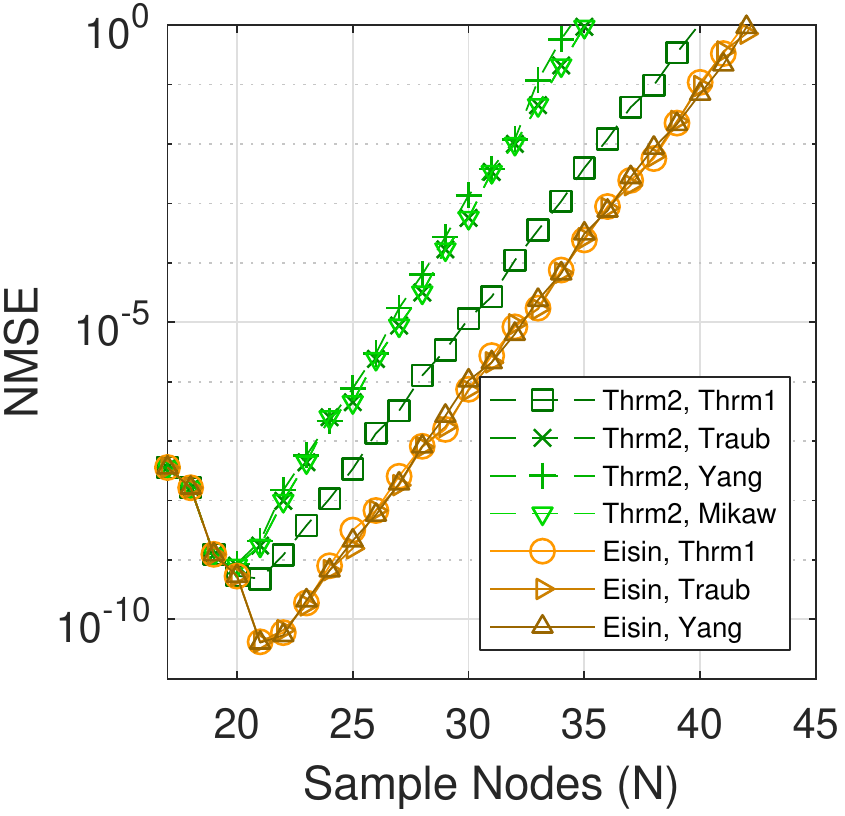}}
\subfloat[$\sin$: Cheby]{\label{fig:NMSE_chebyshev_sin}\includegraphics[width=0.2\linewidth]{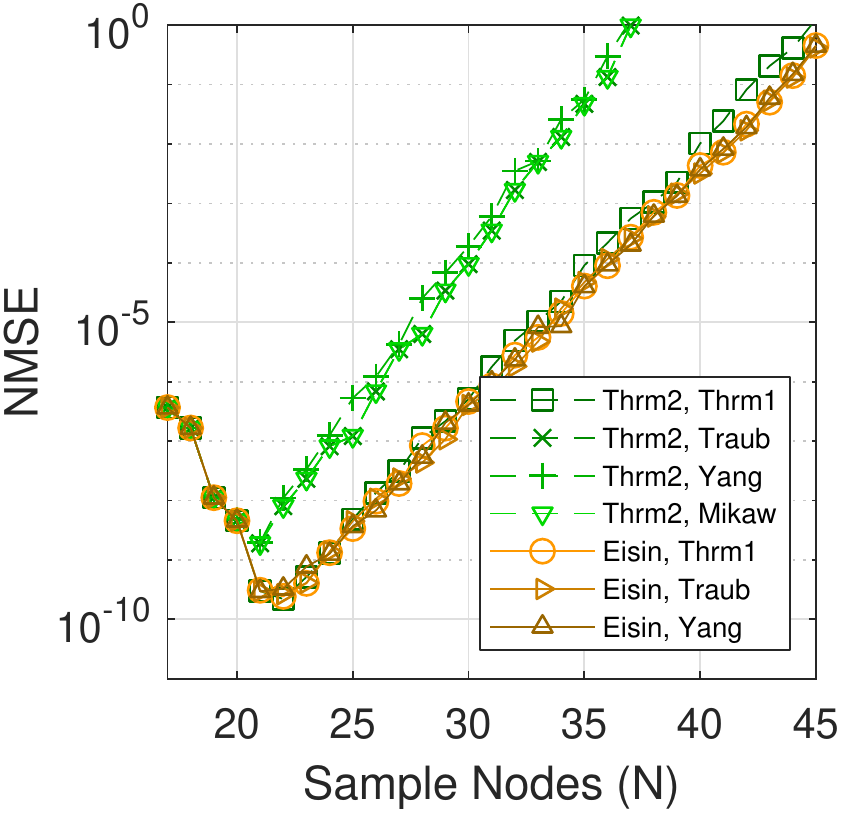}}
\subfloat[$\sin$: Ext-Cheby]{\label{fig:NMSE_extended_chebyshev_sin}\includegraphics[width=0.2\linewidth]{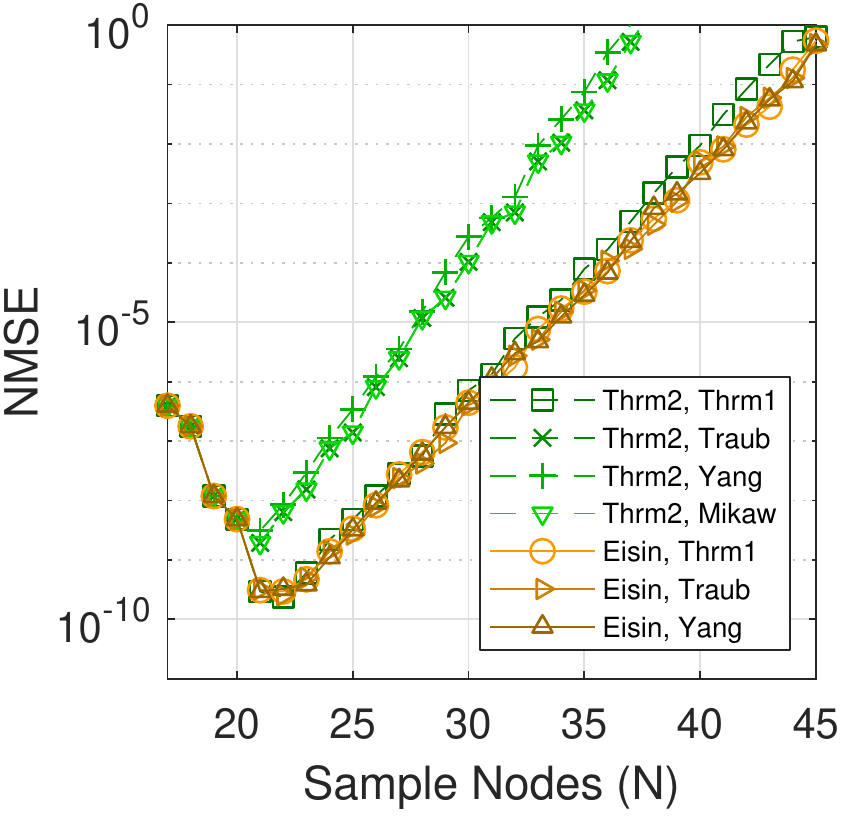}}
\subfloat[$\sin$: Gauss-Lob]{\label{fig:NMSE_gauss_lobbatto_sin}\includegraphics[width=0.2\linewidth]{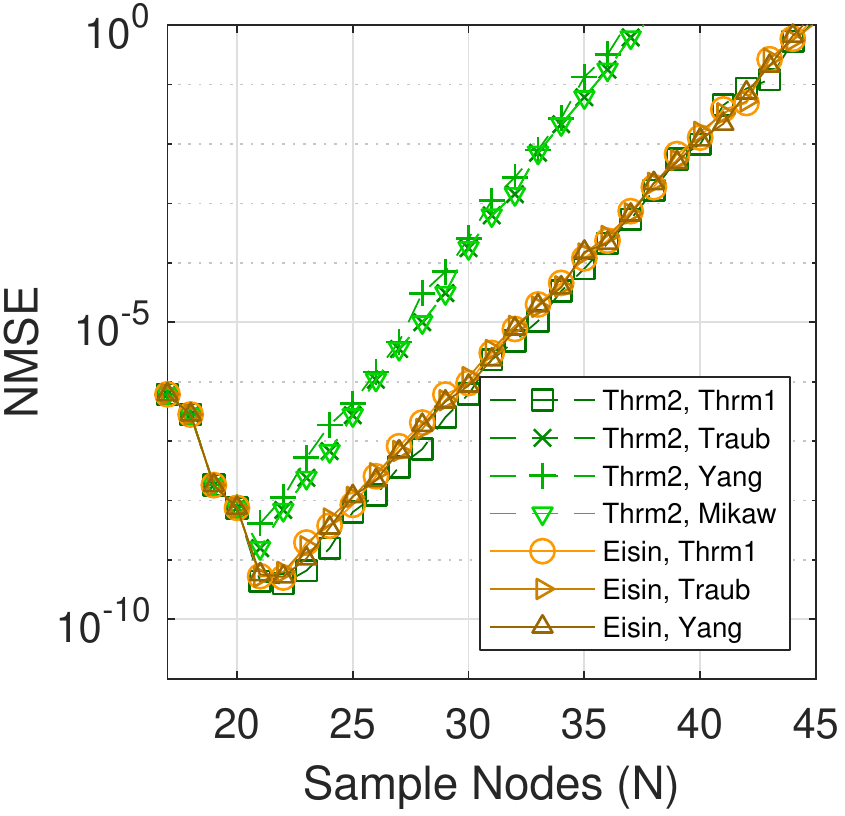}}
\subfloat[$\sin$: $N$th-roots]{\label{fig:NMSE_fekete_sin}\includegraphics[width=0.2\linewidth]{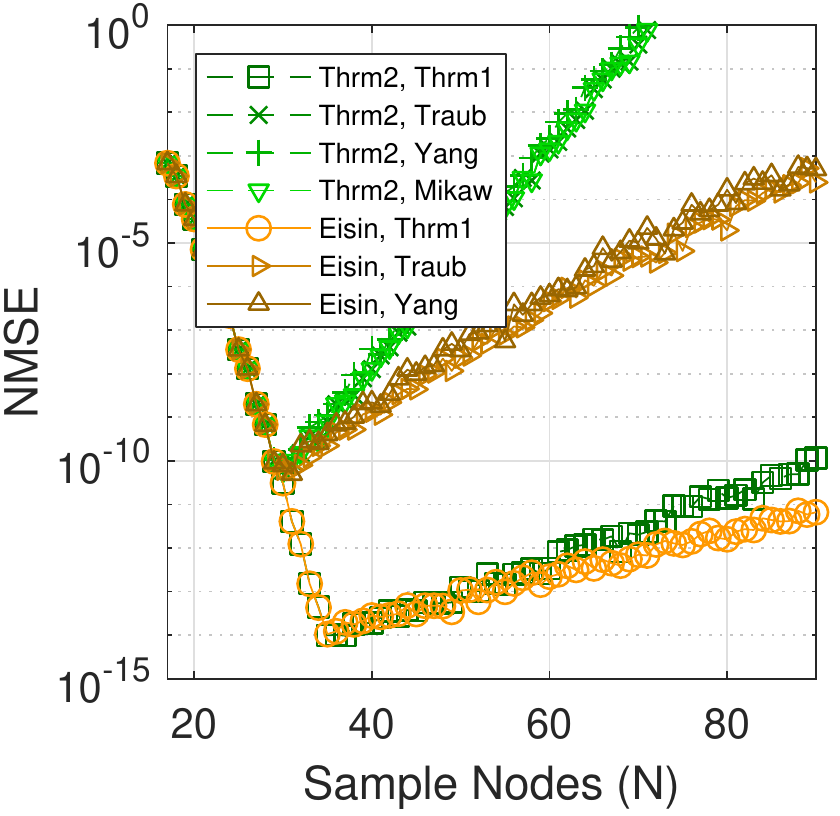}} \\
\subfloat[$\tanh$: Equid]{\label{fig:NMSE_equidistant_tanh}\includegraphics[width=0.2\linewidth]{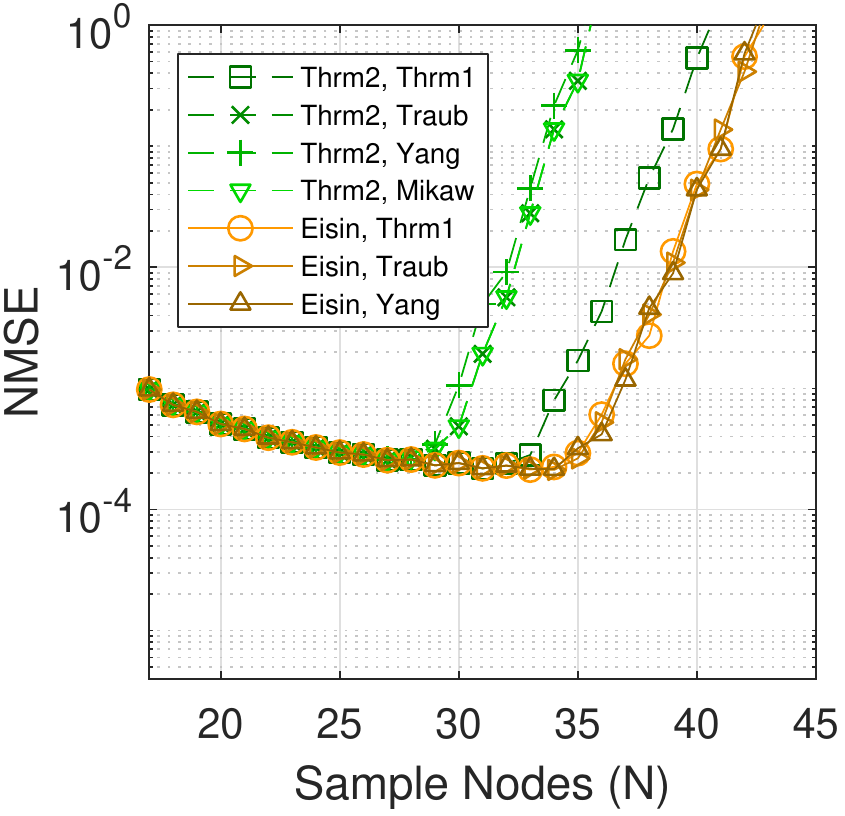}}
\subfloat[$\tanh$: Cheby]{\label{fig:NMSE_chebyshev_tanh}\includegraphics[width=0.2\linewidth]{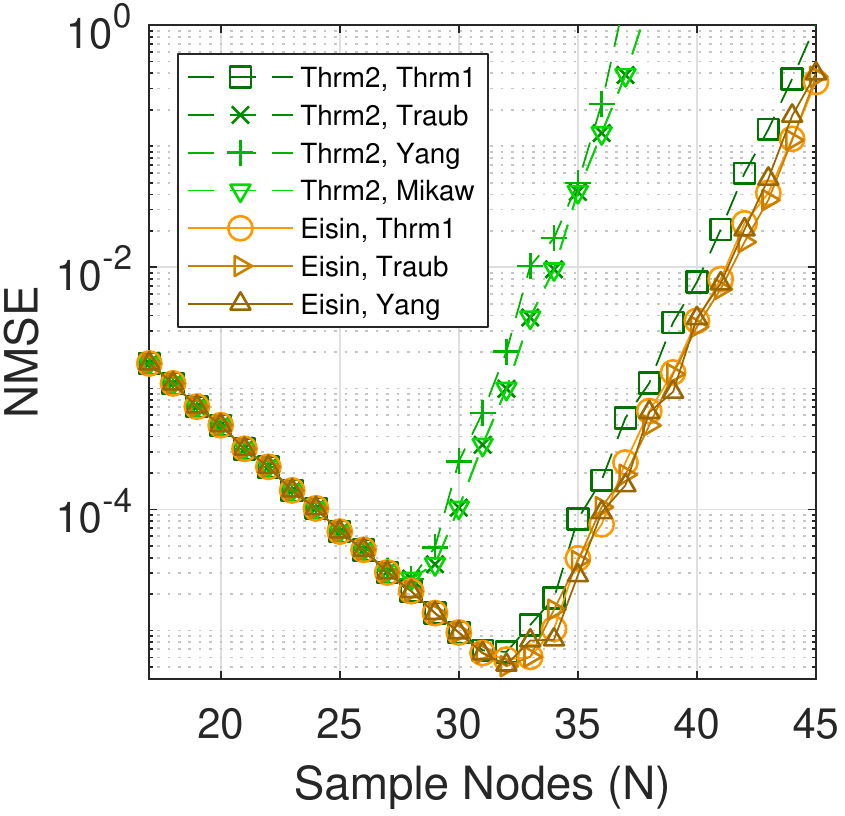}}
\subfloat[$\tanh$: Ext-Cheby]{\label{fig:NMSE_extended_chebyshev_tanh}\includegraphics[width=0.2\linewidth]{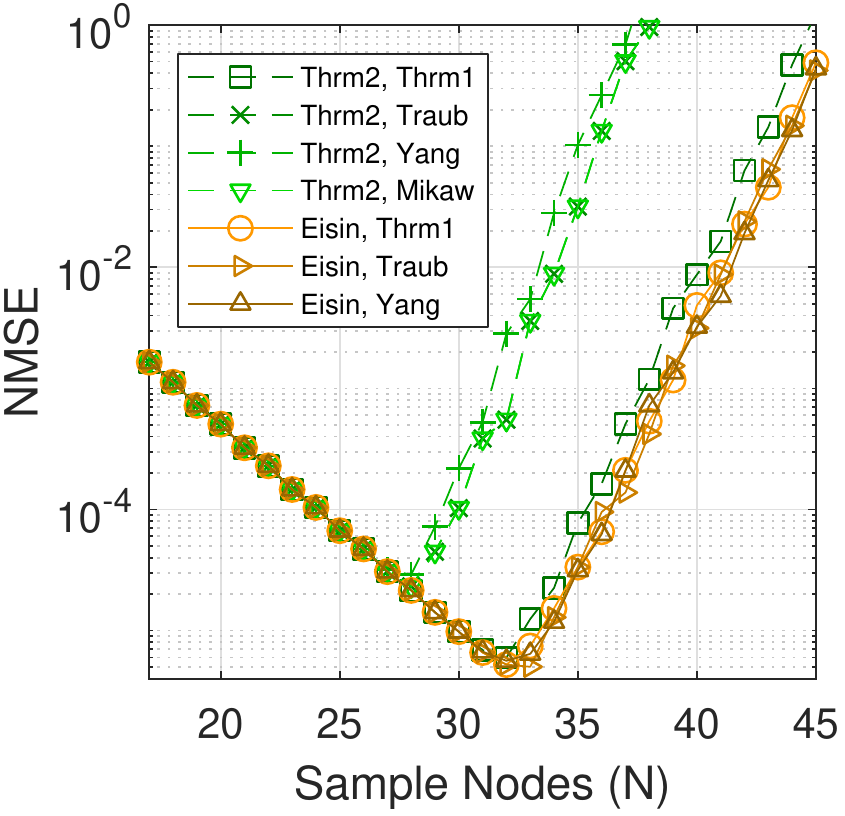}}
\subfloat[$\tanh$: Gauss-Lob]{\label{fig:NMSE_gauss_lobbatto_tanh}\includegraphics[width=0.2\linewidth]{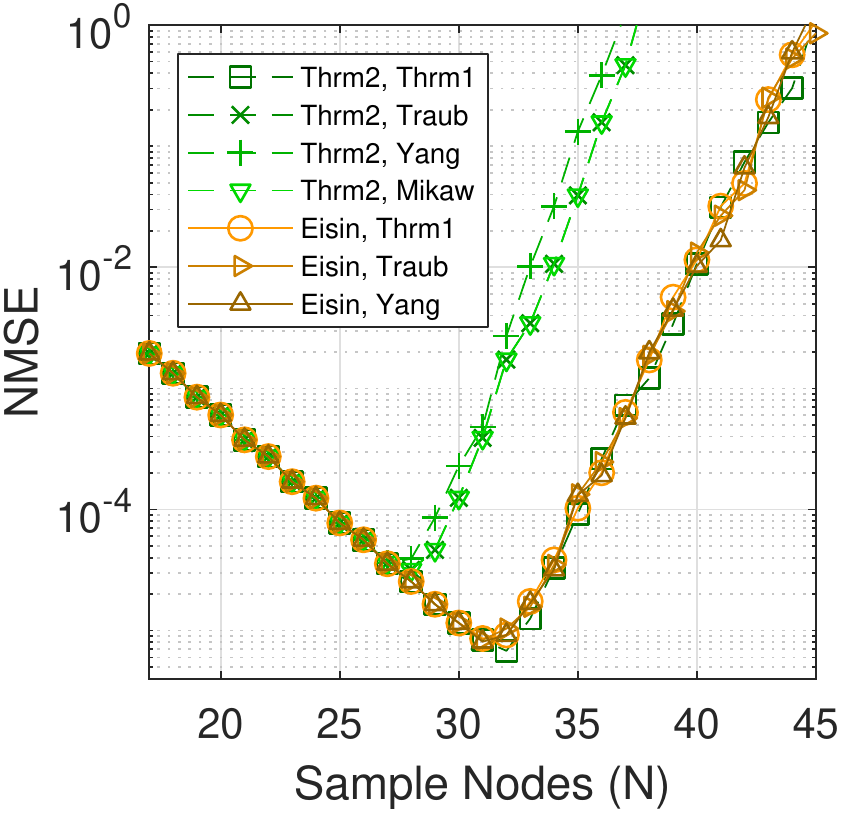}}
\subfloat[$\tanh$: $N$th-roots]{\label{fig:NMSE_fekete_tanh}\includegraphics[width=0.2\linewidth]{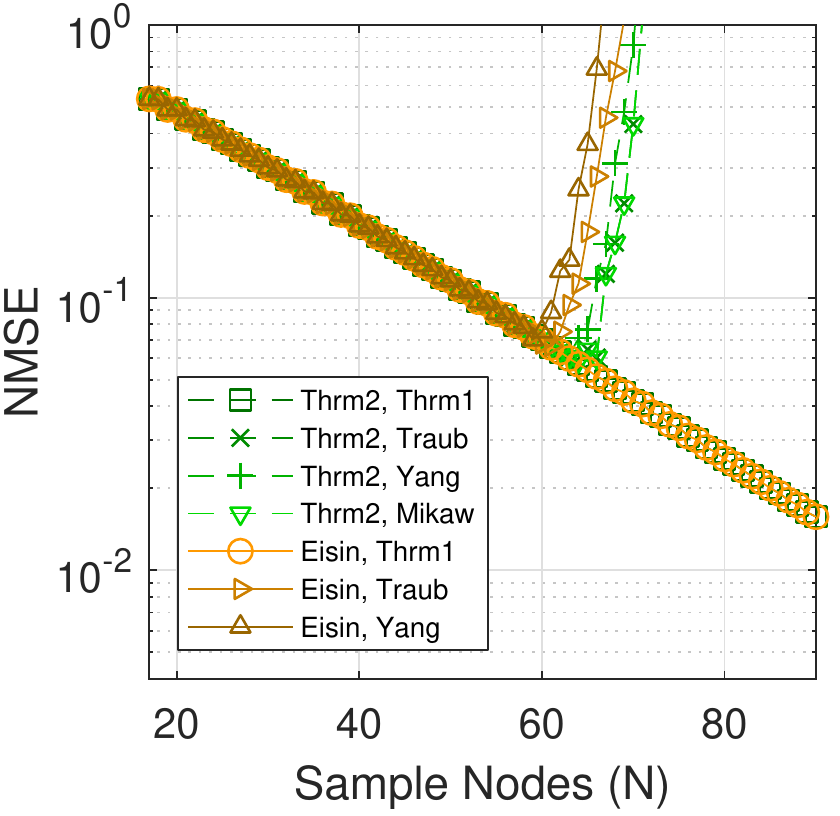}} \\
\subfloat[$\exp$: Equid]{\label{fig:NMSE_equidistant_exp}\includegraphics[width=0.2\linewidth]{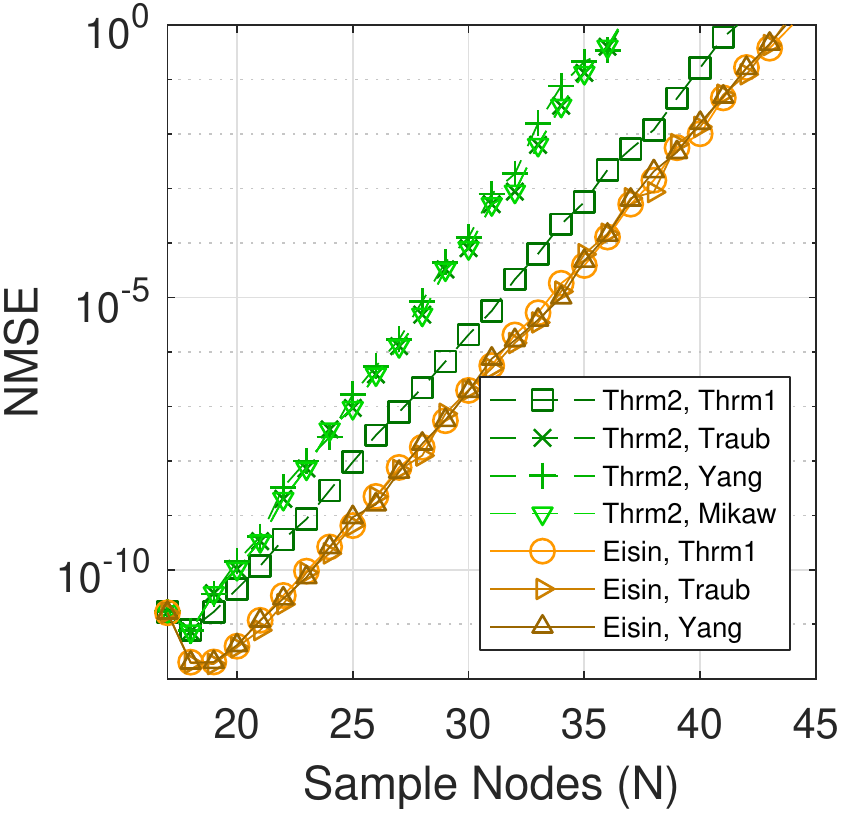}}
\subfloat[$\exp$: Cheby]{\label{fig:NMSE_chebyshev_exp}\includegraphics[width=0.2\linewidth]{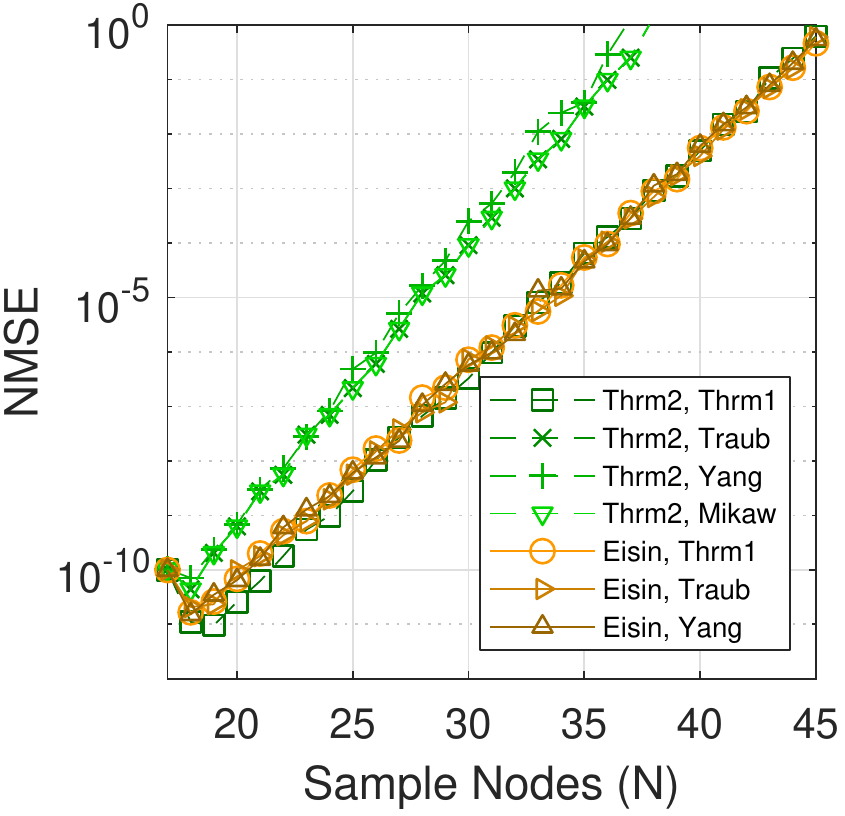}}
\subfloat[$\exp$: Ext-Cheby]{\label{fig:NMSE_extended_chebyshev_exp}\includegraphics[width=0.2\linewidth]{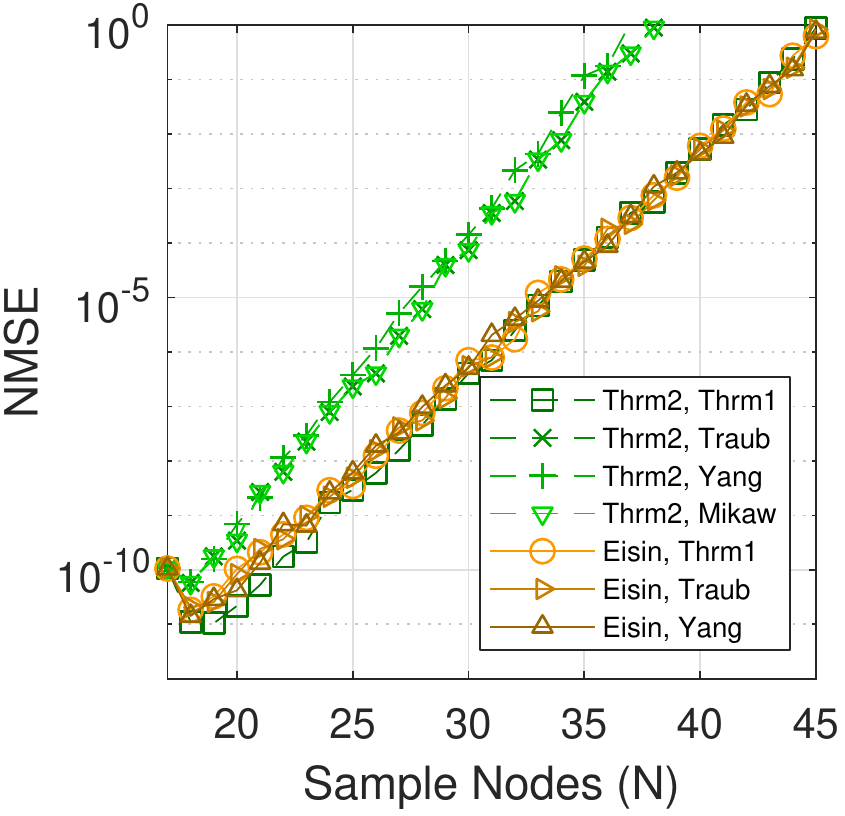}}
\subfloat[$\exp$: Gauss-Lob]{\label{fig:NMSE_gauss_lobbatto_exp}\includegraphics[width=0.2\linewidth]{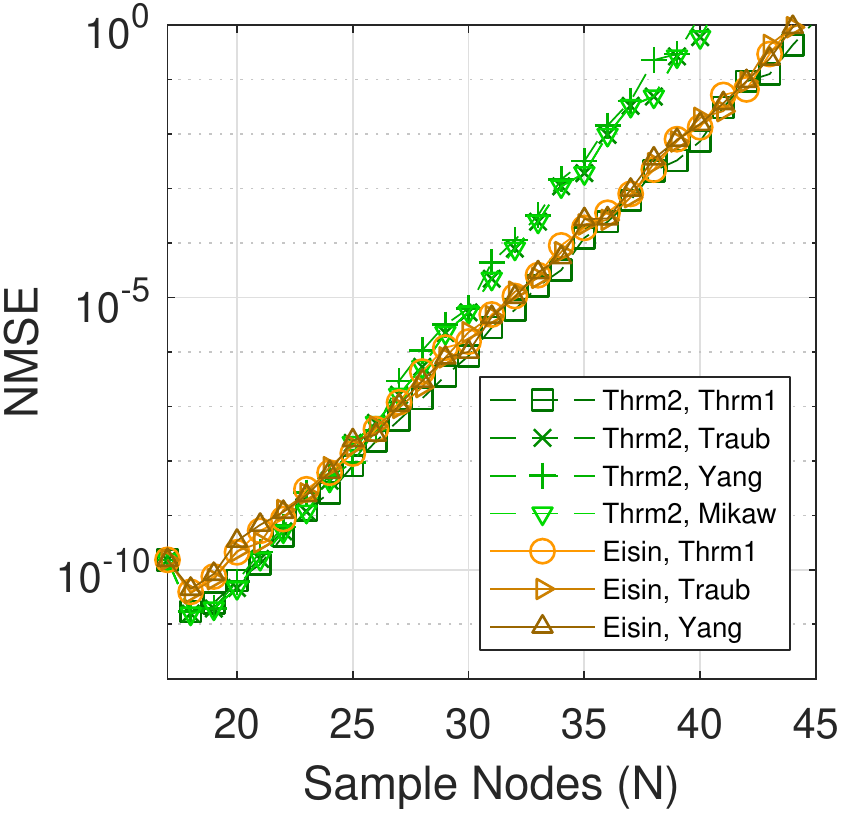}}
\subfloat[$\exp$: $N$th-roots]{\label{fig:NMSE_fekete_exp}\includegraphics[width=0.2\linewidth]{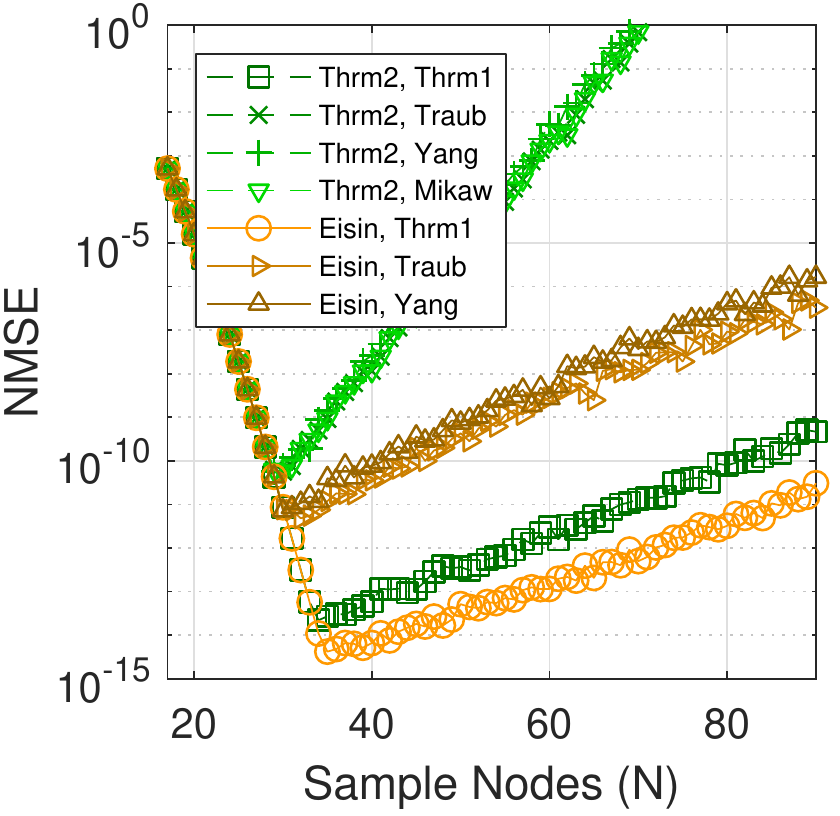}} \\
\caption{NMSE performance on interpolating three different signal samples (i.e. $\sin$, $\tanh$, and $\exp$) by means of different Vandermonde inverse calculation methods. The sampling nodes are obtained by different design i.e. equidistant, Chebyshev, extended-Chebyshev, Gauss-Lobbato, and $N$th-roots of unity on complex circle. }\label{fig:nmse_samples}
\end{figure}

Examples of interpolation results using sample measures from three analytical functions are shown in \cref{fig:plot_samples}. Notice the high deviations from true signal profile using the Eisinberg's approach in $N$th roots of unity sample nodes compared to our proposed solution.

\begin{figure}[tbhp]
\centering 
\subfloat[{\scriptsize{Equid: $1.38\mathrm{e}{-1}$/$9.1\mathrm{e}{-2}$}}]{\label{fig:plot_equidistant_sin}\includegraphics[width=0.32\linewidth]{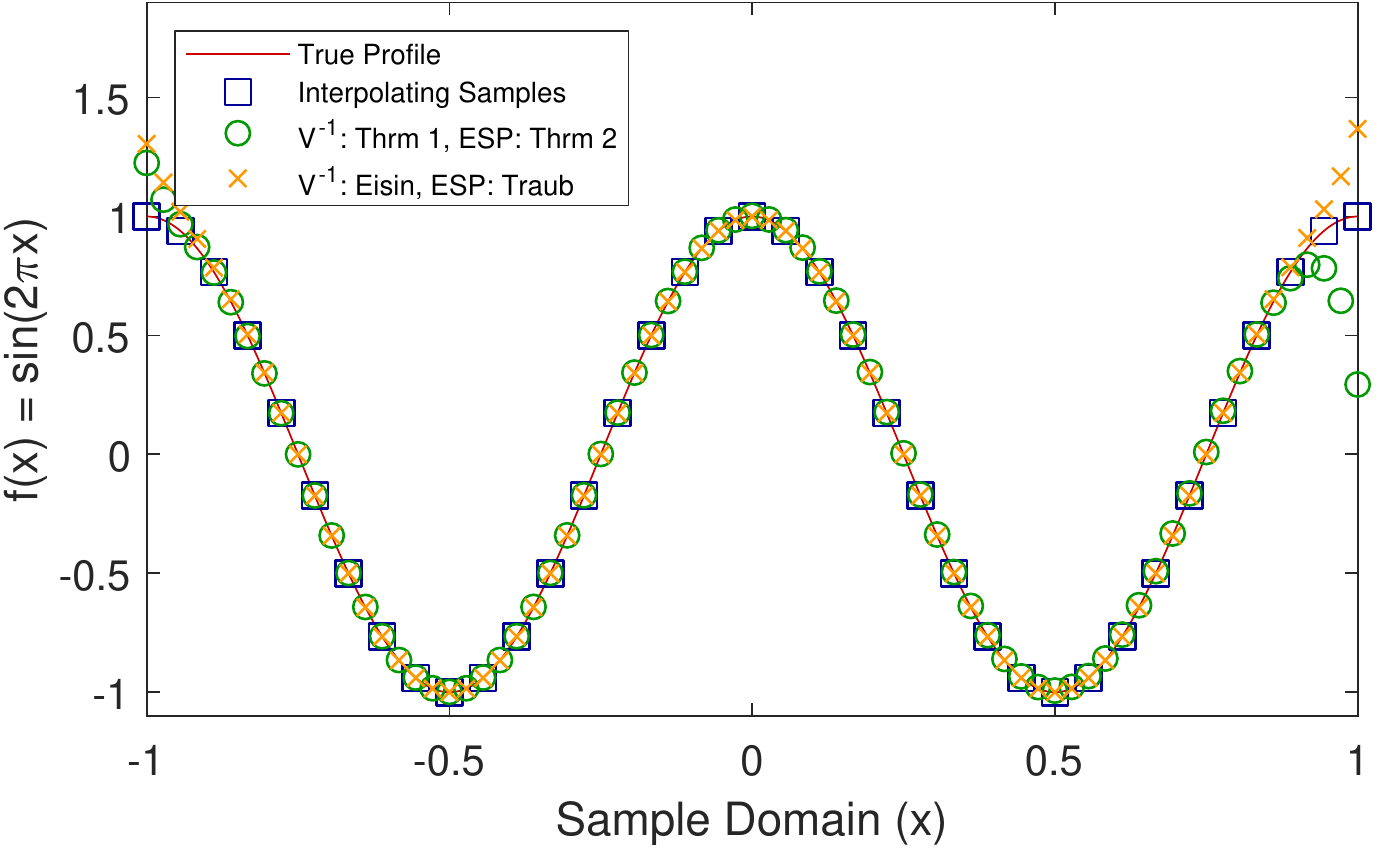}}
\subfloat[{\scriptsize{Cheby: $6.49\mathrm{e}{-4}$/$6.72\mathrm{e}{-5}$}}]{\label{fig:plot_chebyshev_sin}\includegraphics[width=0.32\linewidth]{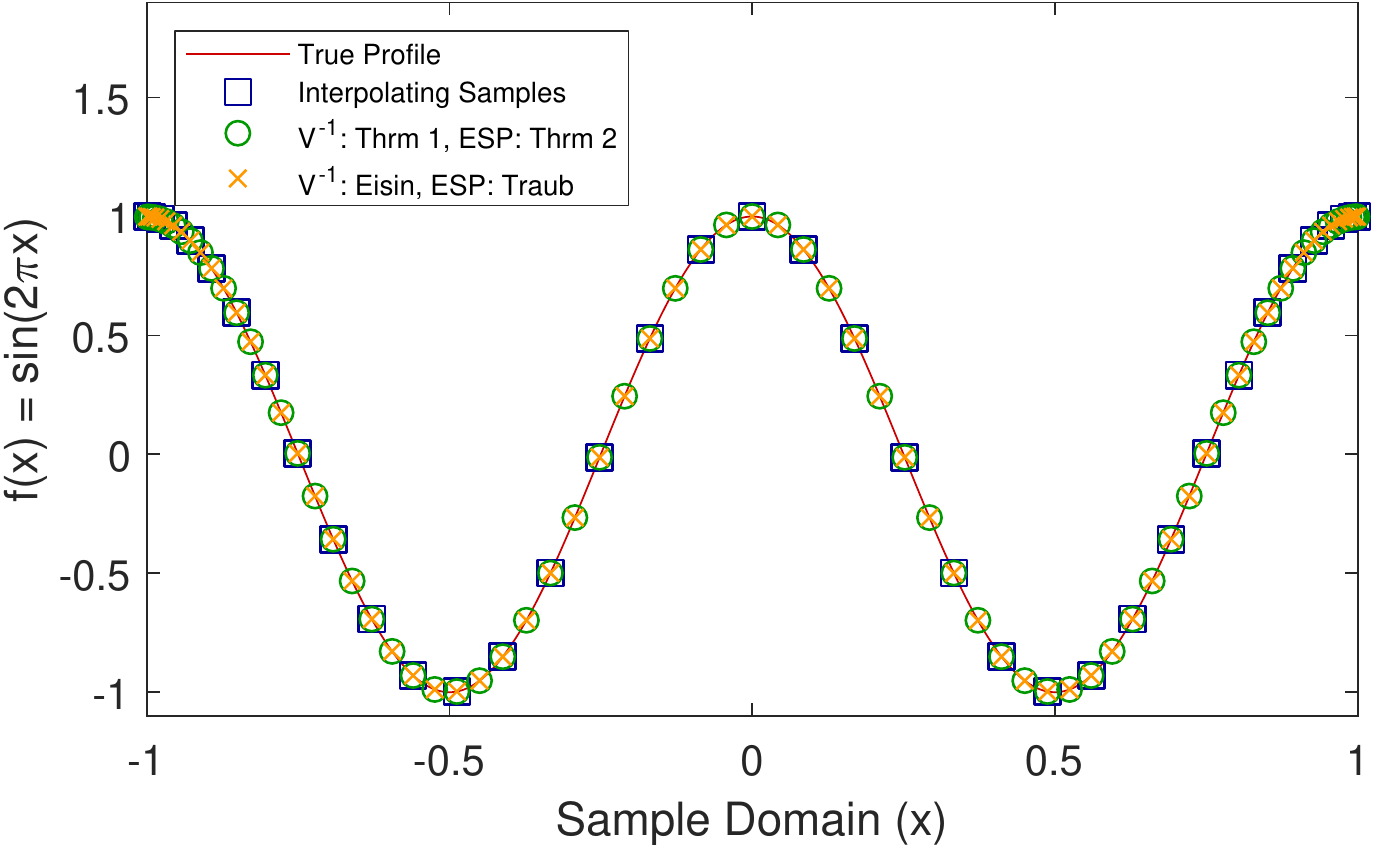}}
\subfloat[{\scriptsize{Nth-Roots: $5.74\mathrm{e}{-15}$/$5.46\mathrm{e}{-7}$}}]{\label{fig:plot_Nth_roots_sin}\includegraphics[width=0.36\linewidth]{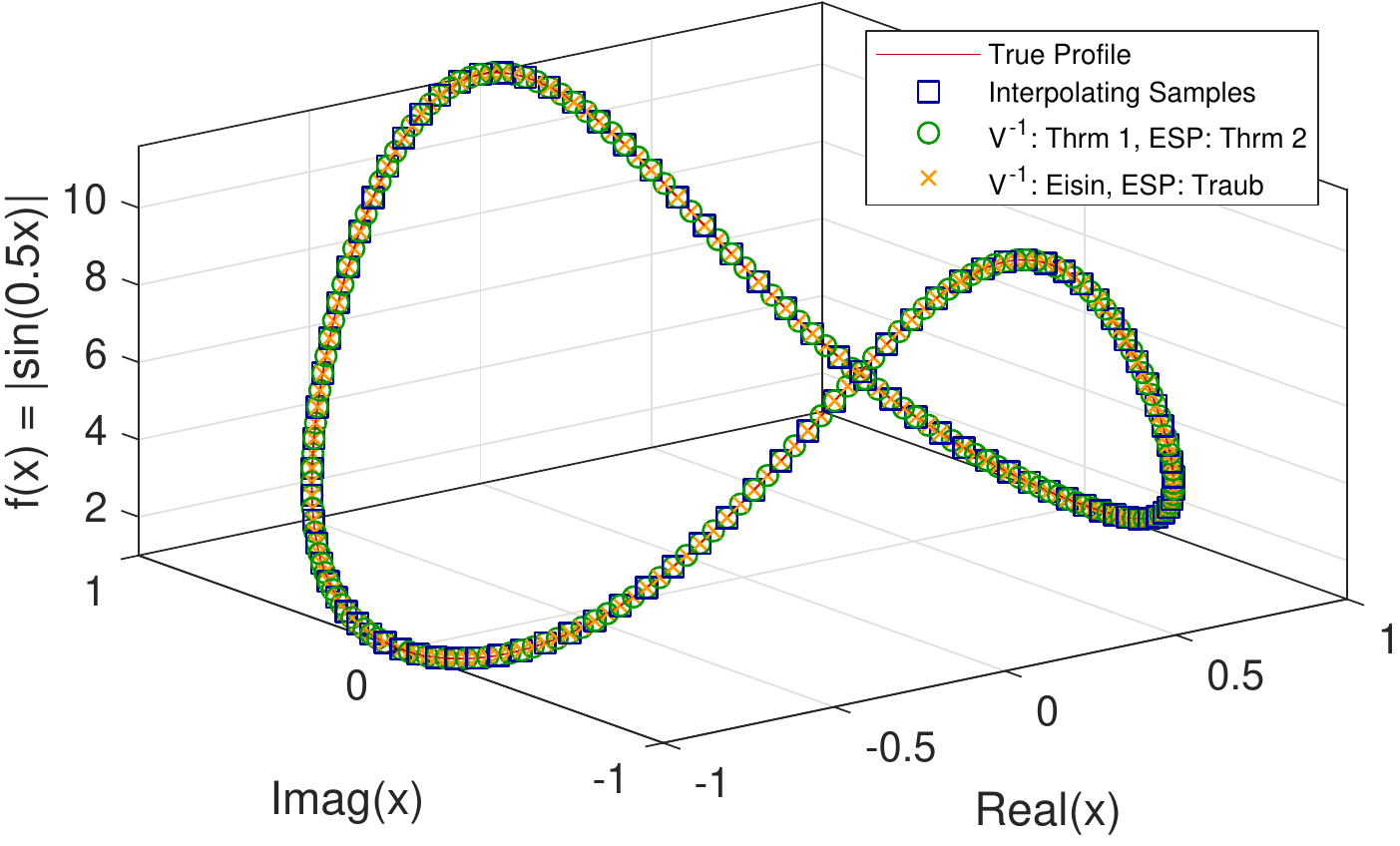}} \\
\subfloat[{\scriptsize{Equid: $5.24\mathrm{e}{-1}$/$2.83\mathrm{e}{-1}$}}]{\label{fig:plot_equidistant_tanh}\includegraphics[width=0.32\linewidth]{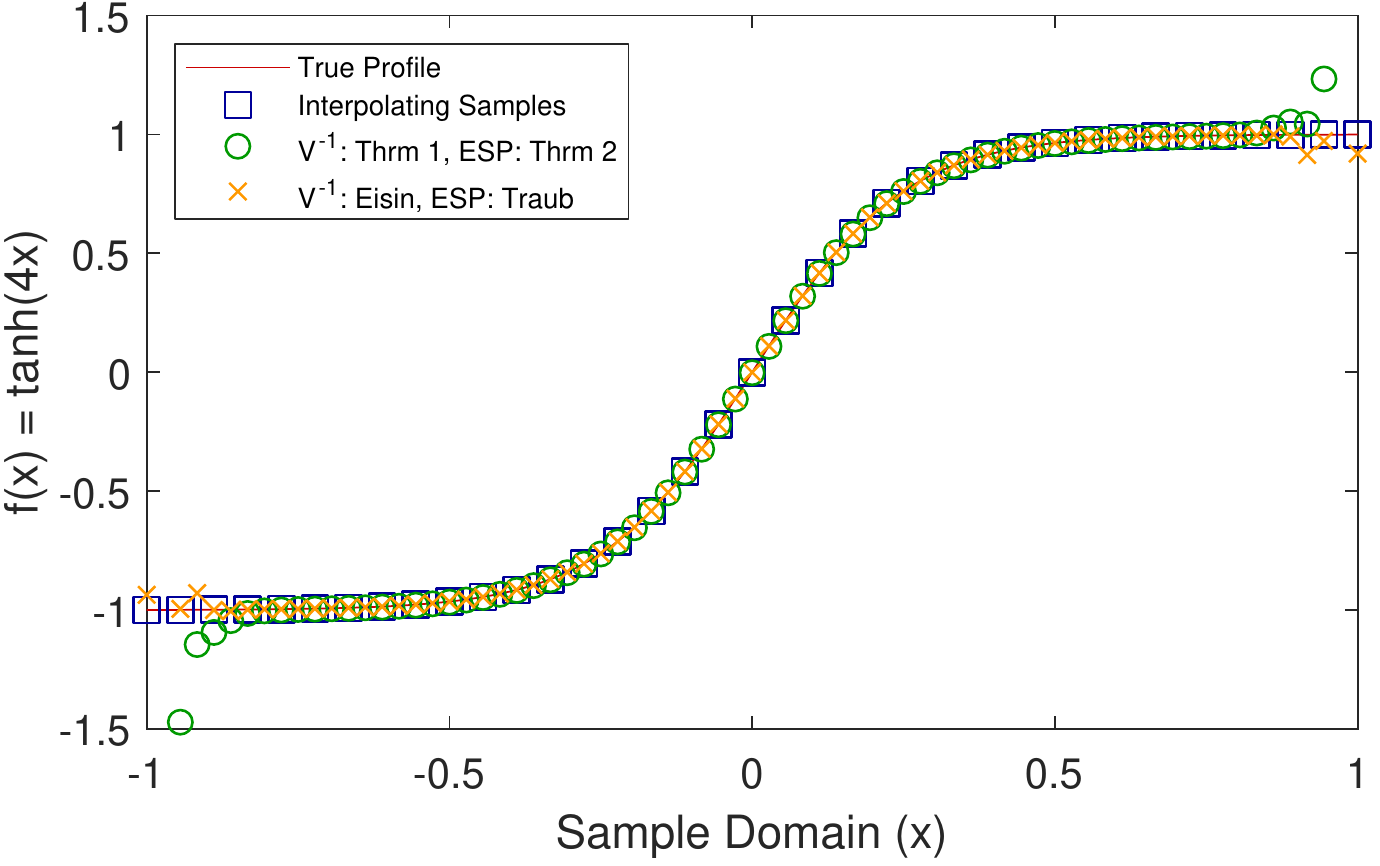}}
\subfloat[{\scriptsize{Ext-Cheby: $6.4\mathrm{e}{-4}$/$2.62\mathrm{e}{-5}$}}]{\label{fig:plot_extended_chebyshev_tanh}\includegraphics[width=0.32\linewidth]{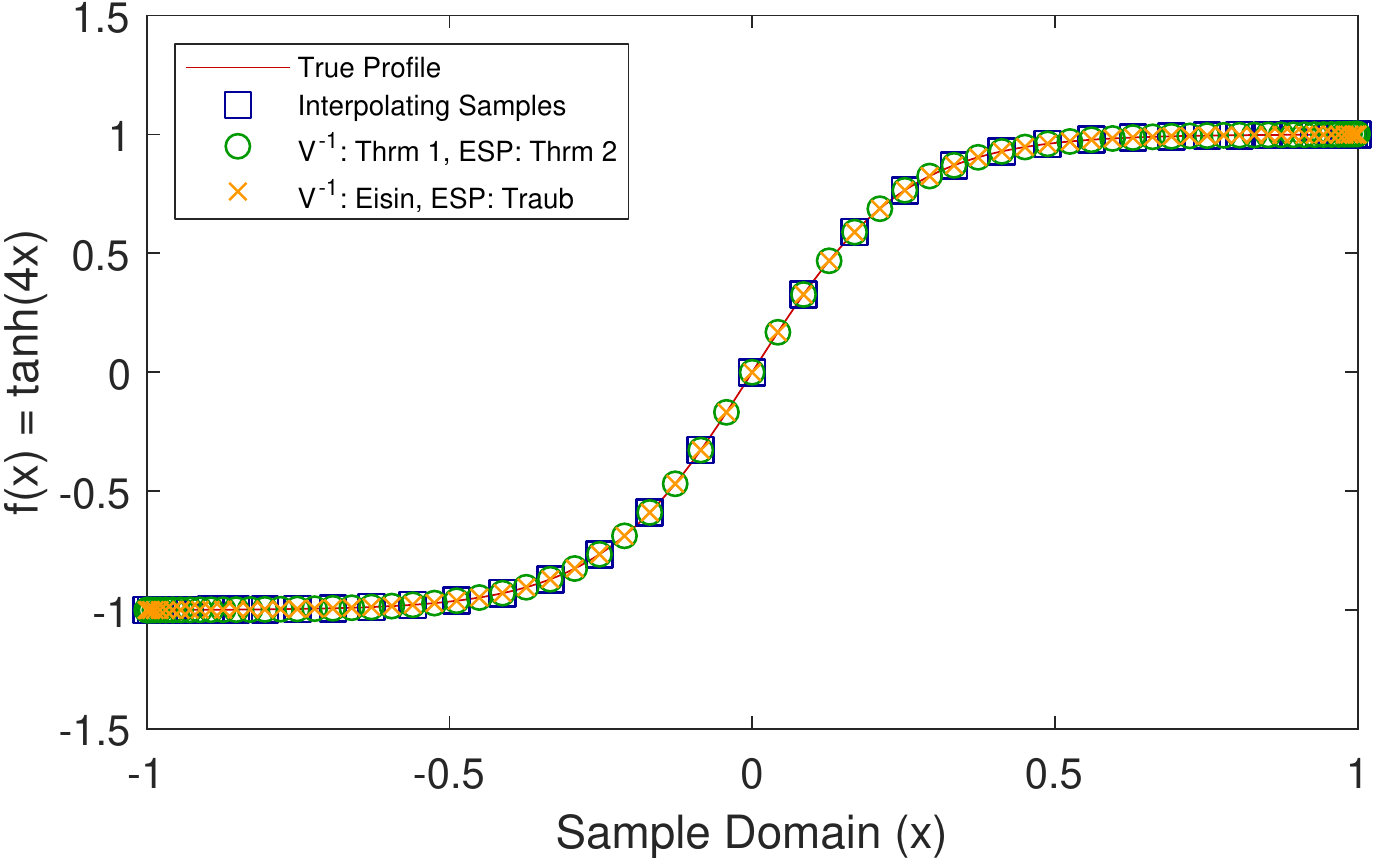}}
\subfloat[{\scriptsize{Nth-Roots: $4.09\mathrm{e}{-4}$/$4.52\mathrm{e}{-1}$}}]{\label{fig:plot_extended_Nth_roots_tanh}\includegraphics[width=0.36\linewidth]{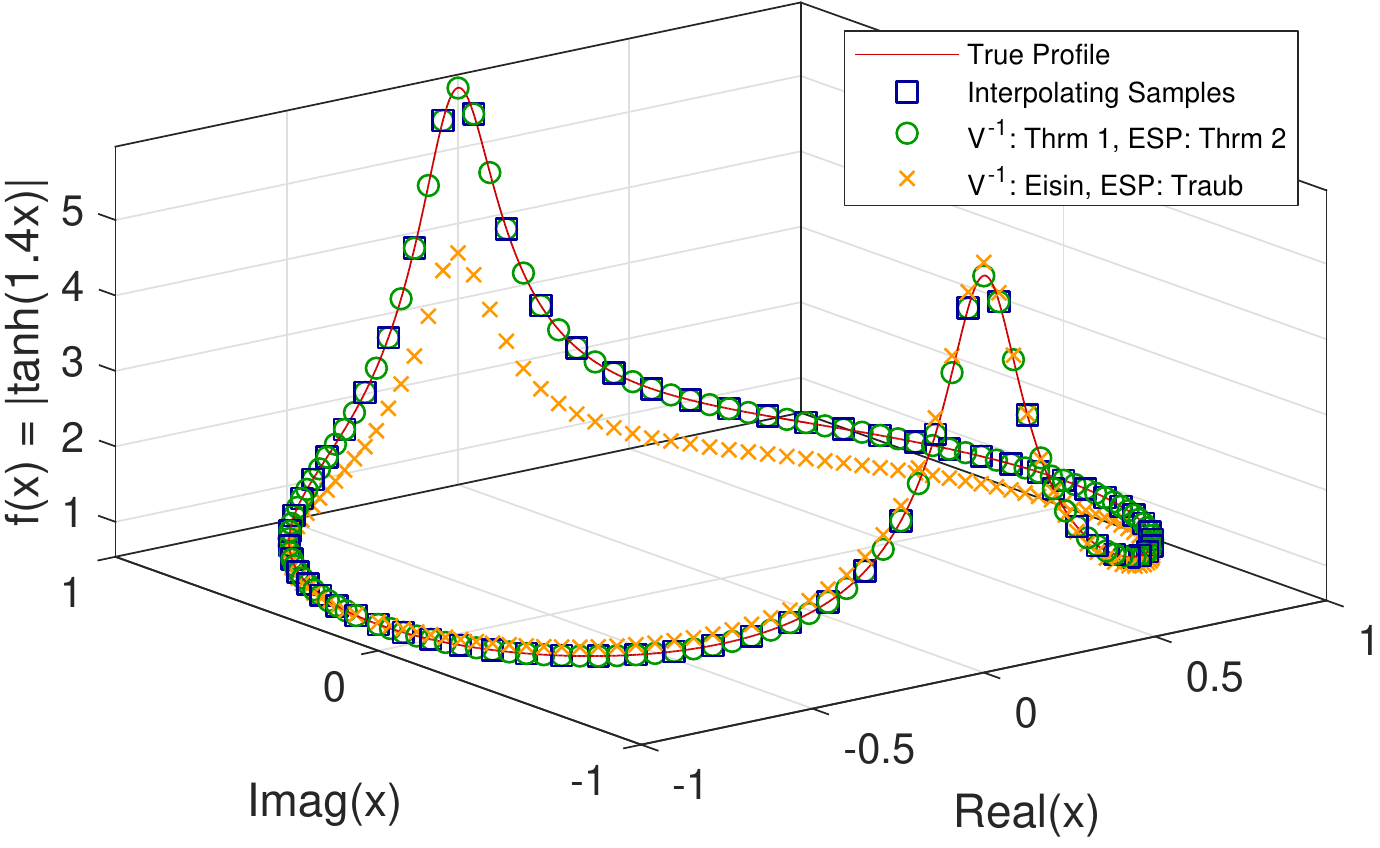}}\\
\subfloat[{\scriptsize{Equid:$3.01\mathrm{e}{-2}$/$1.11\mathrm{e}{-2}$}}]{\label{fig:plot_equidistant_exp}\includegraphics[width=0.32\linewidth]{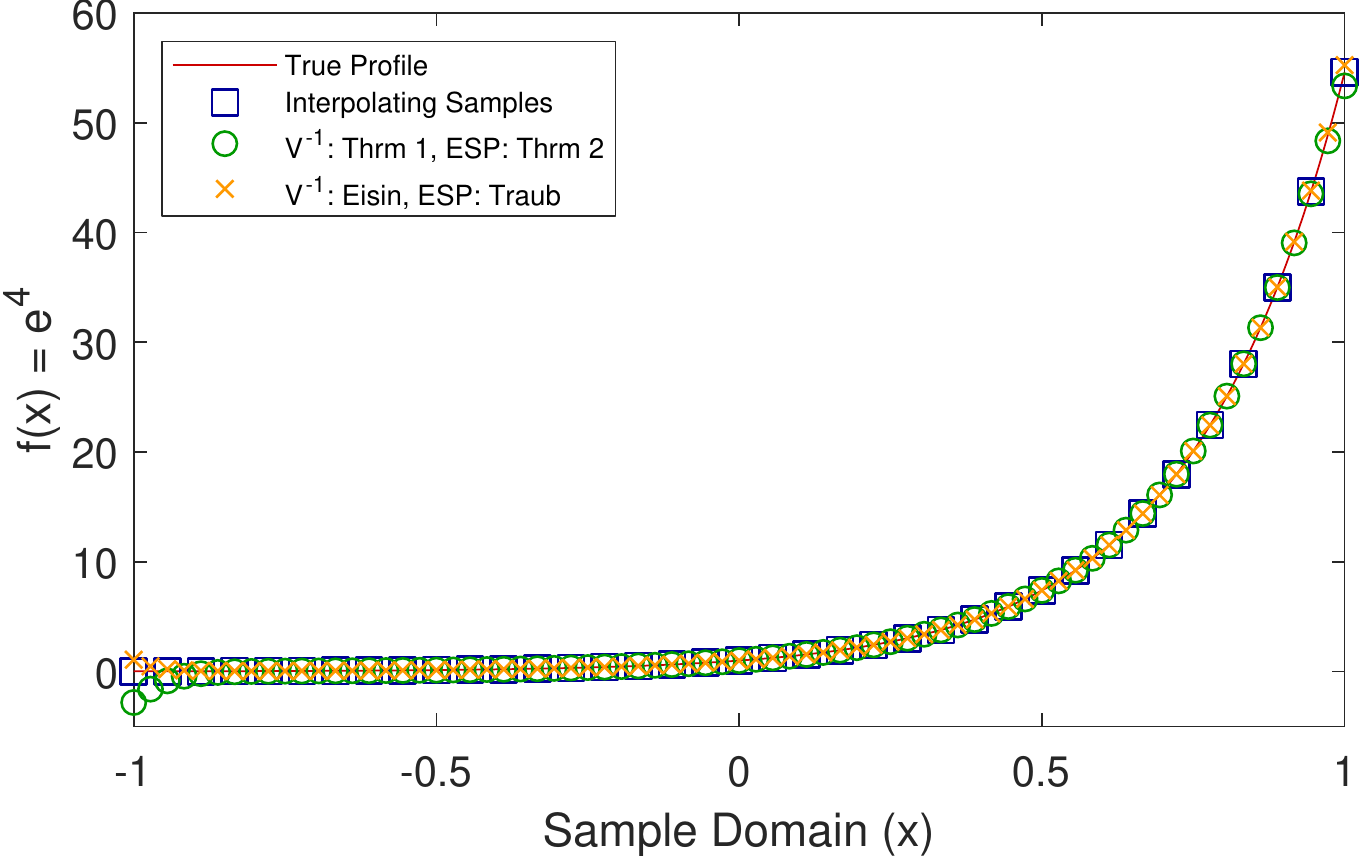}}
\subfloat[{\scriptsize{Gauss-Lob: $6.78\mathrm{e}{-4}$/$2.23\mathrm{e}{-4}$}}]{\label{fig:plot_gauss_lobbatto_exp}\includegraphics[width=0.32\linewidth]{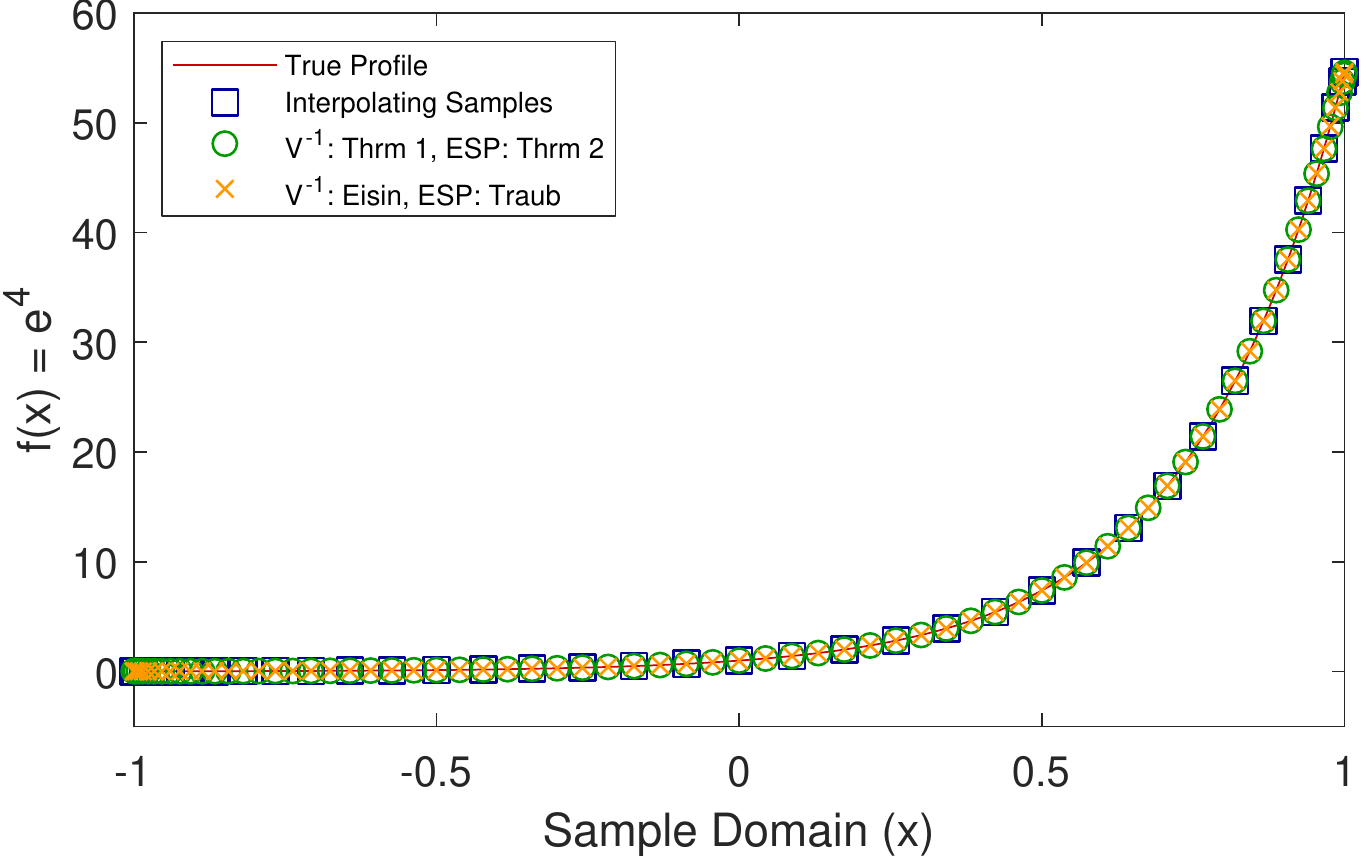}}
\subfloat[{\scriptsize{Nth-Roots: $4.13\mathrm{e}{-15}$/$1.73\mathrm{e}{-6}$}}]{\label{fig:plot_Nth_roots_exp}\includegraphics[width=0.36\linewidth]{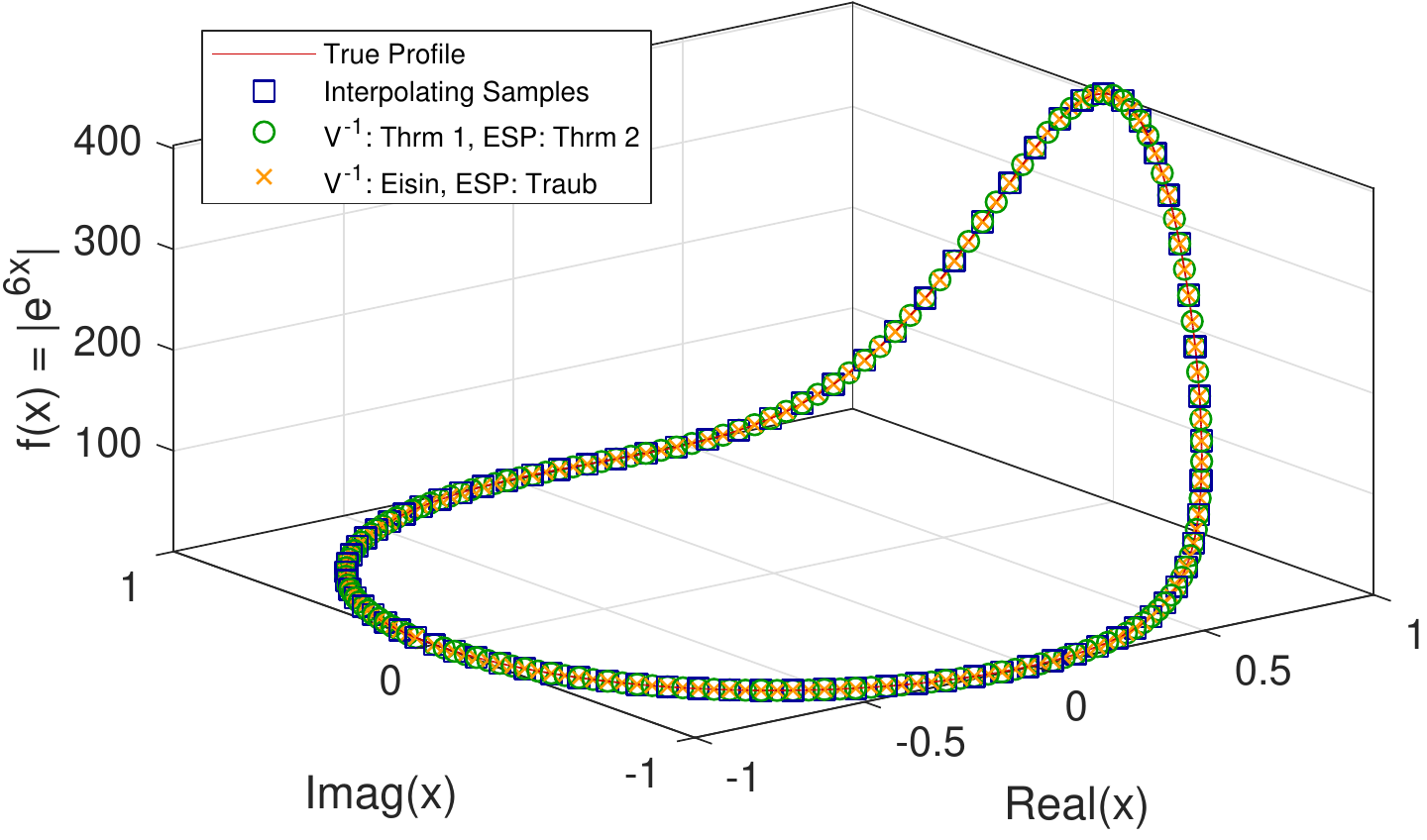}}\\
\caption{Analytical function interpolation using two different Vandermonde inversion methods, i.e. proposed vs Eisinberg \cite{EisinbergFedele2006}, using different sampling nodes. Number of samples used for equidistant, Chebyshev, extended-Chebyshev and Gauss-Lobbato nodes is $37$ for all analytical function. The numbers of samples using $Nth$-roots of unity for $\sin$, $\tanh$, and $\exp$ are $100$, $70$, and $85$, respectively. The NMSEs are stated in each caption of experiment for both proposed/Eisinberg's methods.}\label{fig:plot_samples}
\end{figure}

Overall, both of our inverse Vandermonde and ESP calculation methods, not only outperforms the existing state-of-the-art numerical solutions, but also integrating our ESP method on other inverse calculation methods might give some further boost for estimation. However, the trade off to the proposed inverse solution is the computation time as it must call the ESP solution for every row. This iterative nature magnifies the computation time of simple recursive functions. Summarized in a plot in \cref{fig:plot_Computation} the proposed inverse falls behind in computational efficiency past $30$ samples.

\begin{figure}[tbhp]
\centering 
\subfloat[Eisinberg's Method ]{\label{fig:EisinComputation}\includegraphics[width=0.3\linewidth]{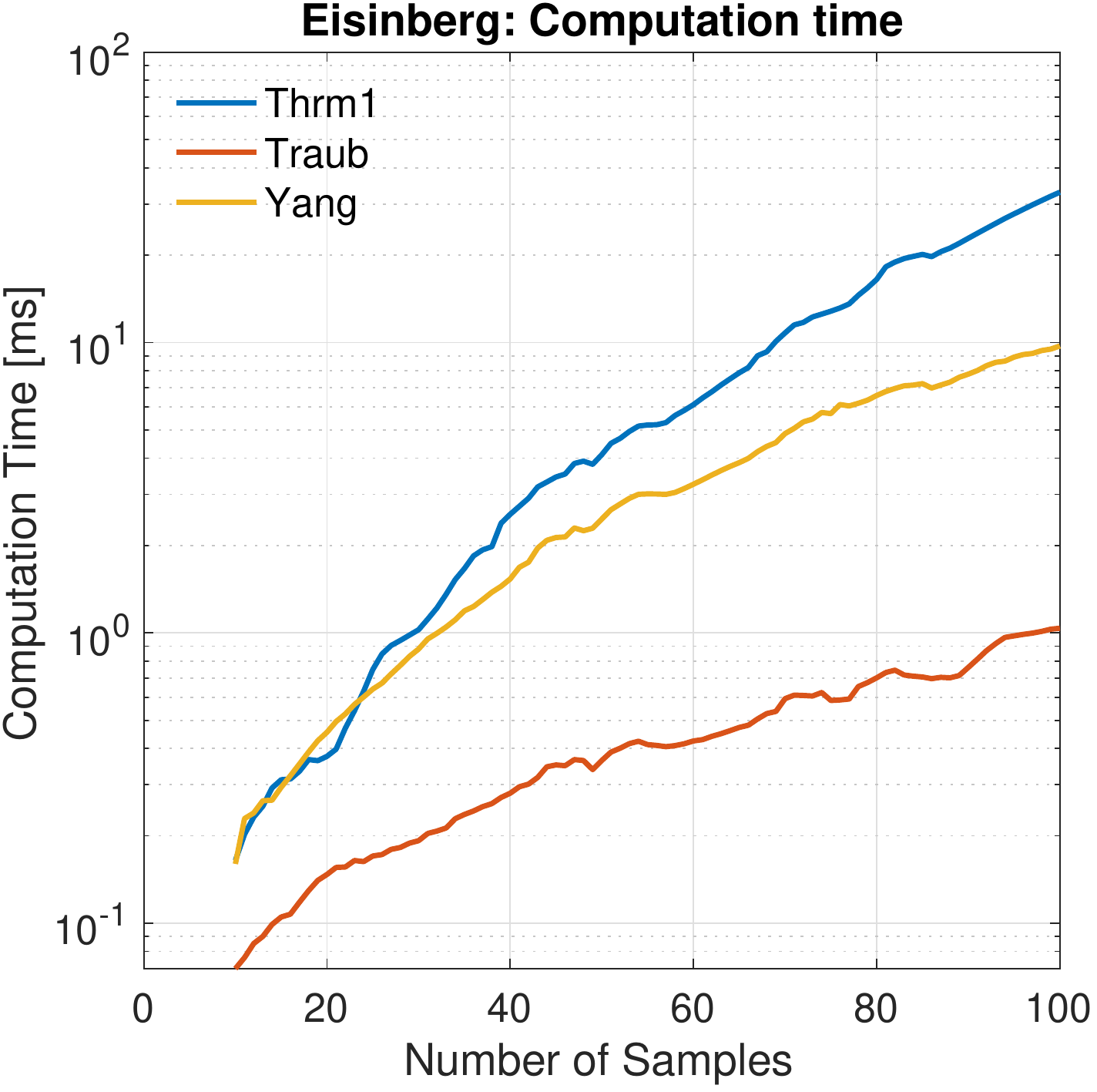}}
\subfloat[Proposed Method]{\label{fig:ProposedComputation}\includegraphics[width=0.3\linewidth]{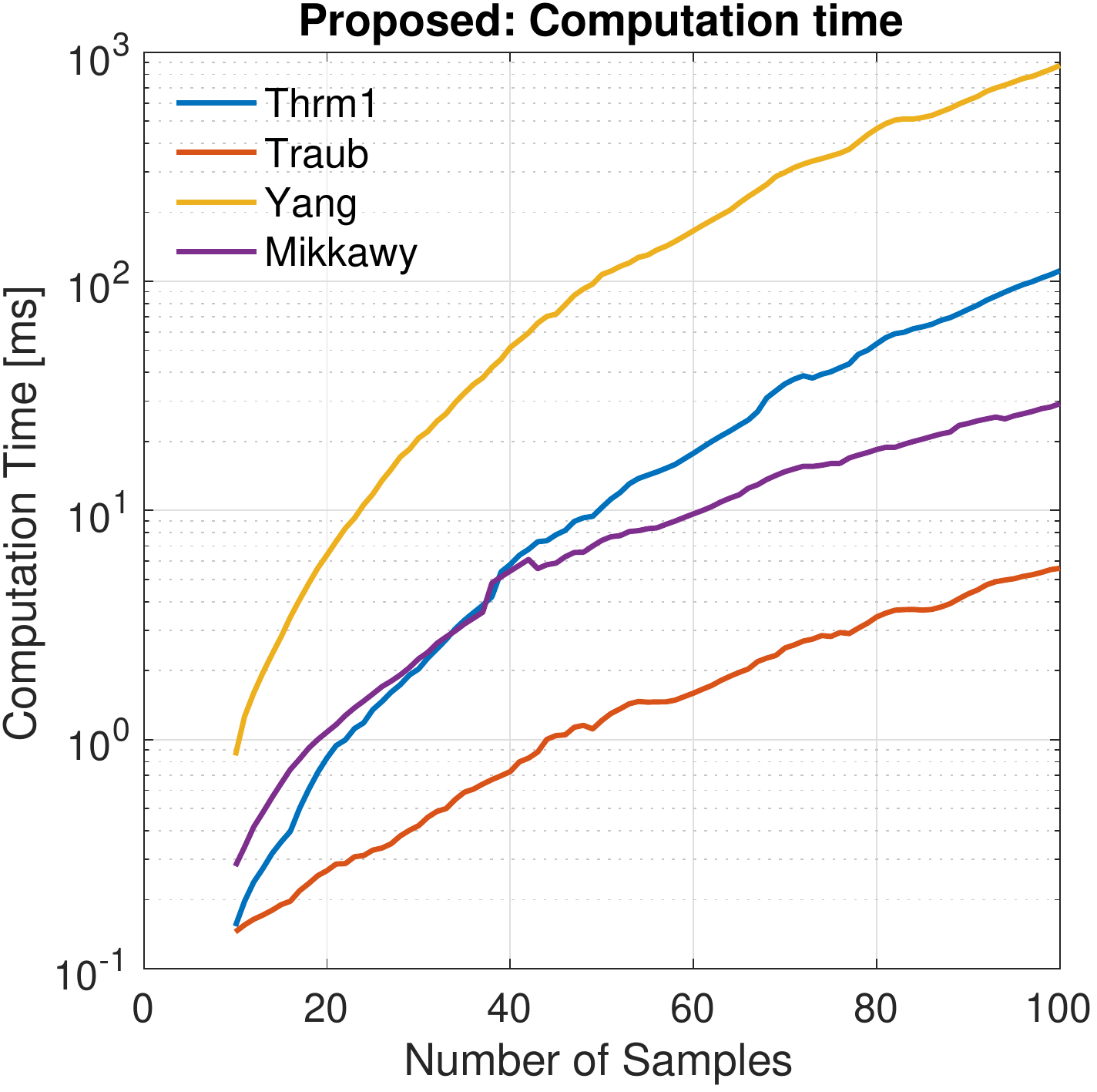}}
\caption{Computation time for inverse Vandermonde calculation using different sample nodes.}
\label{fig:plot_Computation}
\end{figure}

\section{Conclusions}\label{sec:conclusions}
While there is a growing interest in solving Vadermonde equation system in many applied science problems such as in super-resolution, spectral analysis and cryptography, the lack of well generalizable, numerically stable, and accurate inverse solution is yet to be discovered. In this paper, we have presented a framework to unify this mounting needs to solve the inverse for any arbitrary Vandermonde matrix in a close-form. In particular, we made two contributions by (a) expressing the elementary symmetric polynomial in a recursive summation that only takes $\mathcal{O}(N)$ for computational complexity; and (b) obtain a closed-form solution to the inverse Vandermonde matrix developed based on the partial-fraction technique. We showed that not only our proposed method can generalize into any arbitrary sampling nodes, but also demonstrate significant stability and accuracy on the $N$th roots of unity samples on the complex plane. Our results can be of great interest to the researchers in the field of super-resolution e.g. \cite{candes2014towards, superResolution, batenkov2019super, batenkov2019rethinking}, where the robustness and accuracy of our numerical solution can greatly impact the recovery errors. We further demonstrated the utility of proposed inverse method on one-dimensional signal interpolation under different sampling scenarios using equidistant, Chebyshev, Gauss-Lobbato, and $N$th roots of unity nodes. The proposed method is clearly applicable to fairly arbitrary sampling nodes.



\bibliographystyle{siamplain}
\bibliography{references}
\end{document}